\documentclass[a4paper,11pt]{book}
\usepackage{amsmath, amsthm, amssymb}
\usepackage{amsmath,amsbsy}
\usepackage[a4paper, top=3cm, bottom=3cm]{geometry}
\usepackage[latin1]{inputenc}
\usepackage{setspace}
\usepackage{fancyhdr}
\usepackage{tocloft}
\usepackage{graphicx}
\usepackage[USenglish]{babel}
\usepackage[T1]{fontenc}
\usepackage{lmodern}
\usepackage{emptypage}
\usepackage{psfrag}
\newcommand{\HRule}{\rule{\linewidth}{0.5mm}}

\def\eps{\epsilon}%
\def\tensor{\,\raise2pt\hbox{${}_{\otimes}$}\,}
\def\fdg{\,:\,}
\def\ptl{\partial}
\def\rest#1{\raise-2pt\hbox{${\lfloor_{#1}}$}}

\def\cal#1{\mathcal{#1}}
\def\mbo#1{\boldsymbol{#1}}

\def\ip#1#2{\langle#1,#2\rangle}

\def\hatt#1{\widehat #1{}}
\def\olin#1{\overline{#1}{}}
\def\ulin#1{\underline{#1}{}}
\def\tild#1{\widetilde{#1}{}}

\def\grad{{\nabla}}
\newcommand{\leftexp}[2]{{\vphantom{#2}}^{#1}{#2}}

\def\halb{\frac{1}{2}}
\newtheorem{thm}{Theorem}[section]
\newtheorem{cor}[thm]{Corollary}
\newtheorem{lem}[thm]{Lemma}
\newtheorem{pro}[thm]{Proposition}
\newtheorem{rem}[thm]{Remark}
\date{}

\begin{document}
\pagestyle{headings}


\begin{titlepage}
\begin{center}


\textsc{\LARGE Fachbereich Mathematik und Informatik \\[0.75cm] Freie Universit\"at Berlin}\\[3cm]

\textsc{\Large Dissertation}\\[0.5cm]

\HRule \\[0.4cm]
{ \huge \bfseries On the Cauchy Problem for Energy Critical Self-Gravitating Wave Maps}\\[0.4cm]

\HRule \\[0.9cm]

\textsc{\Large Nishanth Abu Gudapati}

\vfill
{\large October 2013\\Angefertigt am Max-Planck-Institut f\"ur Gravitationsphysik \\ (Albert-Einstein-Institut)}

\end{center}
\end{titlepage}
\cleardoublepage
\thispagestyle{empty}
\begin{flushleft}
\vspace*{15cm}
\noindent Vorgelegt am
16. October 2013,\\
\vspace*{0.75cm}
Betreuung \\
Prof. Dr. Gerhard Huisken,\\
Mathematisches Institut,\\
Universit\"at T\"ubingen,\\
Freie Universit\"at Berlin \\
\vspace*{0.25cm}
Prof. Dr. Lars Andersson,\\
Max-Planck-Institut F\"ur Gravitationsphysik,\\
(Albert-Einstein-Institut) \\

\end{flushleft}

\cleardoublepage

{\thispagestyle{empty}
\begin{flushleft}
\noindent Vorgelegt von:\\
Nishanth Abu Gudapati,\\
geboren am 13. November 1985\\ in Hyderabad\\[0.3cm]

\noindent Anschrift:\\
Am M\"uhlenberg 1,\\
14476 Potsdam-Golm\\[3cm]

\noindent {\bf  Erkl\"arung zur Doktorarbeit}\\[1cm]

\noindent Hiermit versichere ich, dass die von mir vorgelegte 
Doktorarbeit mit dem Titel
\end{flushleft}
\begin{center} ``On the Cauchy Problem for Energy Critical Self-Gravitating Wave Maps" \end{center}
\begin{itemize}
\item von mir selbstst\"andig verfasst und keine anderen als die von mir 
angegebenen Quellen und Hilfsmittel benutzt wurden
\item noch nicht anderweitig als Doktorarbeit eingereicht wurde.
\end{itemize}
\vspace{5cm}
\begin{flushleft}
\noindent Berlin, 16.10.2013,
\\[2cm]

\noindent --------------------------------------------------------

\noindent Nishanth Abu Gudapati
\end{flushleft}}

\cleardoublepage

\pagestyle{plain}
\lhead[]{\thepage}
\rhead[\thepage]{}
\pagenumbering{arabic}



\chapter* {Zusammenfassung}

Diese Arbeit handelt von dem Cauchy Problem für Wave--Maps, welche mit den Einstein--Gleichungen der allgemeinen Relativitätstheorie gekoppelt sind.
Wave--Maps sind Abbildungen von einer Lorentz'schen Mannigfaltigkeit auf eine Riemann'sche Mannigfaltigkeit welche kritische Punkte eines Wave--Map Lagrangian sind.
Selbst--gravitative Wave--Maps bilden von einer asymptotosch flachen Lorentz'schen Mannigfaltigkeit ab, welche die Einstein'schen Gleichungen erfüllen, die die Wave--Map als Quelle besitzen.
Die Energie des Wave--Map Lagrangian ist invariant unter Skalierung in 2+1 Dimensionen.
Abgesehen von dem rein geometrischen Interesse ist die Motivation für das Studium von kritischen selbst--gravitativen Wave--Maps, dass die 3+1 Vakuum Einstein Gleichungen auf dem Prinzipalbündel mit eindimensionaler Lie Gruppe auf das Einstein Wave--Map System in 2+1 Dimensionen reduziert werden kann.
Das Ziel dieser Arbeit ist es, ein Programm zur Untersuchung von globaler Regularität von kritischen selbst--gravitativen Wave--Maps ins Leben zu rufen um die globale Regularität der 3+1 Einstein Vakuum Gleichungen zu verstehen.
Die gegenwärtige Herangehensweise hat den Vorteil, dass man in der kritischen Dimension für Wave--Maps arbeitet.
In Laufe der letzten zwanzig Jahre wurde eine Reihe von Techniken entwickelt, um die Frage der globalen Regularität von kritischen Wave--Maps auf dem Minkowski Hintergrund zu klären.
Jeder Vortschritt auf dem Gebiet der globalen Regularität von kritischen selbst--gravitativen Wave--Maps sollte nicht nur diese Methoden im Blick haben, sondern auch neue Ideen und Techniken zur Überwindung von Hindernissen durch die sich entwickelnde Geometrie des Systems einführen.
Diese Arbeit ist ein kleiner Schritt in diese Richtung.

Das wesentliche Resultat dieser Arbeit ist der Beweis, dass die Energie der Einstein--Äquivarianten Wave--Map Systeme sich bei der Cauchy Evolution nicht konzentrierert.
Ein Hauptbestandteil des Beweises ist die Ausnutzung der Tatsache, dass die geometrische Masse im \textit{Unendlichen} des Einstein--Äquivarianten Wave--Map Systems während der Evolution erhalten bleibt.
Diese Beobachtung hat dennoch ein paar subtile lokale Auswirkungen welche benutzt werden um die Energie \textit{lokal} abzuschätzen.
Zum Beispiel konstruieren wir ein Divergenz--freies Vektorfeld, welches Monotonie der Energie auf dem Rückwärts Nullkegel in jedem Punkt gibt.
Außerdem wurde dieser Vektor benutzt um zu Zeigen, dass die Energie sich nicht entfernt von der Achse der Domain--Manigfaltigkeit konzentriert.
Später, wenn die Divergenz des Morawetz Vektors auf dem gestutzten Rückwärts Nullkegel genähert wird, zeigen wir, dass die kinetische Energie sich nicht konzentriert.
Letztendlich, annehmend, dass die Ziel--Mannigfaltigkeit die Grillakis Bedingung erfüllt, fahren wir mit dem Beweis der nicht--Konzentration von Energie für das kritische Einstein--Äquivariante Wave--Map System fort.



\chapter* {Preface}

This work is on the Cauchy problem for wave maps coupled to Einstein's equations of general relativity. 
Wave maps are maps from a Lorentzian manifold to a Riemannian manifold which are critical points of the wave map Lagrangian. 
Self-gravitating wave maps are those from an asymptotically flat Lorentzian manifold which satisfies Einstein's equations 
with the wave map itself as the source field. The energy of the wave map Lagrangian is invariant under scaling in 2+1 dimensions.
Apart from a purely geometrical interest, the motivation for studying critical self-gravitating wave maps is that 3+1 Einstein
vacuum equations on principal bundles with one dimensional Lie group can be reduced to Einstein wave map system in 2+1 dimensions.
The intention of this work is to initiate a program of studying global regularity of critical self-gravitating
wave maps to understand the global regularity of 3+1 Einstein vacuum equations.  In this approach, the advantage is that one
is working in the critical dimension for wave maps. During the last twenty years a rich variety of techniques have been 
developed to address the question of global regularity of critical wave maps on the Minkowski background. Any progress in addressing
the global regularity of critical self-gravitating wave maps should be made by not only keeping these methods in view, but also
by introducing new ideas and techniques to overcome the obstacles caused by the evolving geometry of the system. This work is 
a small step in that direction.

The main result of this work is the proof that the energy of the Einstein-equivariant wave map system does not concentrate during 
the Cauchy evolution. A key ingredient in the proof is the use of the fact that geometric mass at infinity of the Einstein-equivariant 
wave map system is conserved during the evolution. However, this observation has some subtle local implications which have been
used to estimate the energy locally. For instance, we construct a divergence-free vector field which gives monotonicity
of energy in the past null cone of any point. In addition, this vector has also been used to prove that the energy does not 
concentrate away from the axis of the domain manifold. Later, estimating the divergence of a Morawetz vector on a truncated past null cone,
we prove that the kinetic energy does not concentrate. Finally, assuming that the target manifold satisfies the Grillakis condition, 
we proceed to prove the non-concentration of energy for the critical Einstein-equivariant wave map system.



\cleardoublepage
\section*{Conventions}
The letter $c$ is used to denote a generic positive constant which depends on initial energy or 
the universal constants such as the gravitational coupling constant of Einstein's equations. 
We may use it repeatedly in the same estimate to avoid 
cluttering up the notation. Subscripts of scalar functions denote partial differentiation and $\grad$ denotes
covariant differentiation, likewise double subscripts denote second order partial differentiation.
From Chapter 3 onwards, coordinate null triad vectors and their duals are denoted by the
letters $\ell$, $n$ and $m$. The sign convention $({-}{+}{+})$ is used for Lorentzian manifolds and Einstein's summation 
convention is used throughout. 

\section*{Acknowledgements}
I am most grateful to Lars Andersson for numerous insightful discussions and suggestions throughout the 
course of my PhD, which is a reflection of his deep and broad knowledge of the field. I am also grateful
to Gerhard Huisken for many valuable discussions and mentorship 
throughout the course of my stay at the AEI. During the preparation of this work, I had the pleasure 
of talking to a lot of people. Arne G\"odeke, Johannes Mosig, Melanie Rupflin, Alan Rendall, Pieter Blue, 
Thomas B\"ackdahl, Piotr Bizon,
J\'{e}r\'{e}mie Joudioux, Carla Cederbaum and Hermann Nicolai are a few among them. I take this opportunity to express my gratitude
to each of them.\\ \\
Special thanks are due to Vincent Moncrief for many interesting discussions, general advice and support. \\ \\
I am deeply indebted to my parents Gudapati Venkata Krishna and Gudapati Santa Kumari, my sister Spruha and
members of my extended family for their
positivity and unwavering belief in me during the difficult and low-spirited times of my PhD studies, even when
I had fragile confidence in myself.
There are not enough words to express how I am fortunate that I belong to a family with them.\\ \\
Finally, I am also thankful to the Max-Planck-Institut f\"ur Gravitationsphysik (AEI) for the IMPRS Scholarship
and hospitality.
\cleardoublepage
\newpage
\tableofcontents
\chapter{Introduction}

\section{Preliminaries and Definitions}
Let $(M,g)$ be a smooth, orientable, globally hyperbolic $(m+1)$ dimensional Lorentzian manifold and $(N,h)$ an $n$-dimensional smooth, complete, connected Riemannian manifold.
A smooth map $U : M \to N$ is called a wave map if it is a critical point of the action \footnote{this is identical to the action of harmonic 
maps except that the 
manifold $M$ is Lorentzian and consequently it is not nonnegative as opposed to the Dirichlet energy of harmonic maps } 
\[ S_{\text{WM}} (U) \fdg= \halb \int_M \text{Tr}_g(U^* h) \, \bar{\mu}_g  \]
where $U^*h$ is the pull-back of the metric $h$ by $U$, $\text{Tr}_g$ the trace 
with respect to the metric $g$ and $\bar{\mu}_g$ the spacetime volume form of $M$.
In local coordinates $\{ x^\mu \},\,\mu = 0,1, \cdots,m$ on $M$ and $\{y^j\},j =1,\cdots,n$ on $N$, the Lagrangian 

\begin{align*}
\mathcal{L}_{\text{WM}} (U) \fdg =& \halb \text{Tr}_g(U^* h)  \\
\equiv & \halb  g^{\mu \nu} h_{ij}(U) \ptl_\mu U^i \ptl_\nu U^j  
\equiv \halb \langle U^\sigma, U_\sigma \rangle_{h(U)}
\end{align*}
where $ \langle \cdot\,,\, \cdot\rangle_h$ is the first fundamental form of the target manifold $N$ and $\sigma = 0,1, \cdots,m$.
Therefore, in local coordinates

\begin{align} \label{wmlag}
 S_{\text{WM}} (U) \equiv \halb \int_M \langle U^\sigma, U_\sigma \rangle_{h(U)} \,\bar{\mu}_g . 
\end{align}
After performing the first variation with respect to $U$,
the Euler-Lagrange equations in local coordinates take the following form 
\begin{align} \label{wmel_intro}
 \square_g U^i + \leftexp{(h)}{\Gamma}^i_{jk}(U) g^{\mu \nu}\ptl_\mu U^j \ptl_\nu U^k &= 0,
\end{align}
where $\leftexp{(h)}{\Gamma}$'s are the Christoffel symbols of the target $N$
\[ \leftexp{(h)}{\Gamma}^i_{jk} \fdg = \halb h^{il}\left( \ptl_k h_{lj} + \ptl_i h_{lk} - \ptl_l h_{ij}\right),\]
for $i,j,k,l = 1,2, \cdots, n$
and $\square_g \fdg = \grad_\nu \grad^\nu$, $\grad$ is the covariant derivative corresponding to the 
Levi-Civita connection defined on $(M,g)$.\\
Alternatively, one can formulate the Euler-Lagrange equations of \eqref{wmlag} extrinsically. Assume that the target manifold
$(N,h)$ is isometrically embedded into a Euclidean space $R^{n+1}$,
then Euler-Lagrange equations of the wave map action \eqref{wmlag} must satisfy
\begin{align} \label{wmel_ext1}
 \square_g\,U (P) \perp T_P\,N
\end{align}
for any point $P$ on $N$.
Then \eqref{wmel_ext1} is equivalent to 
\begin{align}\label{wmel_ext}
 \square_g \, U = \mathbf{Q}(U)(\grad U, \grad U)
\end{align}
where $\mathbf{Q}$ is the second fundamental form of $N \hookrightarrow \mathbb{R}^{n+1}.$  \\
The canonical energy-momentum tensor $\mathbf{S}$ of the wave map Lagrangian \eqref{wmlag} is
\begin{align}
 \mathbf{S}^\nu{\,_\mu} \fdg = \frac{\ptl \mathcal{L}_{\text{WM}}}{\ptl(\ptl_\nu U)} \,\ptl_\mu U - \mathcal{L}_{\text{WM}}\, \delta^\nu_\mu
\end{align}
and the symmetric energy-momentum tensor after the variation of $S_{\text{WM}}$ with respect to the metric $g$ is 
\begin{align}
 \mathbf{T}_{\mu \nu} \fdg = & \frac{\ptl\mathcal{L}_{\text{WM}}}{\ptl g^{\mu \nu}} - \halb g_{\mu \nu} \mathcal{L_\text{WM}} \notag \\
\equiv& \langle \ptl_\mu U , \ptl_\nu U \rangle_{h(U)} -\halb g_{\mu \nu} \langle \ptl^\sigma U ,\ptl_\sigma U \rangle_{h(U)}.
\end{align}
In view of the Rosenfeld-Belinfante theorem, it follows that
\[ \mathbf{S} \equiv \mathbf{T}. \]
Suppose the spacetime $M$ is foliated by the $t= \text{constant}$ Cauchy surfaces $\Sigma_t$, for some
time function $t$ and let $\mathbf{X}$ be the unit timelike normal to $\Sigma_t$, then we define the 
energy density $\mathbf{e}$
\begin{align*}
 \mathbf{e}(U) \fdg = \mathbf{T} \left(\mathbf{X},\mathbf{X}\right)
\end{align*}
and energy $E(U)(t)$
\begin{align}\label{wmener}
 E(U)(t) \fdg = \int_{\Sigma_t} \mathbf{e}\,\,\bar{\mu}_q
\end{align}
where $q$ is the induced spatial metric on $\Sigma_t$ after the canonical $m+1$ decomposition of $(M,g)$.
\subsection*{Scaling Symmetry}
In any local coordinate chart in $M$, if we scale the wave map 
\[ U (x^0,\cdots,x^m) \to U(d\, x^0,\cdots,d\,x^m) = \fdg U_d \]
for a dimensionless real parameter $d$, the wave maps equation \eqref{wmel_intro} is invariant. However the energy \eqref{wmener}
is invariant\footnote{more generally, $\Vert U_d\Vert_{\dot{H}^s(\Sigma)} =  d^{s-\frac{m}{2}}\Vert U\Vert_{\dot{H}^s(\Sigma)}$ } only in $2+1$ dimensions. 
Hence \eqref{wmel_intro} is referred to as energy \emph{critical} with respect to 
scaling for $m=2$, \emph{subcritical} for $m<2$ and
\emph{supercritical} for $m> 2$.
\subsection*{The Cauchy Problem}
Let $\Sigma$ be the initial data Cauchy surface and $\mathbf{X}$ be its unit normal, then
the Cauchy problem of wave maps is the following
\begin{equation}\label{wmcauchy}
\left. \begin{array}{rcl}
\square_g U^i + \leftexp{(h)}{\Gamma}^i_{jk} g^{\alpha \beta}\ptl_\alpha U^j \ptl_\beta U^k & =& 0 \,\,\, \text{on}\,\, M\\
\left.U\right|_{\Sigma}&=& U_0  \\
\left.\mathbf{X}(U)\right|_{\Sigma} & = & U_1 \end{array} \right\}
\end{equation}
such that 
\begin{align*}
U_0 \fdg \Sigma &\to N  \\
p & \to U_0(p)
\end{align*}
and
\begin{align*}
 U_1 \fdg \Sigma &\to T_{U_0} N \\
p &\to T_{U_0(p)} N
\end{align*}
for $ p \in  \Sigma.$

\section{Background and Overview of Previous Results}
Wave maps being the natural geometrical generalizations of the free wave equation and harmonic maps and the 
fact that their nonlinearity has special structure \footnote{ the equation \eqref{wmel_intro} satisfies the so called
null condition \cite{kl_mach93} } 
has resulted in their extensive study in the last three decades. In the 
following we shall give a brief overview of some of the main results that have been obtained for critical wave maps on the Minkowski space
\footnote{here we put the emphasis on the energy critical dimension of $m=2$, more comprehensive surveys can be found in Chapter 6
in Tao \cite{tao_book}, Chapters 7, 8 in Shatah-Struwe \cite{shatah_struwe}, Struwe \cite{struwe_wmsurvey} and
Tataru \cite{tat_wmsurvey}}. 
Let $M$ be the Minkowski space $\mathbb{R}^{m+1} $, then the Cauchy problem \eqref{wmcauchy} in the Cartesian coordinates $(t,x^1,x^2)$
reduces to 
\begin{equation} \label{wmcauchymin}
\left. \begin{array}{rcl}
\square_g U^i + \leftexp{(h)}{\Gamma}^i_{jk}(U) g^{\mu \nu}\ptl_\mu U^j \ptl_\nu U^k & =& 0 \,\,\, \text{on}\,\, \mathbb{R}^{m+1}\\
U(0,x)&=& U_0(x)  \\
U_t(0,x) & = & U_1(x) \end{array} \right\}
\end{equation}
and the energy
\[ E(U)(t) = \int_{\mathbb{R}^m} \left\Vert U_t \right\Vert^2_h + \left\Vert \grad_x U \right\Vert^2_h \,\,d\,x. \]
\subsection*{Local Existence}

\begin{itemize}
\item The wave map equations are a semi-linear system of equations, so local existence of solutions with smooth data is standard.
\item Let the initial data $(U_0,U_1)$ be in the Sobolev spaces $H^s \times H^{s-1}$, then for $s > \frac{m}{2} +1$ the local-wellposedness 
follows from standard energy methods.

\item For $s > \frac{m}{2}$ (upto critical regularity), using the fact that the wave maps equation satisfies the null condition
 local wellposedness has been proven for $m \geq 3$ Klainerman-Machedon\cite{kl_mach97} and later Klainerman-Selberg \cite{kl_selb} for $m=2$.
In the proof they used the $X^{s,b}$ spaces for the fixed point arguments. In this scenario, one should note that $H^s$ functions are continuous, therefore
 the image of the wave map $U$ is contained in a single chart of $N$, hence the problem becomes local in $N$. The global geometry of 
$N$ doesn't play a decisive role.
 \item For $s = \frac{m}{2}$, the problem is nonlocal and it depends on the global geometry on $N$. However, Tataru \cite{tat_isom} proved
local wellposedness at critical scaling assuming that $N$ isometrically embeds into a larger Euclidean space $\mathbb{R}^{n+1}$. At 
critical scaling small data small time of existence is equivalent to small data large time existence.

\end{itemize}
\subsection*{Global Existence}
\subsubsection*{Small Data Global Existence}
The small data global well-posedness has been obtained by Tataru in the Besov space $ (U_0,U_1) \in \dot{B}^{\frac{n}{2},1} \times 
\dot{B}^{\frac{n}{2}-1,1}$ for $m \geq 4$ in \cite{tat_besovh} and for $m =2,3$ in \cite{tat_besovl}. Due to the embedding $\dot{B}^{\frac{n}{2},1} \hookrightarrow  L^{\infty}$
smallness of the initial data ensures that the wave map stays in a chart in the manifold $N$. Therefore the problem is local in the 
target manifold $N$. The case of $m \geq 4$ can be handled by Strichartz estimates but it doesn't work for $m=2,3$. In the latter
case null frame spaces have been introduced to use a variant of $L^2L^{\infty}$ Strichartz estimate.

Tao proved global regularity for wave maps to $\mathbb{S}^n \subset \mathbb{R}^{n+1}$ with data in critical Sobolev spaces $H^s \times H^{s-1}$ 
for $m \geq 5$ using Strichartz estimates and microlocal gauge \cite{tao_wm1}. This result has been extended to more general targets by Klainerman-
Rodnianski (using the microlocal gauge) \cite{klain_rod_wm2001}, Statah-Struwe (using the Coulomb gauge) \cite{shatah_struwe_wm} and
Nahmoud - Stefanov - Uhlenbeck \cite{NSU_wm}.

Tao extended his result for wave maps to target $\mathbb{S}^n$  to lower dimensional cases $n=2,3$ using a combination of 
Tataru's null-frame spaces and a variant of Strichartz estimate, again using microlocal gauge to remove the ``bad'' terms \cite{tao_wm2}.

Tao's result for wave maps in lower dimensions has been extended by Krieger to wave maps with $\mathbb{H}^2$ targets for $m=3$ \cite{krieg_wmhigh} 
and later for $m=2$ \cite{krieg_wmcrit}. Instead of the microlocal gauge of Tao the Coulomb gauge of Statah-Struwe \cite{shatah_struwe_wm} was
used.

Local wellposedness at critical regularity is equivalent to global well-posedness. This was proven
by Tataru in \cite{tat_isom}.
\begin{thm}[Tataru \cite{tat_isom}]
 Let $m \geq 2$ and the manifold $N$ admits a uniform isometric embedding into Euclidean space $\mathbb{R}^{n+1}$. Then the wave maps
equation \eqref{wmcauchymin} is globally well-posed for initial data which is small in $\dot{H}^{\frac{m}{2}} \times \dot{H}^{\frac{m}{2}-1}.$
\end{thm}

\subsubsection*{ Large Data Global Existence}
One strategy to study the large energy global existence of wave maps is to divide it into the following two parts.
This approach has proven to be particularly effective for wave maps with symmetry as shown in \cite{chris_tah1, jal_tah1}.

\begin{description}
\item[(C1) Non-concentration of energy ] The energy on a spacelike surface inside the past null cone of a point goes to zero (in a
limiting sense) as one approaches the tip of the cone.

\item[(C2) Small energy global existence ] For arbitrarily small initial energy the solution can be extended smoothly and globally from
smooth initial data.
\end{description}
As a consequence of \textbf{(C1)}, the energy on a spacelike surface in the past of every point can be assumed to be small enough so that
using \textbf{(C2)} one can extend existence of solutions beyond the hypothetical singularity. Furthermore, such 
(local) solutions can be glued together to obtain a global solution.

For wave maps on Minkowski space Christodoulou and Tahvildar-Zadeh \cite{chris_tah1}, and Shatah and Tahvidar-Zadeh \cite{jal_tah1, jal_tah}
have proved \textbf{(C1)} and \textbf{(C2)} for spherically symmetric and equivariant cases respectively. In both \cite{chris_tah1} and
\cite{jal_tah1} the target is assumed to be geodesically convex, which is necessary only for the resolution of \textbf{(C1)}.

\subsubsection{Equivariant Wave Maps on the Minkowski Space}

Let $(N,h)$ be a surface of revolution with the line element 
\[ d\, s^2_h = d\, \rho^2 + f^2(\rho) d\,\phi^2 \]
in $(\rho,\phi)$ coordinates, where $f(\rho)$ is an odd, smooth function with $f(0)=0,f_\rho(0)=1.$ Then the equivariant
ansatz for the wave map $U \fdg \mathbb{R}^{2+1} \to N,$
\[ U(t,r,\theta) = (u(t,r),k\,\theta)\] 
reduces the wave maps system to
\[ \square u = k^2\frac{f(u)f_u(u)}{r^2}\]
where $f_u(u)$ is the derivative of $f$ with respect to $u$.\\
The theorem of Shatah and Tahvildar-Zadeh \cite{jal_tah1} is as follows.

\begin{thm}[Tahvildar-Zadeh, Shatah]
 If $N$ is rotationally symmetric and geosedically convex, then the Cauchy problem \eqref{wmcauchy} for an equivariant
 wave map from $ \mathbb{R}^{2+1}\to N$ has a smooth solution for all time, and $u(r,t)/r$ is also smooth.
\end{thm}
\noindent
The geodesic convexity condition is equivalent to
\begin{align}\label{geoconvex}
f_u(u)f(u) >0 \,\,\text{for}\,\, u>0.
 \end{align}
This condition has been relaxed later by Grillakis \cite{grillakis} to the following 
\[ f^2(u) + uf_u(u)f(u) >0\,\, \text{for}\,\,u>0.\]
This result has further been improved by Struwe \cite{struwe_equi} using the techniques of bubbling.

\begin{thm}[Struwe \cite{struwe_equi}]
 Let U be a (smooth) co-rotational solution to \eqref{wmcauchy} blowing up at time $t_0$. Then there exist sequences
 $r_i \to 0^-$ and $t_i \to t_0^+$ such that 
 \[ U_i(t,x)\fdg = U(t_i+r_it,r_ix) \to U_{\infty}(t,x) \]
 strongly in $H^1_{\text{loc}}(-1,1 \times \mathbb{R}^2)$, where $U_\infty$ is a non-constant, time-independent solution
 of \eqref{wmcauchy} giving rise to a non-constant, smooth equivariant harmonic map $\olin{U}\fdg S^2 \to N$.
\end{thm}
This result serves as a blow-up criterion for equivariant wave maps. In particular, if the geometry of domain and target manifolds or the
energy of the system does not admit a non-constant harmonic map then, by contradiction, one can rule out energy concentration. 

The global existence for small energy equivariant wave maps has been proven initially using the representation formula for 
inhomogeneous wave equation in \cite{jal_tah1} and later by a version of Strichartz estimate (Theorem 8.1, \cite{shatah_struwe}).
In \cite{shatah_struwe}, the equivariant wave maps equation was transformed into a critical $4+1$  wave equation \footnote{more
about this is discussed in Chapter 4}. The precise statement is as follows.
\begin{thm}[Theorem 8.1, Shatah-Struwe]
For initial energy $E<\eps$ the equivariant wave maps equation can be globally and smoothly extended from smooth initial data
\end{thm}

\subsubsection{Spherically Symmetric Wave Maps on Minkowski Space}
A similar statement has been proven for spherically symmetric wave maps on Minkowski background by Christodoulou and Tahvildar-Zadeh.
A wave map $U \fdg \mathbb{R}^{2+1} \to N$ is spherically symmetric if it depends only on $t$ and $r$. Therefore the Cauchy 
problem \eqref{wmcauchymin} reduces to
\begin{equation}\label{chris_spher}
\left. \begin{array}{rcl}
-U_{tt}^i+ U^i_{rr} + \frac{1}{r}U^i_r + \leftexp{(h)}{\Gamma}^i_{jk}(U) (-U_t^j U_t^k + U^j_r U^k_r) & =& 0 \,\,\, \text{on}\,\, \mathbb{R}^{2+1}\\
U(0,x)&=& U_0(x)  \\
U_t(0,x) & = & U_1(x) \end{array} \right\}
\end{equation}
The results of Christodoulou and Tahvildar-Zadeh\cite{chris_tah1} are as follows. Let $N$ be a complete, connected Riemannian manifold
satisfying the following conditions
\begin{itemize}
\item[(1)]  There exists an orthonormal frame of smooth vector fields $\Omega_A$ on $N$ whose
structure functions $e^C_{AB}$ are bounded.

\item[(2)] For each $p\in N$, let $\Sigma(p,s)$ be the geodesic sphere of radius $s$ centered
at p, and let $k_{AB}$ be its second fundamental form. Then there exist constants $c_1$ and $c_2$ such that 
\[ s \lambda_{\text{min}} \geq c_1 \,\,\text{and}\,\, s \lambda_{\text{max}} \leq c_2 (1+s),\]
where $\lambda_{\text{min}}$ and $\lambda_{\text{max}}$ are respectively the smallest and largest eigen values of $k_{AB},$
\end{itemize}
then we have the Theorems \ref{chris_1} and \ref{chris_2} based on the conditions (1) and (2) respectively.
\begin{thm}[Non-concentration of energy] \label{chris_1}
 Let $N$ be a Riemannian manifold satisfying (2), and let $U\fdg M \to N$ be a spherically symmetric wave map, with regular
 Cauchy data prescribed at $t=-1$ surface and the first possible singularity at the origin of the spacetime $M$. Then the
 energy of the map $E(t)$ cannot concentrate, i.e.,$E(t) \to 0$ as $t \to 0$.
\end{thm}
\begin{thm}[Small energy global existence] \label{chris_2}
 Let $N$ be a Riemannian manifold satisfying (1). Then there exists an $\eps$ depending only on the properties of $N$,
 such that any spherically symmetric wave map $U \fdg M \to N$ with regular Cauchy data of energy $E_0 < \eps^2$ 
 prescribed at $t=0$, is regular for all time.
\end{thm}
\noindent
These results combine to give the following theorem. \\
\begin{thm}[Large energy global existence]
 The Cauchy problem \eqref{chris_spher}, for a spherically symmetric wave map $U$ from the Minkowski space $R^{2+1}$ into
 a smooth, complete and connected Riemannian manifold $(N,h)$ satisfying the conditions 1 and 2, has a smooth solution
 defined for all time, regardless of the size of the data.
\end{thm}
The study of global existence for general wave maps was initiated by Tao through a series of 
papers \cite{tao_wm3,tao_wm4,tao_wm5,tao_wm6,tao_wm7}. Large energy global
existence of critical wave maps to a hyperbolic 2-plane  has been resolved by Schlag and
Krieger \cite{krieg_schlag_ccwm} by building on the concentration compactness methods of Bahouri, G\'erard \cite{bahou_gerard}
and Kenig, Merle \cite{kennig_merle2008}. In addition, \textbf{(C1)}
has been resolved for general critical wave maps without symmetry by Sterbenz and Tataru \cite{sterb_tata_main, sterb_tata_long}
using bubbling techniques. 

\section{Overview of Results}\label{ovcurrent}
In this work, we prove \textbf{(C1)} in the context of critical self-gravitating wave maps, i.e., wave maps coupled to Einstein's equations of 
general relativity. We restrict to the equivariant case. However, the techniques are expected to be effective also for critical spherically
symmetric self-gravitating wave maps\footnote{some preliminary work seems to indicate in this direction}. In the following we give a brief
overview of the set-up of the problem and the sequence of steps that result in the proof of \textbf{(C1)}. \\
Let $(\Sigma, q_0 , \mathbf{K},U_0, U_1 )$ be a smooth, compactly supported initial data set satisfying the constraint equations,
where $q_0$ is the metric of $\Sigma$, $\mathbf{K}_{\mu \nu}$ the second fundamental form of $\Sigma$ and $U_0,U_1$
are defined as in \eqref{wmcauchy}.
The Cauchy problem of critical self-gravitating wave maps is

\begin{equation}\label{ewmcauchy}
\left. \begin{array}{rcl}
\mathbf{E}_{\mu \nu} \fdg = \mathbf{R}_{\mu \nu} -\halb R_g g_{\mu \nu} &=&\mbo{\alpha} \mathbf{T}_{\mu \nu} \,\,\, \text{on}\,\, M^{2+1}\\
\square_g U^i + \leftexp{(h)}{\Gamma}^i_{jk} g^{\mu \nu}\ptl_\mu U^j \ptl_\nu U^k & =& 0 \,\,\,\,\,\,\,\,\,\,\,\,\,\,\, \text{on}\,\, M^{2+1}\\
\left.U\right|_{\Sigma}&=& U_0  \\
\left.\mathbf{X}(U)\right|_{\Sigma} & = & U_1 \end{array} 
\right\}
\end{equation}
where $\mathbf{R}$ and $R_g$ the Ricci tensor and scalar of $(M,g)$ respectively and $\mathbf{E}$ is called the Einstein tensor.
Let $(\Sigma,q_0 ,\mathbf{K})$ and $(N^2,h)$ be invariant under the action of $U(1)$ symmetry group. In particular,
let $N$ be a surface of revolution with a smooth, odd generating function $f$ such that
\[ d\,s^2_h = d\,\rho^2+ f^2(\rho) d\,\phi^2 \]
in $(\rho,\phi)$ coordinates and $f(0)=0,f_\rho(0)=1$. Let $(U_0,U_1)$ be equivariant under $U(1)$ action.\\
The system of equations in \eqref{ewmcauchy} is a symmetric hyperbolic system with smooth equivariant initial data, so there exists a unique
\footnote{upto an isometry} equivariant maximal development $(M,g,U).$ \\
Therefore, without loss of generality we can assume that the manifold $(M^{2+1},g)$ is $U(1)$ symmetric with the line element
\[ d\,s^2_g = -e^{2\Omega(t,r)} d\, t^2 + e^{2\gamma(t,r)}d\,r^2 + r^2 d\,\theta^2\]
in $(t,r,\theta)$ coordinates and $\Omega(t,r)$ and $\gamma(t,r)$ are scalar functions. $\gamma(t,0)$ is assumed to be 0 for the
regularity at the axis and $\Omega(t,0)$ can be set to $0$ by a reparameterization. With the equivariance symmetry 
$U(t,r,\theta) = (u(t,r),\theta)$, \eqref{ewmcauchy} reduces to

\begin{equation}\label{ewmcauchy_equi}
\left. \begin{array}{rcl}
\mathbf{E}_{\mu \nu} &=&\mbo{\alpha} \mathbf{T}_{\mu \nu} \,\,\,\,\,\,\,\,\,\,\,\,\,\,\, \text{on}\,\, M^{2+1}\\
\square_g u &=& \frac{f_u(u)f(u)}{r^2} \,\,\,\,\,\,\,\, \text{on}\,\, M^{2+1}\\
\left.U\right|_{\Sigma}&=& U_0  \\
\left.\mathbf{X}(U)\right|_{\Sigma} & = & U_1 \end{array} 
\right\}
\end{equation}

It has been proven that during Cauchy evolution the blow up, if it were to happen, can happen only on the axis of 
of $M$\cite{laan_equinotes}. Therefore, it is sufficient to study the properties of evolution near the axis. Furthermore, one could adapt 
the methods developed by Christodoulou \cite{chris_selfgrav} and Dafermos\cite{dafermos_trap} to 2+1 dimensional U(1) space times to prove that 
there are no trapped surfaces or marginally trapped surfaces during the evolution of the space time with equivariant wave map as a source\cite{laan_equinotes}. 
However, it is well known that in 2+1 dimensions the formation of outer trapped surfaces or marginally outer trapped surfaces  can be ruled out due to the 
works of Ida\cite{ida} and Galloway, Schleich, Witt\cite{gallo}.

Without loss of generality, we assume that the initial data is specified at $t=-1$ surface and that the 
first (hypothetical) singularity is at the origin $O$ of our coordinate system. The energy density 
$ \mathbf{e} \fdg = \mathbf{T} (\mathbf{X_1},\mathbf{X_1})$, where $\mathbf{X_1} = e^{-\Omega} \ptl_t$ is 
the unit timelike normal vector
of the $t = \text{constant} $ surface $ \Sigma_t$ and energy $E^O(t) \fdg = \int_{\Sigma_t \cap J^-(O) } \mathbf{e}\,\, d\,\bar{\mu}_q$
where $q$ is the induced spatial metric of $\Sigma_t$ after the $2+1$ decomposition of $(M,g)$.\\
 We use the vector fields method to study the evolution of wave maps in the truncated backward null cone of the point $O$.
We construct the appropriate momentum vector fields $\mathbf{P_X}$ for apt choices of multipliers $\mathbf{X}$ as follows
\[ \mathbf{P}^\mu_{\mathbf{X}} = \mathbf{T}^\mu_{\,\,\,\nu} \, \mathbf{X}^{\nu}. \]
We then use the Stokes' theorem on a truncated backward null cone of $O$ to estimate the divergence of $\mathbf{P_X}$
as we approach $O$ in a limiting sense. In the following we show the sequence of steps  that prove the non-concentration
of energy of critical equivariant self-gravitating wave maps.
\begin{enumerate}
\item
Assuming that the target manifold satisfies the condition 
\begin{align} \label{sphereatinfinity}
 \int_0^u f(s) \,d\,s \to \infty \,\,\,\text{as}\,\,\, u \to \infty
 \end{align}
we prove that 
\[ ||u||_{L^{\infty}} \leq c \]
for every solution $u$ of the equivariant wave map system \eqref{ewmcauchy_equi}.
\item
\begin{figure*}[!hbt]
\psfrag{O}{$O$}
\psfrag{Ktaus}{$K(\tau,s)$}
\psfrag{teq0}{$t=0$}
\psfrag{teqs}{$t=s$}
\psfrag{teqtau}{$t=\tau$}
\psfrag{teq-1}{$t=-1$}
\psfrag{Fluxpx}{$\text{Flux}(\mathbf{P_X})(t,s)$}
\psfrag{E(s)}{$E^O(s)$}
\psfrag{E(tau)}{$E^O(\tau)$}
\centerline{\includegraphics[height=2.5in]{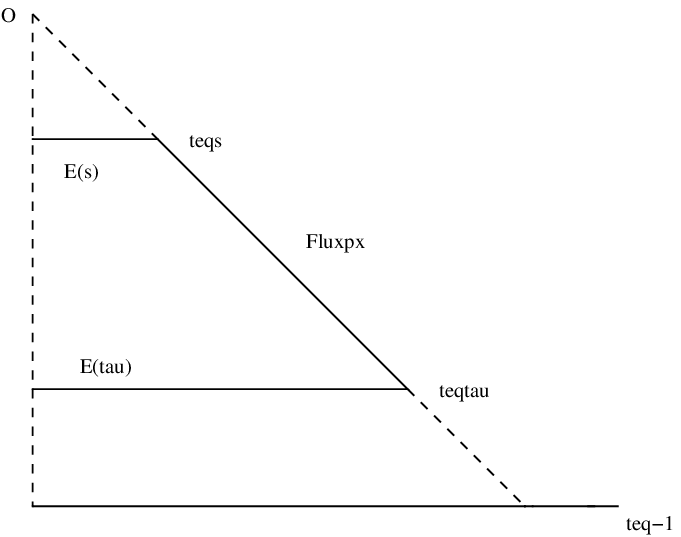}}
\caption{Application of the Stokes' theorem for the divergence of $\mathbf{P}_\mathbf{X_1}$ } 
\label{fig:fried-dyn}
\end{figure*}

Using the multiplier $\mathbf{X_1} \fdg = e^{-\Omega}\ptl_t$ we construct a divergence free momentum vector
$\mathbf{P}_{\mathbf{X_1}}$. Applying the Stokes'theorem in the region $K(\tau,s)$ and observing that the 
flux of $\mathbf{P}_{\mathbf{X_1}}$ through the truncated past null surface of $O$ is non-positive, we prove 
that
\[ E^O(s) \leq E^O(\tau) \,\, \text{for}\,\, -1\leq \tau \leq s < 0. \]
\item
The divergence free vector $\mathbf{P}_{\mathbf{X_1}}$ is used again to relate the fluxes through the surfaces
$\ptl \cal{S}_1$, $\ptl \cal{S}_2$ and $\ptl \cal{S}_3$ in the ``exterior'' of the interior of the past null cone of $O$ (as shown in the figure \ref{fig:annular_disc_1_intro} )
\begin{figure}[!hbt]
\psfrag{O}{$O$}
\psfrag{O}{$O$}
\psfrag{teqtau}{$t=\tau$}
\psfrag{teq-1}{$t=-1$}
\psfrag{S}{$\cal{S}$}
\psfrag{1}{$\ptl \cal{S}_1$}
\psfrag{2}{$\ptl \cal{S}_2$}
\psfrag{3}{$\ptl \cal{S}_3$}
\psfrag{Req0}{$R=0$}
\psfrag{ReqlT}{$R=\lambda T$}
\psfrag{ReqT}{$R=|T|$}

\centerline{\includegraphics[height=2.5in]{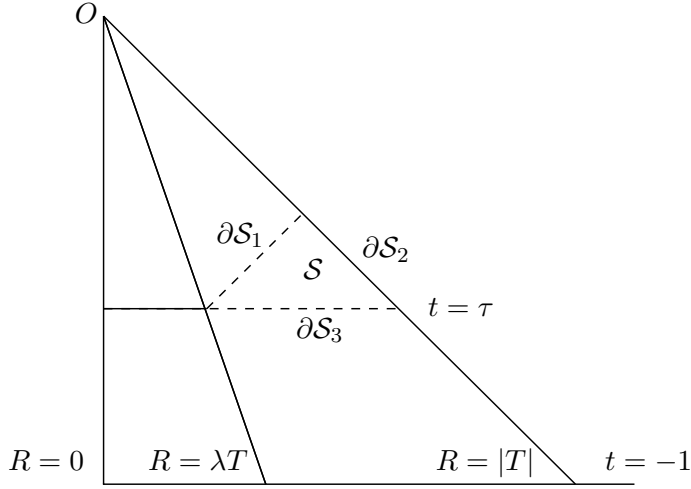}}
\caption{Non-concentration of energy away from the axis}
\label{fig:annular_disc_1_intro}
\end{figure} 

In addition, the multiplier $\mathbf{X_2} \fdg= e^{-\gamma}\ptl_r$ is used to construct an identity which is used
in a Gr\"onwall estimate to prove that energy doesn't concentrate away from the axis is i.e.,
\[ E^O_{\text{ext}} \fdg = \int_1 \mathbf{e} \, \bar{\mu}_q \to 0 \,\,\,\, \text{as}\,\,\, \tau \to 0\]
The introduction of the parameters $k_\ell$ and $k_n$ is crucial in neutralizing the ``bad'' terms in
energy identities, before setting up the Gr\"onwall estimate for the integrand of the flux of $\mathbf{P}_{\mathbf{X_1}}$
through the null surface $\cal{S}_1$. Overall, this is the most technical step and involves various estimates to prove 
the decay of fluxes through the null surfaces $\ptl \cal{S}_1$ and $\ptl \cal{S}_2$.
\item
A Morawetz multiplier $\mathbf{X_3} \fdg = r\ptl_r$ allows us to construct a momentum $\mathbf{P}_{\mathbf{X_3}}$ whose
divergence is the kinetic energy density $\mathbf{e}_{\text{kin}} \fdg = \halb e^{-2\Omega} u^2_t$. Extending the framework
of step 2, we use the Stokes' theorem on $K(\tau,s)$ and prove non-concentration of the bulk term by estimating the boundary
terms

\[\frac{1}{r(\tau)}\int_{K_\tau} e^{-2\Omega} u_t^2 \,\bar{\mu}_g \to 0 \] as $\tau \to 0$,
where $K_\tau$ is the backward null cone with the tip at the origin and base in the $ t= \tau$ slice, and $r(\tau)$
is the radial function along the mantel of the cone. In estimating the boundary term we also use the step 3. \\
It is predicted that a critical concentrating self-gravitating equivariant wave map goes to a harmonic map (static
solution) in $H^1_{\text{loc}}$. The statement of step 4 is expected to play a vital role in proving this statement.
\item
Finally the non-concentration of energy is proven by constructing a vector field $\mathbf{P}_{\text{tot}}$ as
\[ \mathbf{P}^\nu_{\text{tot}} \fdg= \mathbf{P}^\nu_{\mathbf{X_4}} + \mathbf{P}^\nu_{\kappa},  \]
where $\nu= 0, 1, 2$ in $(t,r,\theta)$ coordinates, $\mathbf{P}_{\mathbf{X_4}}$ is the corresponding momentum of 
the multiplier $\mathbf{X_4} \fdg= r^a \ptl_r $ for
$ a \in (\halb,1) $
and 
\[ \mathbf{P}^\nu_{\kappa} \fdg = \kappa u^\nu u - \ptl^\nu \kappa \frac{u^2}{2} \]
for $\kappa \fdg = \frac{1-a}{2}r^{a-1}$.
A similar technique of using the Stokes' theorem and estimating the boundary terms as above gives us the 
result that energy in the past null cone of any point does not concentrate provided the target manifold $(N,h)$ 
satisfies 
\begin{align} \label{grillakis_condition}
f(s)f_s(s)s + f^2(s) > 0 \,\,\,\text{for}\,\,\,  s >0 .
\end{align}
 This is the self-gravitating equivalent of the theorem of Grillakis \cite{grillakis} for equivariant wave maps
 on Minkowski background.
\end{enumerate}
The final statement can be formulated in terms of the following theorem.
\begin{thm}[Non-concentration of energy]
Let $(M,g,U)$ be a smooth, globally hyperbolic, equivariant maximal development of smooth, compactly supported equivariant 
initial data set $(\Sigma,q,\mathbf{K}, U_0, U_1)$ with finite initial energy $E_0$ and satisfying the constraint equations,
and let $(N,h)$ be a rotationally symmetric, complete, connected Riemannian manifold satisfying \eqref{sphereatinfinity}
and \eqref{grillakis_condition}
then the energy of the Einstein-wave map system cannot concentrate, i.e., $E^O (t) \to 0$, where $O$ is the
first (hypothetical) singularity of $M$.
\end{thm}

\chapter{Critical Self-Gravitating Wave Maps}
\section{Variational Formulation}
As introduced in the \eqref{wmlag}, the wave map action is as follows
\begin{align} \label{wmlag_new}
 S_{\text{WM}} (U) = \halb \int_M \langle U^\sigma, U_\sigma \rangle_{h(U)} \,\bar{\mu}_g . 
\end{align}
In the following we shall derive the Euler-Lagrangian equations corresponding to the first 
variation of $S_{\text{WM}}$.
Let $ U_\lambda \fdg M \to N$ be a one parameter family of maps \footnote{for every compactly supported $\mathcal{U}$, such a family of maps
is achieved for instance by taking $U_\lambda \equiv \exp_U {(\lambda\, \mathcal{U})}$ } such that 
\begin{align*}
 U_0 \equiv& \,U \\
U_\lambda \equiv&\, U \,\,\text{outside a compact set }
\end{align*}
and
\begin{align*}
 \mathcal{U} \fdg = \left. \frac{d}{d\lambda} U_\lambda \right|_{\lambda =0} 
\end{align*}
where $\mathcal{U} \in U^*TN$ and $\mathcal{U} \equiv 0$ outside the compact set,
then $U$ is a critical point iff the first variation of $S_{\text{WM}}$
\[ \left. \frac{d}{d\,\lambda} S_{\text{WM}}(U_\lambda)\right|_{\lambda =0} =0.  \]
Explictly,
after using the Leibnitz rule on the integrand, we get
\begin{align}
 \left . \frac{d}{d\,\lambda} S_{\text{WM}}(U_\lambda)\right|_{\lambda =0} 
=& \,\halb \int_M \left( g^{\mu \nu} \frac{\ptl h_{ij}} {\ptl U^k} 
\mathcal{U}^k \ptl_\mu U^i \ptl_\nu U^j  + 
2g^{\mu \nu} h_{ij} (U) \ptl_\mu \mathcal{U}^i \ptl_\nu U^j \right)\, \bar{\mu}_g \label{fvar}
\end{align}
Let $\cal{D}$ be a closed, oriented subset of $M$ such that $\text{supp}(\cal{U}) \subset \cal{D}$.
Now consider the quantity $\grad_\mu (h_{ij}(U)\,\mathcal{U}^i\, \ptl^\mu U^j)$, we have
\begin{align*}
 \grad_\mu (h_{ij}(U)\, \mathcal{U}^i\, \ptl^\mu U^j) 
=&\, h_{ij} \mathcal{U}^i\, \square_g  U^j + h_{ij}(U)\, \ptl^\mu U^j\, \grad_\mu \mathcal{U}^i +
\ptl^\mu U^j \mathcal{U}^i \frac{\ptl h_{ij}}{\ptl U^k}\,\ptl_\mu U^k
\end{align*}
However, since $\mathcal{U} =0$ on $\ptl \mathcal{D}$, the Stokes' theorem gives
\[ \int_D \grad_\mu (h_{ij}(U) \mathcal{U}^i \ptl^\mu U^j)\,  \bar{\mu}_g = 0, \]
and therefore we have,
\begin{align*}
 \int_\mathcal{D} g^{\mu \nu} h_{ij} (U) \ptl_\mu \mathcal{U}^i \ptl_\nu U^j \, \bar{\mu}_g = 
- \int_\mathcal{D} h_{ij} \mathcal{U}^i \square_g  U^j + \ptl^\mu U^j \mathcal{U}^i \frac{\ptl h_{ij}}{\ptl U^k} \ptl_\mu U^k \,\bar{\mu}_g.
\end{align*}
Now if we go back to \eqref{fvar},
\begin{align*}
 \left . \frac{d}{d\,\lambda} S_{\text{WM}}(U_\lambda)\right|_{\lambda =0} = & - \int_\mathcal{D} h_{ij} \mathcal{U}^i \square_g  U^j + \ptl^\mu U^j \mathcal{U}^i \frac{\ptl h_{ij}}{\ptl U^k}\ptl_\mu U^k -
\halb g^{\mu \nu} \frac{\ptl h_{ij}} {\ptl U^k} \mathcal{U}^k \ptl_\mu U^i \ptl_\nu U^j \, \bar{\mu}_g  \\
=& -\int_\mathcal{D} h_{ij} \mathcal{U}^i \square_g  U^j + \halb g^{\mu \nu}\mathcal{U}^i  \left( \frac{\ptl h_{ij}}{\ptl U^k} + \frac{\ptl h_{ik}}{\ptl U^j} \right)  \ptl_\mu U^k\ptl_\nu U^j \\
& \quad -\halb g^{\mu \nu} \frac{\ptl h_{ij}} {\ptl U^k} \mathcal{U}^k \ptl_\mu U^i \ptl_\nu U^j \,\bar{\mu}_g \\
=& -\int_\mathcal{D} h_{ij} \mathcal{U}^i \square_g  U^j+ \halb g^{\mu \nu}\mathcal{U}^i  \left( \frac{\ptl h_{ij}}{\ptl U^k} + \frac{\ptl h_{ik}}{\ptl U^j} 
-\frac{\ptl h_{jk}}{\ptl U^i} \right)  \ptl_\mu U^j\ptl_\nu U^k \,\bar{\mu}_g \\
=&-\int_\mathcal{D} h_{ij} \mathcal{U}^i \square_g  U^j+ \halb g^{\mu \nu}\,\mathcal{U}^i  \delta^s_i \left( \frac{\ptl h_{sj}}{\ptl U^k} 
+ \frac{\ptl h_{sk}}{\ptl U^j} 
-\frac{\ptl h_{jk}}{\ptl U^s} \right)  \ptl_\mu U^j\ptl_\nu U^k \,\bar{\mu}_g \\
=&-\int_\mathcal{D} h_{il} \mathcal{U}^i \square_g  U^l+ \halb g^{\mu \nu}\,\mathcal{U}^i  h^{ls}h_{il} \left( \frac{\ptl h_{sj}}{\ptl U^k} 
+ \frac{\ptl h_{sk}}{\ptl U^j} 
-\frac{\ptl h_{jk}}{\ptl U^s} \right)  \ptl_\mu U^j\ptl_\nu U^k \,\bar{\mu}_g \\
=&-\int_\mathcal{D} h_{il} \mathcal{U}^i \left( \square_g  U^l+ \halb g^{\mu \nu}\,\mathcal{U}^i  h^{ls} \left( \frac{\ptl h_{sj}}{\ptl U^k} 
+ \frac{\ptl h_{sk}}{\ptl U^j} 
-\frac{\ptl h_{jk}}{\ptl U^s} \right)  \ptl_\mu U^j\ptl_\nu U^k \right) \, \bar{\mu}_g \\
=&-\int_\mathcal{D} h_{il} \mathcal{U}^i \left( \square_g  U^l+ g^{\mu \nu}\, \leftexp{(h)}{\Gamma}^l_{jk}(U) \ptl_\mu U^j\ptl_\nu U^k \right) \,\bar{\mu}_g \\
\end{align*}

Therefore, after relabeling the indices,
the Euler-Lagrange equations in local coordinates take the following form \footnote{throughout
the course of this work we shall use this intrinsic form of Euler-Lagrangian equations }
\begin{align} \label{wmel}
 \square_g U^i + \leftexp{(h)}{\Gamma}^i_{jk}(U) g^{\alpha \beta}\ptl_\alpha U^j \ptl_\beta U^k &= 0,
\end{align}
where $\leftexp{(h)}{\Gamma}$'s are the Christoffel symbols of the target $N$
\[ \leftexp{(h)}{\Gamma}^i_{jk} \fdg = \halb h^{il}\left( \ptl_k h_{mj} + \ptl_i h_{mk} - \ptl_m h_{ij}\right).\]

The action \eqref{wmlag_new} can be generalized to include Einstein-Hilbert action of general relativity as follows
\begin{align}\label{Einstein_wavemap_action}
 S_{\text{EWM}} [U,g] \fdg = \halb \int_M \frac{1}{\mbo{\alpha}}R_g - \langle \ptl^\sigma U, \ptl_\sigma U \rangle _h \, \bar{\mu}_g 
\end{align}
where $R_g$ is the scalar curvature of $(M,g).	$
In local coordinates the Euler-Lagrange equations of the functional $S_{\text{EWM}} [U,g]$ are the following system of equations
\begin{subequations}\label{ewm_el}
\begin{align}
\mathbf{E}_{\mu \nu} \fdg = \mathbf{R}_{\mu \nu} -\halb R_g g_{\mu \nu} &=\mbo{\alpha} \mathbf{T}_{\mu \nu} \\
 \square_g U^i + \leftexp{(h)}{\Gamma^i_{jk}} g^{\mu \nu}\ptl_\mu U^j \ptl_\nu U^k &= 0 
\end{align}
\end{subequations}
where $\mathbf{R}$ is the Ricci tensor of $(M,g)$, $\mathbf{E}$ is the Einstein tensor and $\mbo{\alpha}$ is the 
gravitational coupling constant.

\section{Reduction of 3+1 Einstein's Equations}
Critical wave maps coupled to Einstein's equations of general relativity can be interpreted as reduced 3+1 vacuum Einstein's equations
on principal bundles with 1-parameter spacelike isometry groups\cite{kaluz1,kaluz2}. Following \cite{kaluz3, laanglob}, we shall show the derivation 
of this reduction for the case of self-gravitating wave maps. In the following sections we illustrate how the equations can be 
reduced further with various symmetry assumptions.

Let $G_1$ be a one-dimensional Lie group and $\olin{\Sigma}$ be a principal fiber bundle with base, the 
Riemannian 2-manifold $\Sigma$
and group $ G_1$. Now consider a 3+1 dimensional Lorentzian manifold $(\bar{M},\bar{g})$ such that 
$\bar{M} \fdg= \olin{\Sigma} \times \mathbb{R}$ such that the submanifolds $\olin{\Sigma}_t \fdg= 
\olin{\Sigma} \times \{t\} $ are space-like and $ {x} \times \mathbb{R}$ time-like. Note that, by definition,
$G_1$ acts transitively and freely on $\bar{M}$. Furthermore, we suppose that the metric $\bar{g}$ is invariant 
under the right action of the group $G_1$ on $\bar{M}$. We introduce coordinates adapted to this symmetry. Let
$(x^\mu)$ be the local coordinates on $M$, $\mu = 0,1,2$ and let $x^3$ be a local coordinate in $G_1$ corresponding
to a local trivialization of $M$ over $D_{\Sigma}$.
With the assumptions above, the metric $\bar{g}$ can be expressed in terms of $\tild{g}$, the $\pi$-induced metric on $\Sigma \times \mathbb{R}$ as follows
\[ \bar{g} = \pi^{*} \tild{g} + \pi^{*} (e^{2\psi}) (\ulin{\theta})^2\]
where $\psi$ is a function on $\Sigma \times \mathbb{R}$ and $\pi$ is the 
bundle projection $M \to \Sigma \times \mathbb{R}$ and 
\[ \ulin{\theta} \fdg= d\,x^3 + \mathbf{A}_\nu\,d\,x^\nu\]
Note that in the coordinate system chosen above, $\pi^*\tild{g} =\tild{g}. $
The Ricci tensor $\bar{\mathbf{R}}$ of $(\bar{M},\bar{g})$ can be expressed in terms of the Ricci tensor $\tild{\mathbf{R}}$
of $(M,\tild{g})$ as follows
\begin{subequations}\label{kkreduced_ricci}
\begin{align}
\bar{\mathbf{R}}_{\mu \nu} \equiv& \tild{\mathbf{R}}_{\mu \nu} - \ptl_\mu \psi \ptl_\nu \psi - \grad_\mu \ptl_\nu \psi -\halb e^{2\psi} \mathbf{F}_{\mu \sigma} \mathbf{F}_{\nu}^{\sigma} \\
 \bar{\mathbf{R}}_{\mu 3} \equiv& \halb e^{-\psi} \grad_\sigma (e^{3\psi} \mathbf{F}^\sigma_{\mu}) \\
 \bar{\mathbf{R}}_{33} \equiv& -e^{2\psi} (\tild{g}^{\mu \nu}\grad_\mu \ptl_\nu \psi + \tild{g}^{\mu \nu} \ptl_{\mu}\psi \ptl_{\nu} \psi -\frac{1}{4} e^{2\psi}\mathbf{F}_{\mu \nu} \mathbf{F}^{\mu\nu}  )
\end{align}
\end{subequations}
where $\mathbf{F}_{\mu \nu}$ is a 2-form on $M$ such that
\[ \mathbf{F}_{\mu \nu} \fdg = \grad_{\mu} \mathbf{A}_\nu- \grad_{\nu} \mathbf{A}_\mu \]
in the chosen coordinate frame $d\,x^\nu$.
We introduce the dual of the 2-form $e^{3\gamma}\mathbf{F}$
\[ \mathbf{G} \fdg = e^{3\psi} \, \leftexp{*}{\mathbf{F}} \]
However the 1-form $\mathbf{G}$ is closed due to the equations $ \mathbf{R}_{\mu 3}=0$ i.e.,
\[ d\, \mathbf{G} =0. \]
We assume that $M$ is contractible, therefore by Poincar\'e lemma,
there exists a potential $\omega$ such that 
\[ \mathbf{G} = d\,\omega \]
The scalar function $\omega$ on $M$ is called the twist potential.
The equation $d\, \mathbf{F} =0$ translates to
\begin{align} \label{twist_pot}
 \tild{\grad}_{\mu}(e^{-3\psi} \tild{g}^{\mu \nu} \ptl_\nu \omega) =0 
\end{align}
for the twist potential $\omega$. To reduce the equations in \eqref{kkreduced_ricci}
to Einstein-wave map system, we introduce the conformal metric $g$ such that
\[ g \fdg= e^{2\psi} \tild{g} \]
Then the equation $e^{-4\psi}\bar{\mathbf{R}}_{33} \equiv 0$ can be rewritten as 
\begin{align}\label{kk_wm1}
 \grad^\mu \ptl_\mu \psi + \halb e^{-4\psi} g^{\mu \nu} \ptl_\mu \omega \ptl_\nu \omega =& 0 
\end{align}
where $\grad$ is the covariant derivative with respect to the metric $g$.
On the other hand, the equation \eqref{twist_pot} translates to
\begin{align}\label{kk_wm2}
 \grad^\mu \ptl_\mu \omega - 4\, g^{\mu \nu} \ptl_\mu \psi \ptl_\nu \omega =&0.
\end{align}
Using the formulas to relate the Ricci tensors in two conformal metrics we get the following
for $(\bar{M},\bar{g})$ satisfying Einstein's equations
\begin{align}\label{kk_tracerev}
0 = \bar{\mathbf{R}}_{\mu \nu} + \tild{g}_{\mu \nu} \bar{\mathbf{R}}_{33} = \mathbf{R}_{\mu\nu}
- \halb \left( e^{-4\psi} \ptl_{\mu}\omega \ptl_{\nu} \omega + 4 \,\ptl_{\mu}\psi \ptl_{\nu}\psi\right)
\end{align}
Now consider a wave map 
\begin{align}
U \fdg\, (M,g) &\to (N,h) \notag \\
 p &\to (\psi,\omega)
\end{align}
where $N$ is the hyperbolic 2-plane with the line element
\begin{align} \label{hyp2plane}
 ds_h^2 = 2\,d\rho^2 + \halb e^{-4\rho}\, d\vartheta^2 
\end{align}
then we have
\[\leftexp{(h)}{\Gamma}^1_{11} =0, \leftexp{(h)}{\Gamma}^1_{12} =0,\leftexp{(h)}{\Gamma}^1_{22} = \halb e^{-4\rho} \]
and 
\[ \leftexp{(h)}{\Gamma}^2_{22} =0, \leftexp{(h)}{\Gamma}^2_{12} =-2, \leftexp{(h)}{\Gamma}^2_{11} =0. \]
The equations \eqref{kk_wm1} and \eqref{kk_wm2} resemble the wave map equations \eqref{wmel} with $(N,h)$ as the
target, and the equations \eqref{kk_tracerev} 
\begin{align}
 \mathbf{R}_{\mu \nu} =& \halb \left( e^{-4\psi} \ptl_{\mu}\omega\, \ptl_{\nu} \omega + 4 \,\ptl_{\mu}\psi\, \ptl_{\nu}\psi\right) \notag \\
=& \ip{U_\mu}{U_\nu}_{h(U)} \label{trace_rev}
\end{align}
are the trace reversed Einstein's equations on $(M,g)$. Therefore, after reversing the trace of \eqref{trace_rev} we get,
\begin{align} 
 \mathbf{R}_{\mu \nu} - \halb g_{\mu\nu} R_g =& \halb \left( e^{-4\psi} \ptl_{\mu}\omega \ptl_{\nu} \omega + 4 \,\ptl_{\mu}\psi \ptl_{\nu}\psi\right)
- \halb g_{\mu \nu} g^{\sigma \upsilon} \left( 2 \psi_\sigma \psi_\upsilon + \halb e^{-4\psi} \omega_\sigma \, \omega_\upsilon ) \right) \notag \\
=& \ip{U_\mu}{U_\nu}_{h(U)} -\halb g_{\mu \nu} \ip{U^\sigma}{U_\sigma}_{h(U)} \notag \\
=& \,\mathbf{T}_{\mu \nu} \label{ewm_aftertrace}.
\end{align}
Finally, collecting the equations \eqref{kk_wm1}, \eqref{kk_wm2} and \eqref{ewm_aftertrace},
we get the Einstein-wave map system in the form shown in \eqref{ewm_el}.
\begin{subequations}
\begin{align}
\mathbf{E}_{\mu \nu} \fdg=\mathbf{R}_{\mu \nu} - \halb g_{\mu\nu} R_g =&\, \mathbf{T}_{\mu \nu} \\
\grad^\mu \ptl_\mu \psi + \halb e^{-4\psi} g^{\mu \nu} \ptl_\mu \omega \ptl_\nu \omega =& \,0 \\
\grad^\mu \ptl_\mu \omega - 4 \, g^{\mu \nu} \ptl_\mu \psi \ptl_\nu \omega =&\,0.
\end{align}
\end{subequations}


\subsection*{The Polarized Case}
If we restrict to the spacetimes where the Killing vector field generating the isometry is orthogonal
\footnote{this condition is satisfied only if $\bar{M}$ is a trivial bundle} to the hypersurface 
$\Sigma \times \mathbb{R}$ then we have $\mathbf{A}_{\mu} \equiv 0$, as a consequence we have $\omega =$ constant. 
The matter field now is just a linear wave equation for $\psi$. 
So the system of equations reduce to the following form
\begin{subequations} \label{kk_polarized}
\begin{align}
 \mathbf{R}_{\mu \nu} =& \grad_\mu\,\psi \grad_\nu \,\psi  \\
 \square_{\,g}\, \psi =& 0.
\end{align}
\end{subequations}

\section{Equivariant Self-Gravitating Wave Maps}
Let us define equivariant wave maps.
Let $M$ be $(2+1)$ dimensional with  $SO(2)$ symmetry with the line element of the form

\[ d\,s_g^2 = -e^{2\Omega} dt^2 + e^{2\gamma}dr^2 + r^2 d\theta^2 \]
in the polar coordinates $t, r, \theta$ and, $\Omega = \Omega(t,r)$ and
$\gamma = \gamma(t,r) $ are functions of $t$ and $r.$ In the null coordinates
$\xi$, $\eta$ ,$\theta$ 
\[ d\,s_g^2= -e^{2z} d\,\xi \, d\, \eta + r^2 d\theta^2 ,\]
where $z = z(\xi,\eta)$ and $r = r(\xi,\eta)$ are functions of $\xi$ and $\eta.$
Further suppose that $N$ is a surface of revolution with the metric
\[ ds_h^2 = d\rho^2 + f^2(\rho)\, d\phi^2 \]
where $f$ is a smooth function with $f(0) =0$ and $f_{\rho}(0) = 1$ ( $f_\rho$ 
is the derivative of $f$ with respect to $\rho$).
Then the equivariant wave maps $U \fdg M \to N$ are the ones which have the following form
\begin{align*}
 U(t,r,\theta) & = (U^1,U^2) \\
&= (u(t,r),k \theta)
\end{align*}
in the $(\rho,\theta)$ coordinates, for some scalar function $u(t,r)$ and integer $k$(the homotopy degree). 
In other words $U$ maps the orbits
of $M$ under $U(1)$ action to the orbits of $N$ under the $U(1)$ symmetry action.
We have $\leftexp{(h)}{\Gamma}^{1}_{11}, \leftexp{(h)}{\Gamma}^{1}_{12}  =0$ and
$ \leftexp{(h)}{\Gamma}^{1}_{22}= -f(\rho)f_\rho(\rho)$. So the system \footnote{due to the decoupling nature
of the equivariant ansatz the other equation for $U^2$ is a triviality}
\[ \square_g U^1 + \leftexp{(h)}{\Gamma}^{1}_{jk}(U) g^{\mu \nu} \ptl_\mu U^j \ptl_\nu U^k =0 \]
reduces to 
\begin{align*}
 \square_g U^1 + \leftexp{(h)}{\Gamma}^{1}_{22}(U) g^{\theta \theta} (\ptl_\theta (k\theta))^2 =0.
\end{align*}
Therefore, the wave maps system reduces to the following equation for the function $u(t,r)$
\[ \square_g u = k^2\frac{f(u)f_u(u)}{r^2}\]
where
\[ \square_g u = -e^{-2\Omega}(u_{tt} + (\gamma_t-\Omega_t)u_t) + e^{-2 \gamma}(u_{rr} + \frac{u_r}{r} + (\Omega_r - \gamma_r)u_r).\]
The self-gravitating Einstein equivariant wave map system is 

\begin{align*}
 \mathbf{E}_{\mu \nu} =& \,\mbo{\alpha}\, \mathbf{T}_{\mu \nu} \\
\square_g u  =& k^2 \frac{f(u)f_u(u)}{r^2} \\
\end{align*}
where $ \mbo{\alpha} $ is the gravitational coupling constant.
\subsection*{Polarized Case With Two Killing Vector Fields}
We can also consider a special case of $k=0$, the system reduces to the polarized case of
3+1 vacuum Einstein equation with two Killing space-like vector fields
\begin{align*}
 \mathbf{R}_{\mu \nu} =& \,\ptl_\mu u \ptl_\nu u  \\
 \square_{\,g}\, u =& \,0.
\end{align*}

\section{Spherically Symmetric Self-Gravitating Wave Maps}
In the following we shall explore another variant of symmetry for self-gravitating wave maps.
Let $(M,g)$ be invariant under the action of $U(1)$ isometry group and let us choose polar 
coordinates $(t,r,\theta)$ as above, then the metric is
\[ d s_g^2 = -e^{2\Omega} dt^2 + e^{2\gamma}dr^2 + r^2 d\theta^2. \]
We define a spherically symmetric wave map to be the map $U (M,g) \to (N,h)$ which depends only on $t$ and $r$,
i.e., $U = U(t,r)$. Therefore, the wave map system of equations
\[ \square_g U^i + \leftexp{(h)}{\Gamma}^i_{jk}(U) g^{\mu \nu}\ptl_\mu U^j \ptl_\nu U^k = 0 \]
reduces to
\[\square_g U^i + \leftexp{(h)}{\Gamma}^i_{jk} (U) ( -e^{-2\Omega} U_t^jU_t^k + e^{-2\gamma}U^j_r U^k_r) =0 \]
where
 \[ \square_g U^i = -e^{-2\Omega}(U^i_{tt} + (\gamma_t-\Omega_t)U^i_t) + e^{-2 \gamma}(U^i_{rr} + \frac{U^i_r}{r} + (\Omega_r - \gamma_r)U^i_r). \]
Note that, unlike in the equivariant case we don't need to assume symmetry on the target manifold $N$. However, if we assume
that target manifold to be a hyperbolic 2-plane, the critical spherically symmetric self-gravitating wave map system 
can be interpreted as reduction of 3+1 vacuum Einstein equations with two Killing vector fields. Therefore, if 
we assume the target manifold to be $(N,h)$ with $h$ as in \eqref{hyp2plane}
we get the following Einstein-spherically symmetric wave map system
\begin{align*}
 \mathbf{E}_{\mu \nu} =& \,\mbo{\alpha}\, \mathbf{T}_{\mu \nu} \\
\square_g U^1 \,+& \halb e^{-4U^1} \left( -e^{-2\Omega} U_t^2U_t^2 + e^{-2\gamma}U^2_r U^2_r \right) =0 \\
\square_g U^2 \,-&4\left(-e^{-2\Omega} U_t^1U_t^2 + e^{-2\gamma}U^1_r U^2_r\right) =0.
\end{align*}
If the two Killing vector fields of $M$ commute, the 3+1 Einstein equations can be reduced further to wave maps on Minkowski space
with spherical symmetry as shown in \cite{G2spacetimes}. In \cite{G2spacetimes}, the work of Christdoulou and Tahvildar-Zadeh 
\cite{chris_tah1} has been used to prove that the maximal Cauchy development of the corresponding Cauchy problem is geodesically complete.

 \chapter{The Problem of Critical Equivariant Self-Gravitating Wave Maps}
\section{The Cauchy Problem}

We formulate the Cauchy problem for the critical self-gravitating equivariant
wave map system. Let $(\Sigma, q_0, \mathbf{K}, U_0, U_1)$ be an initial data set
of the Einstein-wave map system satisfying the constraint equations,
where $q_0$ is the metric of $\Sigma$, $\mathbf{K}_{\mu \nu}$ is the second fundamental form of $\Sigma$,

\begin{align*}
U_0 \fdg \Sigma &\to N  \\
p & \to U_0(p),
\end{align*}
and $U_1$ is the derivative at the initial data surface $\Sigma$ of the wave map $U$ 
with respect to the unit timelike normal $\mathbf{X}$ \footnote{the normal $\mathbf{X}$ is also used to define the second
fundamental form $\mathbf{K}$}  of $\Sigma$ . So,
\begin{align*}
 U_1 \fdg \Sigma &\to T_{U_0} N \\
p &\to T_{U_0(p)} N.
\end{align*}
Now we are ready to state the Cauchy problem of the critical self-gravitating 
wave map system.
\begin{equation}\label{ewmcauchy_new}
\left. \begin{array}{rcl}
\mathbf{E}_{\mu \nu} &=&\mbo{\alpha} \mathbf{T}_{\mu \nu} \,\,\, \text{on}\,\, M^{2+1}\\
\square_g U^i + \leftexp{(h)}{\Gamma}^i_{jk} g^{\mu \nu}\ptl_\mu U^j \ptl_\nu U^k & =& 0 \,\,\,\,\,\,\,\,\,\,\,\,\,\,\, \text{on}\,\, M^{2+1}\\
\left.U\right|_{\Sigma}&=& U_0  \\
\left.\mathbf{X}(U)\right|_{\Sigma} & = & U_1 \end{array} 
\right\}
\end{equation}
We assume that all the data $ q_0, \mathbf{K}, U_0, U_1$ on the initial surface $\Sigma$ are smooth and compactly supported.
In Chapter 1 we mentioned large energy global existence of wave maps on the Minkowski background. In contrast to the picture
there, in the self-gravitating case we consider coupled Einstein-wave map system. Therefore, we need to construct the spacetime in the 
sense of \cite{Bruhat_Geroch_classic}. In the context of self-gravitating wave maps the equivalent geometric picture of global 
existence is geodesic completeness of the future development of the Einstein-wave map system. However, to prove geodesic completeness
of the maximal Cauchy development of \eqref{ewmcauchy_new}, a strategy similar to that of the Minkowski space seems promising
\footnote{especially for the cases of equivariant and spherical symmetry}. We formulate the two steps as follows.

\begin{description}
\item[(C'1) Non-concentration of energy ] The energy on a spacelike surface inside the past null cone of a point
goes to zero as one approaches the tip of the cone.

\item[(C'2) Geodesic completeness for small energy ] For arbitrarily small initial energy the maximal Cauchy development of \eqref{ewmcauchy_new}
is geodesically complete, hence inextendible.
\end{description}

\begin{rem}
 The Einstein wave map system of equations \eqref{ewmcauchy_new} are the Euler-Lagrange equations of a covariant action, 
 which can be interpreted as a symmetric hyperbolic system. Hence there exists a smooth globally hyperbolic maximal 
 development $ (M,g,U)$ of the initial data set.
\end{rem}
\subsection*{Equivariant Initial Data}
We define equivariance of initial data set as follows. Let $(\Sigma, q, \mathbf{K})$ and the 
target manifold $N$ be invariant under the action of the U(1) symmetry group. The initial 
data $(U_0, U_1)$ of the wave map $U \fdg M \to N$ is said to be equivariant if 
\begin{align*}
 U_0 \circ e^{i\theta} = e^{ik\theta} \circ U_0 \\
 U_1 \circ e^{i\theta} = e^{ik\theta}_{*} \circ U_1
\end{align*}
where $e^{ik\theta}_{*}$ is the pushforward of $e^{ik\theta} \fdg N \to N.$
\begin{pro}\label{maximal_equi}
 The maximal development $(M,g,U)$ of the equivariant initial data of the Einstein wave map system is 
 equivariant under the action of $U(1)$ symmetry group.
\end{pro}
\begin{proof}
The proof is based on the geometric uniqueness of the maximal development $(M,g,U).$ Following Rendall \cite{rendall_book}, let 
us define a Lie group $\cal{G}$ which acts on $\Sigma$ in such a way that for each $g_1 \in \cal{G}$ a transformation
$\Phi_{g_1} \fdg \Sigma \to \Sigma$ is a $U(1)$ symmetry of the initial data $(\Sigma, q, \mathbf{K})$. Furthermore,
let us define $\Phi_{U_0}$ and $\Phi_{g_1}$ as follows,
\begin{align}
 \Phi_{U_0} \fdg = \Phi_{g_1} \circ U_0 
\end{align}
and 
\begin{align}
 \Phi \fdg \Sigma &\to \Sigma \times N  \notag\\
p & \to (\Phi_{g_1}(p), \Phi_{U_0}(p))
\end{align}
so that the map $\Phi$ defines the equivariance of the initial data of the Einstein-equivariant wave map system. The
proof follows by the construction of unique extensions of $\Phi_{g_1}$ and $\Phi$ using the elements of the isometry group
of the maximal Cauchy development, identically as shown in p.176 in Rendall \cite{rendall_book}.

\end{proof}
\begin{pro}\label{laan_mots}
 The maximal development $(M,g,U)$ of the equivariant initial data  $(\Sigma,q,\mathbf{K},U_0,U_1)$ doesn't have trapped or marginally trapped surfaces.
\end{pro}
\begin{pro}\label{laan_fsing}
 The first (hypothetical) singularity occurs on the axis of the maximal development $(M,g,U).$
\end{pro}
The proofs of Propositions \ref{laan_mots} and \ref{laan_fsing} are due to Andersson\cite{laan_equinotes}.
In view of Proposition \ref{maximal_equi}, without loss of generality, let us assume that the metric $g$ in the polar coordinates $(t,r,\theta)$
is of the following form
\begin{align} \label{metric_polarform}
 d\,s^2_g = -e^{2\Omega(t,r)} d\, t^2 + e^{2\gamma(t,r)}d\,r^2 + r^2 d\,\theta^2 
\end{align}
for some scalar functions $\Omega(t,r)$ and $\gamma(t,r)$. Furthermore, in view of the critical dimension, without loss of generality
we can assume that the initial data surface is at $t=-1$ and $O$ is at $t=0$. The axis of $M$ is given by $r=0$. $\gamma(t,0)$ should
be $0$ to avoid a conical singularity at the axis and $\Omega(t,0)$ can be set to $0$ by a re-parameterization. Let us assume that
$N$ is a surface of revolution with a smooth, odd generating function $f$ such that 
\[ d\,s^2_h = d\,\rho^2+ f^2(\rho) d\,\phi^2 \] 
in $(\rho,\phi)$ coordinates and $f(0)=0,f_\rho(0)=1.$ The Cauchy problem of energy critical self-gravitating wave maps, with the equivariant
ansatz $U(t,r,\theta) = (u(t,r),k\theta)$ reduces\footnote{here we restrict to the case of $k=1$, the results in this work can be extended similarly to 
a general rotation number $k$ } to the following, \eqref{ewmcauchy_new} reduces to

\begin{equation}\label{ewmcauchy_equi_new}
\left. \begin{array}{rcl}
\mathbf{E}_{\mu \nu} &=&\mbo{\alpha} \mathbf{T}_{\mu \nu} \,\,\,\,\,\,\,\,\,\,\,\,\,\,\, \text{on}\,\, M^{2+1}\\
\square_g u &=& \frac{f_u(u)f(u)}{r^2} \,\,\,\,\,\,\,\, \text{on}\,\, M^{2+1}\\
\left.U\right|_{\Sigma}&=& U_0  \\
\left.\mathbf{X}(U)\right|_{\Sigma} & = & U_1 \end{array} 
\right\}
\end{equation}
with the equivariant initial data set $(\Sigma,q_0,\mathbf{K},U_0,U_1).$ 
 \section{Einstein Tensor}
In this section we shall explicitly calculate the components of the Einstein tensor 
\[ \mathbf{E}_{\mu \nu} \fdg = \mathbf{R}_{\mu \nu} -\halb R_g g_{\mu \nu} \]
in the polar coordinates $(t,r,\theta)$. \\
The Ricci tensor
\[\mathbf{R}_{\mu\nu}= \ptl_{\sigma}\leftexp{(g)}{\Gamma}^\sigma_{\mu \nu} - \partial_{\nu}
\leftexp{(g)}{\Gamma}^\sigma_{\sigma\mu} + \leftexp{(g)}{\Gamma}^\upsilon_{\upsilon \sigma} \leftexp{(g)}{\Gamma}^\sigma_{\nu\mu} 
- \leftexp{(g)}{\Gamma}^\upsilon_{\mu \sigma} \leftexp{(g)}{\Gamma}^\sigma_{\nu\upsilon} \]
is explicitly given by
\begin{align*}
\mathbf{R}_{tt} =& e^{2(\Omega- \gamma)}r^{-1} \left(\Omega_r(1-r\gamma_r  + r \Omega_r) + r \Omega_{rr}  \right) + \gamma_t(\Omega_t -\gamma_t) -\gamma_{tt},\\
\mathbf{R}_{rr} =& e^{2(\gamma-\Omega)} \left(\gamma^2_t -\gamma_t \Omega_t + \gamma_{tt} \right) -\Omega^2_r + \gamma_r(r^{-1}+ \Omega_r), \\
\mathbf{R}_{tr} =& r^{-1}\gamma_t, \\
\mathbf{R}_{\theta \theta} =& e^{-2\gamma} r (\gamma_r -\Omega_r), \\
\mathbf{R}_{t\theta} =&0, \\
\mathbf{R}_{r\theta} =&0, 
\end{align*}
and the scalar curvature
\begin{align*}
 R_g = 2e^{-2\gamma} \left(r^{-1}(\gamma_r -\Omega_r) + \gamma_r \Omega_r -\Omega^2_r -\Omega_{rr} \right) + 2e^{-2\Omega}\left( -\Omega^2_r + \gamma^2_r-\gamma_r
 \Omega_r + \gamma_{tt} \right).
\end{align*}
Therefore, the Einstein tensor is
\begin{align*}
 \mathbf{E}_{tt} &= e^{2(\Omega -\gamma)} \gamma_r r^{-1}, \\
\mathbf{E}_{tr} &= \gamma_t r^{-1}, \\
\mathbf{E}_{rr} &= \Omega_r r^{-1}, \\
\mathbf{E}_{\theta \theta} &= r^2 \bigl(e^{-2\gamma} (-\gamma_r\Omega_r + \Omega^2_r + \Omega_{rr}) -e^{-2\Omega} (\gamma^2_t -\gamma_t\Omega_t + \gamma_{tt}) \bigr), \\
\mathbf{E}_{t \theta} & = 0\,\,\text{and} \\
\mathbf{E}_{r \theta}&=0.
\end{align*}

\section{Energy Momentum Tensor}
The energy-momentum tensor $ \mathbf{T}$ for a wave map $U : (M,g) \to (N,h) $ is 
\begin{align}
 \mathbf{T}_{\mu \nu} \fdg = & \frac{\ptl\mathcal{L}_{\text{WM}}}{\ptl g^{\mu \nu}} - \halb g_{\mu \nu} \mathcal{L_\text{WM}} \notag \\
=& \langle \ptl_\mu U , \ptl_\nu U \rangle_{h(U)} -\halb g_{\mu \nu} \langle \ptl^\sigma U ,\ptl_\sigma U \rangle_{h(U)},
\end{align}
where $\mu$,$\nu$,$\sigma = 0,1,2$. 
In the following we will calculate each of the components of the energy momentum tensor in $(t,r,\theta)$ coordinates.
Note,
\begin{align}
 \langle \ptl^\sigma U ,\ptl_\sigma U \rangle_{h(U)} = -e^{-2\Omega}u_t^2 + e^{-2\gamma}u_r^2 + \frac{f^2(u)}{r^2}.
\end{align}
Now we proceed to calculate $ \mathbf{T}_{\mu \nu}$ 
\begin{align*}
 \mathbf{T}_{tt} &= h_{ij}\, \ptl_t U^i\, \ptl_t U^j - \halb g_{tt}\,\langle \ptl^\sigma U ,\ptl_\sigma U \rangle_h   \\
&= u_t^2 - \halb (-e^{2\Omega} )\left(-e^{-2\Omega}u_t^2 + e^{-2\gamma}u_r^2 + \frac{f^2(u)}{r^2} \right) \\
&= \halb e^{2\Omega}\left(e^{-2\Omega}u_t^2 +e^{-2\gamma}u_r^2 + \frac{f^2(u)}{r^2} \right) , \\
\mathbf{T}_{tr} &=  h_{ij}\, \ptl_t U^i\, \ptl_r U^j - 0 \\
&= u_t u_r, \\
\mathbf{T}_{rr} &= h_{ij}\, \ptl_r U^i \, \ptl_r U^j - \halb g_{rr}\, \langle \ptl^\sigma U ,\ptl_\sigma U \rangle_h  \\
&= u_r^2 -\halb (e^{2\gamma})\left(-e^{-2\Omega}u_t^2 + e^{-2\gamma}u_r^2 + \frac{f^2(u)}{r^2}\right) \\
&= \halb e^{2\gamma}\left(e^{-2\Omega}u_t^2 + e^{-2\gamma}u_r^2 - \frac{f^2(u)}{r^2}\right), \\
\mathbf{T}_{\theta \theta} &= h_{ij}\, \ptl_\theta U^i \, \ptl_\theta U^j - \halb g_{\theta \theta} \, \langle \ptl^\sigma U ,\ptl_\sigma U \rangle_h \\
&=f^2(u)- \halb r^2\left(-e^{-2\Omega} u_t^2 + e^{-2\gamma}u_r^2 + \frac{f^2(u)}{r^2}\right) \\
&=\halb r^2 \left(e^{-2\Omega} u_t^2 - e^{-2\gamma}u_r^2 + \frac{f^2(u)}{r^2}\right), \\
\mathbf{T}_{t \theta} &= 0\,\,\text{and}\\
\mathbf{T}_{r \theta}&=0.
\end{align*}
 
Let $\mathbf{X_1} \fdg = e^{-\Omega}\ptl_t$ be the future directed unit timelike normal and $\mathbf{X_2} \fdg = e^{-\gamma}\ptl_r $. 
We define the energy density $\mathbf{e} \fdg = \mathbf{T} (\mathbf{X_1},\mathbf{X_1}) $ and momentum density 
$\mathbf{m} \fdg = \mathbf{T}(\mathbf{X_1},\mathbf{X_2}).$ So
\begin{align*}
\mathbf{e}  & = \halb \left( e^{-2\Omega} \, u_t^2 + e^{-2\gamma} \, u_r^2 + \frac{f^2(u)}{r^2} \right) \\
&= \halb \left( (\mathbf{X_1}(u))^2 + (\mathbf{X_2}(u))^2 + \frac{f^2(u)}{r^2} \right) \\
\mathbf{m} & = e^{-(\Omega + \gamma)} u_t \, u_r  \\
 &= \mathbf{X_1}(u) \, \mathbf{X_2}(u) \\
\end{align*}
for the sake of brevity we further define $\mathbf{e_0} \fdg = \left( \mathbf{X_1}(u) \right)^2 + \left(\mathbf{X_2}(u)\right)^2 $ 
and $\mathbf{f} = \frac{f^2(u)}{r^2}.$
Let us also define the energy on a Cauchy surface $\Sigma_t$
\begin{align*}
E(U)(t)  \fdg =& \int_{\Sigma_t} \mathbf{e}\, \bar{\mu}_q  \\
=& 2\pi \int^{\infty}_{0} \mathbf{e}(t,r') r'e^{\gamma(t,r')} \, d\, r'\,, \\
\intertext{the energy in a coordinate ball $B_r$} 
 E(U)(t,r) \fdg =& \int_{B_r} \mathbf{e}\, \bar{\mu}_q\, , \\
 =& 2\pi \int^r_{0} \mathbf{e}(t,r') r'e^{\gamma(t,r')} \, d\, r' \\
\intertext{ the energy inside the causal past $J^-(O)$ of $O$} 
E^O(t) \fdg =& \int_{\Sigma_t \cap J^-(O)} \mathbf{e}\, \bar{\mu}_q \,.
\end{align*}

\section*{Einstein equivariant wave map system of equations}
The Einstein-equivariant wave map system of equations as in \eqref{ewmcauchy_equi_new} is redundant.
Here we collect the equations of the Einstein equivariant wave map system that we shall use and write them 
in the following form for more convenient usage later on.
\begin{subequations}\label{ewmequations}
\begin{align}
\gamma_r & = \halb \, r \, \mbo{\alpha} \, e^{2\gamma} \left( e^{-2\Omega} \, u_t^2 + e^{-2\gamma}\, u_r^2 + \frac{f^2(u)}{r^2} \right)\label{gamma_r} \\
\gamma_t & = r \, \mbo{ \alpha} \, u_t \,  u_r \label{gamma_t} \\
\Omega_r & = \halb \, r \, \mbo{\alpha} \, e^{2\gamma}   \left( e^{-2\Omega} \, u_t^2 + e^{-2\gamma} \, u_r^2 - \frac{f^2(u)}{r^2} \right)\label{omega_r} \\
\leftexp{3}{\square}_g u & = \frac{f_u(u)f(u)}{r^2} \label{wmequi}
 \end{align}
\end{subequations}
where,
 \[\leftexp{3}{\square}_g u = -e^{-2\Omega}(u_{tt} + (\gamma_t-\Omega_t)u_t) + e^{-2 \gamma}(u_{rr} + \frac{u_r}{r} + (\Omega_r - \gamma_r)u_r) \]
and $f_u (u)$ is the derivative of $f(u)$ with respect to $u$. \\
\begin{lem} \label{energy_cons}
 The energy $E(U)(t)$ is conserved
\footnote{
Even though we do not necessarily assume the existence of a timelike Killing vector, the fact that the energy is conserved is surprising yet
consistant with the works of Thorne \cite{thorne_cenergy} and Ashtekar-Varadarajan \cite{ash_var}} during the evolution of the Cauchy 
problem \eqref{ewmcauchy_equi_new} .
\end{lem}
\begin{proof}
 Consider two Cauchy surfaces $\Sigma_{s}$ and $\Sigma_{\tau}$ at $t=s$ and $t=\tau$ respectively, with $-1\leq \tau \leq s < 0$.
The compactly supported initial data ensures that each $\Sigma_t$ is asymptotically flat and each component of $\mathbf{T}_{\mu \nu} \to 0$
as $r \to \infty$. We shall now construct a divergence free vector field $\mathbf{P}_{\mathbf{X_1}}$ as follows
\footnote{ Later, in Section \ref{vfm} we shall illustrate a more general procedure of the construction of such ``momentum'' vector fields as
part of the vector fields method. Thanks are due to Vincent Moncrief for pointing out a simpler construction used here}.
Consider the Einstein's equations \eqref{gamma_r} and \eqref{gamma_t}. They can be rewritten as follows
\begin{align*}
 -\ptl_r\left(e^{-\gamma} \right) &=  r\,\mbo{\alpha}  e^{\gamma} \mathbf{e} \\
 -\ptl_t \left(e^{-\gamma} \right) &=  r\,\mbo{\alpha}  e^{\Omega} \mathbf{m}.
\end{align*}
From the smoothness of $\gamma$ we have $ -\ptl^2_{rt}\left(e^{-\gamma} \right)= -\ptl^2_{tr}\left(e^{-\gamma} \right)$, which implies
\begin{align}\label{commute_partials}
 - \ptl_t\left(re^{\gamma} \mathbf{e}\right) + \ptl_r(re^{\Omega}\mathbf{m}) =0.
\end{align}
Now define a vector 
\[ \mathbf{P}_{\mathbf{X_1}} \fdg= -e^{-\Omega}\,\mathbf{e}\, \ptl_t + e^{-\gamma}\, \mathbf{m}\, \ptl_r, \]
then the divergence of $\mathbf{P}_{\mathbf{X_1}}$ is given by
\begin{align*}
\nabla_\nu  \mathbf{P}_\mathbf{X_1} ^\nu &= \frac{1}{\sqrt{|g|}}\, \ptl_\nu \left(\sqrt{|g|}\,\mathbf{P}_\mathbf{X_1} ^\nu \right) \\
&= \frac{1}{re^{\gamma + \Omega}} \left(-\ptl_t \left( re^{\gamma}\,\mathbf{e}\,+ \ptl_r \left( re^{\Omega}\,\mathbf{m}\,\right)\right) \right) \\
&= 0
\end{align*}
from \eqref{commute_partials}.
Now let us apply the Stokes' theorem in the region whose boundary is $\Sigma_{s} \cup \Sigma_{\tau}$, then we have
\begin{align}
0= \int_{\Sigma_s} e^{\Omega} \mathbf{P}^t_\mathbf{X_1} \bar{\mu}_q-
\int_{\Sigma_{\tau}} e^{\Omega} \mathbf{P}^t_\mathbf{X_1} \bar{\mu}_q.
\end{align}
Therefore, it follows that 
\begin{align}
 E(U)(\tau)= E (U)(s)
\end{align}
for any $\tau$, $s$ such that  $-1\leq \tau \leq s < 0$.
\end{proof}
In the following lemma we shall prove that the metric functions $\gamma(t,r)$ and $\Omega(t,r)$ are uniformly bounded during the evolution of 
the Einstein-wave map system. This is also discussed in \cite{laan_equinotes}.
\begin {lem} \label{metric_uniform}
There exist constants $c^-_\gamma,c^+_\gamma,c^-_\Omega,c^+_\Omega$  such that the following uniform bounds 
\[c^-_\gamma \leq \gamma(t,r) \leq c^+_\gamma\]
\[c^-_\Omega \leq \Omega(t,r) \leq c^+_\Omega\]
on the metric functions $\gamma(t,r)$ and $\Omega(t,r)$ hold.
 \end {lem}
\begin{proof}
For simplicity of notation, we use a generic constant $c$ for the estimates on $\gamma(t,r)$ and $\Omega(t,r)$.
The Einstein's equation \eqref{gamma_r} for $\gamma_r$ can be rewritten as 
\[ - (e^{-\gamma} )_r  = \mbo{\alpha}\, r \,e^\gamma \mathbf{e}  \]
and integrating with respect to $r$, we get
\[
 1-e^{-\gamma} =  \mbo{\alpha} \int^r_0 \mathbf{e}\, r\, e^{\gamma} d\,r 
 = \frac{\mbo{\alpha}}{2\pi}  E(U)(t,r)
\]
so,
\[e^{\gamma} = \left( 1- \frac{\mbo{\alpha}}{2\pi}  E(U)(t,r)\right)^{-1}. \]
Let us define $\gamma_{\infty}(t) \fdg = \lim_{r \to \infty } \gamma (r,t),$
then we have
\[ e^{\gamma_\infty (t)} = \left( 1- \frac{\mbo{\alpha}}{2\pi}  E(U)(t)\right)^{-1}. \]
The energy is conserved $ E(U)(t) = E(U)(-1)$ so
$\gamma_\infty (t) = \gamma_\infty (-1)$ is also conserved during the evolution of the Einstein
wave map system. 
\\
In addition $E(U)(t,r)$ is a nondecreasing function of $r$, then so is $\gamma(t,r)$
\[1 = e^{\gamma(t,0)} \leq e^{\gamma(t,r)} \leq e^{\gamma_\infty(t)} = c(E_0). \]
Similarly let us consider the Einstein's equation \eqref{omega_r} for $\Omega_r$
\[ \Omega_r  = r \, \mbo{\alpha} \, e^{2\gamma} (\mathbf{e}-\mathbf{f})\]
and integrating with respect to $r$ we get
\begin{align*}
 \Omega(t,r) -\Omega(t,0) & \leq c(E_0) \int^r_0 (\mathbf{e}-\mathbf{f}) re^{\gamma} d\, r \\
& \leq c(E_0) \int^r_0 \mathbf{e}\,r\,e^{\gamma} d\, r \\
& \leq c(E_0)
\end{align*}
and
\begin{align*}
\Omega(t,r) & \geq - c(E_0) \int^r_0 \frac{\mathbf{f}}{2} r e^{\gamma} \, d\, r \\
& \geq - c(E_0) \int^r_0 \mathbf{e}\, r\, e^{\gamma} \, d\,r \\
& \geq - c(E_0)
\end{align*}
\end{proof}

\begin{lem}
 Assume that the target manifold $(N,h)$ satisfies
 \begin{align}\label{nosphereinfinity}
   \wp := \int^u_0 f(s)\,ds \to \infty \,\,\, \text{as}\,\,\, u \to \infty,
 \end{align}
 then there exists a constant $c$ dependent on initial
energy $E_0$ such that
\[ u \in L^{\infty} \text{with } ||u||_{\infty} \leq c(E_0) \] for every solution
$u$ of the equivariant wave map equation.
 \end{lem}
\begin{proof}
Extending the technique used in Lemma 8.1 in \cite{shatah_struwe}, we consider
\begin{align*}
\wp(u(t,r)) & = \int_0^r \ptl_r (\wp(u(t,r))) \, dr  \\
& = \int_0 ^r f(u) \ptl_r u \,dr \\
& = \int_0 ^r \left( f(u) (re^{-\gamma})^{-1/2} \right) \left( \ptl_r u (re^{-\gamma})^{1/2} \right) \, dr. \\
\intertext{Consequently,}
|\wp(u(t,r))| & \leq \left( \int^r_0 (f(u))^2(re^{-\gamma})^{-1}\, dr\right)^{1/2} \left(\int^r_0 (\ptl_r u)^2 re^{-\gamma}\, dr \right)^{1/2} \\
& \leq \left( \int^\infty_0 (f(u))^2(re^{-\gamma})^{-1}\, dr\right)^{1/2} \left(\int^\infty_0 (\ptl_r u)^2 re^{-\gamma} \, dr\right)^{1/2} \\
& \leq c(E_0). 
 \end{align*}
Arguing via contradiction, the result follows.
\end{proof}

\section{Vector Fields Method, Monotonicity of Energy} \label{vfm}
Let $\mathbf{X}= F(t,r)\ptl_t + G(t,r) \ptl_r$ be a vector field in the spacetime and the corresponding momentum $\mathbf{P}_\mathbf{X}$ is
given by the contraction of $\mathbf{T}$ with $\mathbf{X}$ i.e.,
\[ \mathbf{P}_{\mathbf{X}} = \mathbf{T}(\mathbf{X})\]
in coordinates,
\begin{align}\label{mom_def}
\mathbf{P}_\mathbf{X}^\mu = \mathbf{T}^\mu_{\,\,\,\nu} \mathbf{X}^{\nu}.
\end{align}
Henceforth we refer to the vector $\mathbf{X}$ as a multiplier due to \eqref{mom_def}.
The momentum $\mathbf{P}_\mathbf{X} ^\mu$ in $(t,r,\theta)$ coordinates is then
\begin{align*}
 \mathbf{P}_\mathbf{X} ^\mu &=  \mathbf{T}^\mu_{\nu} \mathbf{X}^\nu \\
& = \mathbf{T}^\mu_t \, F + \mathbf{T}^\mu_r \, G \\
\mathbf{P}_\mathbf{X}^t & = \mathbf{T} ^t_t \, F +  \mathbf{T}^t_r \, G  \\
&= -(\mathbf{e} \, F + e^{(\gamma-\Omega)} \mathbf{m} \, G ) \\
\mathbf{P}_\mathbf{X}^r &= \mathbf{T} ^r_t \, F +  \mathbf{T}^r_r \, G  \\
&= ( e^{(\Omega - \gamma)}\mathbf{m} \, F + (\mathbf {e} -\mathbf{f})\, G) \\
\mathbf{P}_\mathbf{X}^\theta &= \mathbf{T}^\theta _{\,\,\,\nu}\, \mathbf{X}^{\nu} \\
&= 0.
\end{align*}
So the momentum vector $ \mathbf{P}_\mathbf{X}$
\begin{align*}
 \mathbf{P}_\mathbf{X} =& -(e^{\Omega}\mathbf{e}\,F + e^{\gamma}\mathbf{m}\,G)\mathbf{X_1} + (e^{\Omega}\,\mathbf{m}\,F + e^\gamma(\mathbf{e}-\mathbf{f})G)
\mathbf{X_2} \\
 =&-(\mathbf{e} \, F + e^{(\gamma-\Omega)} \mathbf{m} \, G ) \ptl_t 
+ ( e^{(\Omega - \gamma)}\mathbf{m} \, F + (\mathbf {e} -\mathbf{f})\, G) \ptl_r.
\end{align*}
In the following we shall calculate the covariant divergence of the momentum vector $\mathbf{P}_\mathbf{X}$. We have,
\begin{align}
 \grad_\nu \mathbf{P}_\mathbf{X}^\nu =& \grad_\nu (\mathbf{T}^\nu_{\,\,\mu}\,\mathbf{X}^\mu) \notag\\
\intertext{using the Leibnitz rule,} \notag \\
\grad_\nu \mathbf{P}_\mathbf{X}^\nu =& \,\mathbf{X}^\mu\, \grad_\nu (\mathbf{T}^\nu_{\,\,\mu}) + \mathbf{T}^\nu_{\,\,\mu}\,\grad_\nu (\mathbf{X}^\mu) \label{se_divfree}\\
\intertext{since the stress energy tensor $\mathbf{T}$ is divergence free, the first term in the right hand side of $\eqref{se_divfree}$ drops out, therefore} \notag \\
\grad_\nu \mathbf{P}_\mathbf{X}^\nu =& \mathbf{T}^{\mu\nu} \, \grad_{\mu}\mathbf{X}_\nu \notag \\
=& \halb \, \leftexp{(\mathbf{X})} {\mbo{\pi}}_{\mu \nu} \mathbf{T}^{\mu \nu},  \notag
\end{align}
where the so called deformation tensor $\leftexp{(\mathbf{X})} {\mbo{\pi}}_{\mu \nu}$ is given by
\begin{align*}
\leftexp{(\mathbf{X})} {\mbo{\pi}}_{\mu \nu} \fdg=& \nabla_\mu \mathbf{X}_\nu + \nabla_\nu \mathbf{X}_\mu \\
=& g_{\sigma \nu}\ptl_\mu \mathbf{X}^\sigma + g_{\sigma \mu}\ptl_\nu \mathbf{X}^\sigma + \mathbf{X}^\sigma \ptl_\sigma g_{\mu \nu}.
\end{align*}
For the form of the metric chosen in \eqref{metric_polarform}, we get different components of 
$ \leftexp{(\mathbf{X})} {\mbo{\pi}}_{\mu \nu} $
to be the following
\begin{align*}
\leftexp{(\mathbf{X})} {\mbo{\pi}}_{\mu \nu} =  \left( \begin{array}{ccc}
-2 e^{2\Omega}\left( F_t + F \Omega_t + G \Omega_r \right) & e^{2\gamma}G_t-e^{2\Omega}F_r & 0 \\
e^{2\gamma}G_t-e^{2\Omega}F_r & 2e^{2\gamma} \left(G_r + G \gamma_r + F \gamma_t \right) & 0 \\
0 & 0 & 2rG \end{array} \right).
\end{align*}
The divergence of $\mathbf{P_X}$ then is 
\begin{align}\label{general_div}
 \nabla_\nu  \mathbf{P}_{\mathbf{X}} ^\nu &= \halb \leftexp{(\mathbf{X})} {\mbo{\pi}}_{\mu \nu} \mathbf{T}^{\mu \nu} \notag \\
&= \halb \left(\leftexp{(\mathbf{X})} {\mbo{\pi}}_{tt} \mathbf{T}^{tt} + 2(\leftexp{(\mathbf{X})} {\mbo{\pi}}_{tr} \mathbf{T}^{tr})
 + \leftexp{(\mbo{X})} {\mathbf{\pi}}_{rr} \mathbf{T}^{rr} + \leftexp{(\mathbf{X})} {\mbo{\pi}}_{\theta \theta} \mathbf{T}^{\theta \theta} \right) \notag \\
& =  -e^{-2\Omega} \left( F_t + F \Omega_t + G \Omega_r \right) \mathbf{T}_{tt} + \left( F_r \, e^{-2\gamma} - G_t e^{-2\Omega}\right) \mathbf{T}_{tr} \notag \\
& \quad  + e^{-2\gamma} \left( G_r \,+ G \gamma_r + G \gamma_t \right) \mathbf{T}_{rr} + r^{-3}G \, \mathbf{T}_{\theta \theta} \notag \\
& = \halb e^{-2\Omega}\left( F(-\Omega_t + \gamma_t) + G(-\Omega_r + \gamma_r + r^{-1}) + G_r - F_t \right) u_t^2 \notag \\ 
& \quad + \halb e^{-2\gamma}\left( F(-\Omega_t + \gamma_t)+ G(-\Omega_r + \gamma_r - r^{-1}) + G_r - F_t \right) u_r^2 \notag \\
&\quad + \halb \left( F(-\Omega_t - \gamma_t)+ G(-\Omega_r - \gamma_r + r^{-1}) -G_r - F_t \right) \frac{f^2(u)}{r^2} \notag \\
& \quad + \left(F_r e^{-2 \gamma} - G_t \,e^{-2 \Omega}\right) u_t u_r
\end{align}
As mentioned is Section \ref{ovcurrent} in Introduction, construction of relevant identities using \eqref{general_div} and the
Stokes' theorem is central to our method to prove non-concentration of energy of equivariant self-gravitating wave maps.
In the following let us calculate the divergence of $\mathbf{P_X}$ for various choices of $\mathbf{X}$'s.
Consider  $\mathbf{X_1} = e^{-\Omega} \ptl_t$, the corresponding momentum $\mathbf{P}_\mathbf{X_1}$ is
\begin{align}
 \mathbf{P}_\mathbf{X_1} =& -\mathbf{e}\,\mathbf{X_1} + \mathbf{m}\,\mathbf{X_2} \notag \\
=& - e^{-\Omega}\mathbf{e}\, \ptl_t + e^{-\gamma}\mathbf{m}\,\ptl_r 
\end{align}
then we have,
\begin{align}\label{div_pxi}
 \nabla_\nu  \mathbf{P}_\mathbf{X_1} ^\nu & = \halb e^{-2\Omega} \left( e^{\Omega}\gamma_t \right) u_t^2  + 
 \halb e^{-2\gamma} \left( e^{\Omega}\gamma_t \right) u_r^2  \notag \\ 
& \quad - \halb \left( e^{\Omega}\gamma_t \right) \frac{f^2(u)}{r^2} - \Omega_r e^{-\Omega-2\gamma} u_t u_r \notag \\
&= e^{-\Omega} \left(\gamma_t (\mathbf{e}-\mathbf{f})-\Omega_r e^{-\gamma} \mathbf{m} \right) \notag \\
& = 0
\end{align}
after the usage of Einstein's equations \eqref{gamma_t} and \eqref {omega_r}. \\
Equivalently,
\begin{align}\label{equiv_div_px1}
0=\grad_\nu \mathbf{P}_\mathbf{X_1}^\nu &= \frac{1}{\sqrt {-g}} \ptl_\nu (\sqrt{-g}\, \mathbf{P}_\mathbf{X_1}^\nu) \notag \\
 & = \frac{1}{re^{\gamma + \Omega}} \left(-\ptl_t (re^{\gamma}\mathbf{e}) 
+ \ptl_r (re^{\Omega}\mathbf{m}) \right).
\end{align}
For $\mathbf{X_2} = e^{-\gamma}\ptl_r$ and 
\begin{align}\label{intro_px2}
\mathbf{P}_\mathbf{X_2} =& -\mathbf{m}\mathbf{X_1} + (\mathbf{e}-\mathbf{f}) \mathbf{X_2} \notag \\
 =&- e^{-\Omega} \mathbf{m} \, \ptl_t + e^{-\gamma} (\mathbf{e}-\mathbf{f}) \,\ptl_r,
\end{align}
the divergence $\grad_\nu \mathbf{P}_\mathbf{X_2}^\nu $ using \eqref{general_div} is
\begin{align}\label{div_px2}
\grad_\nu \mathbf{P}_\mathbf{X_2}^\nu & =  \halb \, \leftexp{(\mathbf{X_2})} {\mbo{\pi}}_{\alpha \beta} \mathbf{T}^{\alpha \beta} \notag\\
& = - e^{-\gamma} \Omega_r \,\mathbf{e} + \frac{1}{2r} e^{-\gamma} (e^{-2\Omega} u_t^2 - e^{-2\gamma}u_r^2 + \mathbf{f}) 
+ e^{-\Omega} \gamma_t \mathbf{m}.
\end{align}
Equivalently,
\begin{align}\label{equiv_div_px2}
 \grad_\nu \mathbf{P}_\mathbf{X_2}^\nu &= \frac{1}{\sqrt {-g}} \ptl_\nu (\sqrt{-g}\, \mathbf{P}_\mathbf{X_2}^\nu) \notag\\
 & = \frac{1}{re^{\gamma + \Omega}} \left(-\ptl_t (re^{\gamma}\mathbf{m}) 
+ \ptl_r ((\mathbf{e}-\mathbf{f}) re^{\Omega}) \right) .
\end{align}
Similarly for the choices of $\mathbf{X_3} \fdg = re^{-k\gamma}\ptl_r$ and $\mathbf{X_4} \fdg=r^a\,\ptl_r$, $a \in (\halb,1)$
we have
\begin{align}\label{intro_px3}
\mathbf{P_{X_3}} =& e^{(1-k)} \left( -r\,\mathbf{m}\,\mathbf{X_1} + r (\mathbf{e}-\mathbf{f}) \mathbf{X_2}\right) \notag\\
=& -re^{(1-k)\gamma - \Omega} \,\mathbf{m}\, \ptl_t + re^{-k\gamma}(\mathbf{e} -\mathbf{f})\ptl_r, \\
\notag\\
\nabla_\nu  \mathbf{P}^\nu_{\mathbf{X_3}} =&\, \halb \leftexp{(\mathbf{X_3})} {\mbo{\pi}}_{\alpha \beta} \mathbf{T}^{\alpha \beta} \notag\\
= &\,  e^{-k\gamma} e^{- 2 \Omega} u_t^2 - r e^{-k\gamma} \Omega_r \, \mathbf{e} + r e^{-k \gamma} (1-k) \gamma_r (\mathbf{e}-\mathbf{f}) 
 + r \, k \,\gamma_t \, e^{ (1-k)\gamma -\Omega} \, \mathbf{m} \notag \\
= &\, e^{-k\gamma} e^{- 2 \Omega} u_t^2 - k\mbo{\alpha} r^2 e^{(2-k)\gamma} ( \mathbf{e} (\mathbf{e}-\mathbf{f}) - \mathbf{m}^2) \notag\\
= & \,e^{- 2 \Omega} u_t^2
\end{align}
for the choice of $k =0,$ and 

\begin{align}\label{intro_px4}
 \mathbf{P}_\mathbf{X_4}  =& -e^{\gamma -\Omega} r^a \, \mathbf{m} \ptl_t + r^a \,(\mathbf{e}-\mathbf{f}) \ptl_r \notag \\
 =& e^{\gamma} r^a (-\mathbf{m} \mathbf{X_1} + (\mathbf{e}-\mathbf{f}) \mathbf{X_2}) \\
 \notag\\
\nabla_\nu  \mathbf{P}^\nu_\mathbf{X_4} =&\, \halb \left ( r^a (-\Omega_r + \gamma_r) + (1+a) r^{a-1}  \right) e^{-2\Omega} u^2_t \notag\\
&\quad+ \halb \left ( r^a (-\Omega_r + \gamma_r) + (a-1) r^{a-1}  \right) e^{-2\gamma} u^2_r \notag\\
& \quad + \halb \left ( - r^a (\Omega_r + \gamma_r) + (1-a) r^{a-1}  \right) \frac{f^2(u)}{r^2}\notag \\
=&\, \halb \left (  (1+a) r^{a-1}  \right) e^{-2\Omega} u^2_t
+ \halb \left ( (a-1) r^{a-1}  \right) e^{-2\gamma} u^2_r \notag\\
& \quad + \halb \left ( (1-a) r^{a-1}  \right) \frac{f^2(u)}{r^2}
 \end{align}
where we used Einstein's equations \eqref{omega_r} and \eqref{gamma_r} for $ \Omega_r$ and $\gamma_r$ respectively.
Let $J^-(O)$ be the causal past of the the point $O$ and $ I^- (O)$ the chronological past of $O$.
We will need the following definitions 
\begin{align*}
 \Sigma^O_t \fdg &= \Sigma_t \cap J^- (O) \\
K(t) \fdg &= \cup_{ t \leq t' < 0}\, \Sigma_{t'} \cap J^-(O) \\
C(t) \fdg &= \cup_{t \leq t' < 0} \, \Sigma_{t'} \cap (J^-(O) \setminus I^-(O)) \\
K(t,s) \fdg &= \cup_{ t \leq t' < s}\, \Sigma_{t'} \cap J^-(O) \\
C(t,s) \fdg &= \cup_{t \leq t' < s} \, \Sigma_{t'} \cap (J^-(O) \setminus I^-(O))
\end{align*}
for $-1\leq t < s <0.$
In the following we will try to understand the behaviour of various quantities of the wave map as one approaches this point in a limiting sense.
For this we will use the Stokes' theorem in the region $K(\tau,s),-1\leq \tau \leq s <0$ (as shown in the figure \ref{fig:Stokes_picture}) for divergence of 
vector fields $\mathbf{P}_\mathbf{X}$ for apt choices of the vector field $\mathbf{X}$. 
\begin{figure}[!hbt]
\psfrag{O}{$O$}
\psfrag{Ktaus}{$K(\tau,s)$}
\psfrag{teq0}{$t=0$}
\psfrag{teqs}{$t=s$}
\psfrag{teqtau}{$t=\tau$}
\psfrag{teq-1}{$t=-1$}
\psfrag{Fluxpx}{$\text{Flux}(\mathbf{P_X})(t,s)$}
\psfrag{E(s)}{$E^O(s)$}
\psfrag{E(tau)}{$E^O(\tau)$}
\centerline{\includegraphics[height=2.5in]{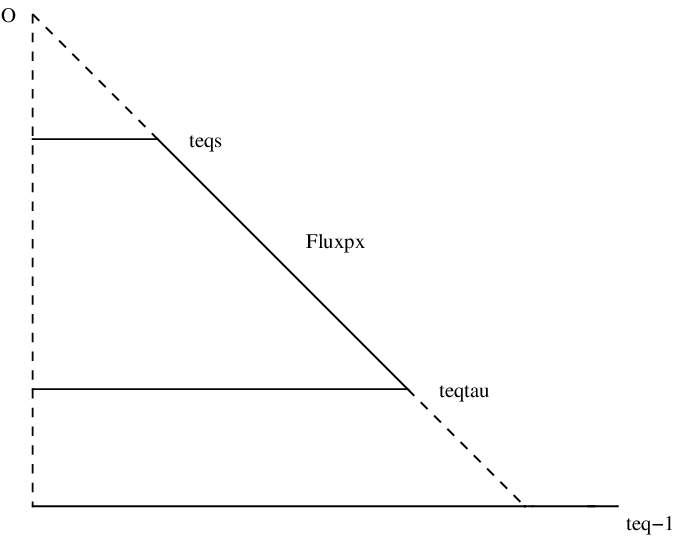}}
\caption{Application of the Stokes' theorem for the divergence of $\mathbf{P_X}$ } 
\label{fig:Stokes_picture}
\end{figure}
Let the volume 3-form of $(M,g)$ be 
\begin{align*}
\bar{\mu}_g \fdg = re^{\gamma + \Omega} d\,t \wedge d\,r \wedge d\,\theta
\end {align*}
and the area 2-form of $(\Sigma,q)$ be
\begin{align*}
 \bar{\mu}_q = r e^{\gamma} d\,r \wedge d\,\theta.
\end{align*}
Let us define 1-forms $\tild{\ell}$, $\tild{n}$ and $\tild{m}$ as follows
\begin{align*}
 \tild{\ell} \fdg =& -e^\Omega d\,t + e^{\gamma} d\,r \\
 \tild{n}\fdg =& -e^\Omega d\,t - e^{\gamma} d\,r \\
 \tild{m} \fdg =&\,rd\,\theta
\end{align*}
so we have,
\begin{align*}
 \bar{\mu}_g= \halb \left( \tild{\ell} \wedge \tild{n} \wedge \tild{m} \right).
\end{align*}
Let us also define\footnote{we chose these definitions for consistency
in orientation, which allows us to compare the relative signs of the terms in the estimates that follow}
the 2-forms $ \bar{\mu}_{\tild{\ell}}$ and $ \bar{\mu}_{\tild{n}}$ such that
\begin{align*}
 \bar{\mu}_{\tild{\ell}} \fdg=&\,-\halb \tild{n} \wedge \tild{m} \\
\bar{\mu}_{\tild{n}} \fdg =&\, \halb \tild{\ell} \wedge \tild{m}
\end{align*}
so that
\begin{align*}
 \bar{\mu}_g =\, & -\tild{\ell} \wedge \bar{\mu}_{\tild{\ell}} \,\,\,\,\text{and} \\
\bar{\mu}_g =\, & -\tild{n} \wedge \bar{\mu}_{\tild{n}}\,.
\end{align*}
We now apply the Stokes' theorem for the $\bar{\mu}_g$-divergence of $\mathbf{P_X}$ in the region $K(\tau,s)$ to get 

\begin{align}\label{stokes_gen} 
\int_{K(\tau,s)} \nabla_\nu  \mathbf{P}^\nu_{\mathbf{X}} \, \bar{\mu}_g  = \int_{\Sigma^O_{s}} e^{\Omega}\, \mathbf{P}^t_{\mathbf{X}}\, \bar{\mu}_q
- \int_{\Sigma^O_{\tau}}e^{\Omega}\,\mathbf{P}^t_{\mathbf{X}}\, \bar{\mu}_q
+ \text{Flux}(\mathbf{P_X}) (\tau,s) 
\end{align}so
where\footnote{note that, by definition, $\tild{\ell}_\mu \tild{\ell}^\mu= \tild{n}_\mu \tild{n}^\mu=0 $},
\begin{align*}
 \text{Flux} (\mathbf{P_X})(\tau,s) = - \int_{C(\tau,s)} \tild{n}(\mathbf{P_X})\, \bar{\mu}_{\tild{n}} .
 \end{align*}
\begin{lem}\label{mon}
 $ E^O(\tau) \geq E^O(s)$ for $-1 \leq \tau < s <0 $
\end{lem}
\begin{proof}
 
 Let us apply the Stokes' theorem \eqref{stokes_gen} to the vector field $\mathbf{P}_\mathbf{X_1}$.  We have
\begin{align} \label{stokes_mono}
 0  = -\int_{\Sigma^O_{s}} \mathbf{e}\,\,\bar{\mu}_q
+ \int_{\Sigma^O_{\tau}} \mathbf{e} \,\, \bar{\mu}_q
+ \text{Flux} (\mathbf{P}_\mathbf{X_1})(\tau,s) 
\end{align}
and
\begin{align*}
\text{Flux} (\mathbf{P}_\mathbf{X_1})(\tau,s) =&- \int_{C(\tau,s)} \tild{n}\,(\mathbf{P}_{\mathbf{X}_1})\,  \bar{\mu}_{\tild{n}}  \\
=&- \int_{C(\tau,s)} (\mathbf{e-m})\,  \bar{\mu}_{\tild{n}}\,.  
\end{align*}

\begin{figure}[!hbt]
\psfrag{O}{$O$}
\psfrag{Ktaus}{$K(\tau,s)$}
\psfrag{teq0}{$t=0$}
\psfrag{teqs}{$t=s$}
\psfrag{teqtau}{$t=\tau$}
\psfrag{teq-1}{$t=-1$}
\psfrag{Fluxpx}{$\text{Flux}(\mathbf{P_X})(t,s)$}
\psfrag{E(s)}{$E^O(s)$}
\psfrag{E(tau)}{$E^O(\tau)$}
\centerline{\includegraphics[height=2.5in]{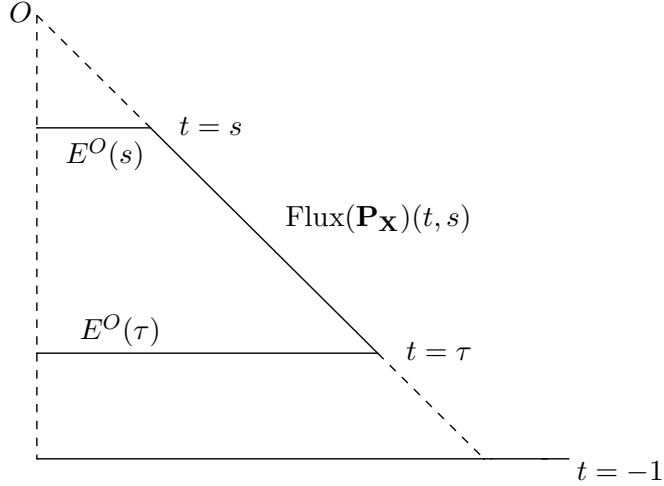}}
\caption{Monotonicity of Energy inside the past null cone of $O$ } 
\label{fig:Stokes_picture1} 
\end{figure}

By the Cauchy-Schwarz inequality we have $\mathbf{e} \geq |\mathbf{m}|.$
So we have $\text{Flux} (\mathbf{P}_\mathbf{X_1})(\tau,s) \leq 0$,
this implies  
\[ E^O(\tau) - E^O(s) \geq 0 \,\,\forall \,\, -1\leq \tau \leq s <0.\]
\end{proof}
However, one may note that although $E^O(\tau) \geq 0$, $0$ need not be the infimum of the sequence $E^O(\tau_n)$, $\{ \tau_n\}_n \in [-1,0), \tau_n \to
0$ as $n \to \infty.$ Hence, in principle there is a possibility of the energy blowing up at $O$ i.e $E^O(\tau) \not\to 0$
as $\tau \to 0.$
In the follow-up work of this paper we will try to understand the blow up criteria for self-gravitating wave maps
by rescaling (bubbling) the wave map in the neighbourhood of $O$. Just as on the Minkowski background, it is expected that a certain
minimum non-zero energy is needed for the blow up to happen at $O$ if the target manifold $N$ is positively curved, for instance a sphere
$\mathbb{S}^2$.\\
Let us define
\begin{align} \label{E_conc}
 E^O_{\text{conc}} \fdg =  \inf_{\Sigma^O_\tau} E^O (\tau)\,\, \text{for}\,\, \tau \in [-1,0]
\end{align}
As a consequence of Lemma \ref{mon}, \eqref{E_conc} is equivalent to 
\begin{align}
 E^O_{\text{conc}} = \lim_{\tau \to 0} E^O(\tau).
\end{align}
We say that the equivariant Cauchy problem \eqref{ewmcauchy_new} blows up\footnote{we use the phrases ``blow up'', 
``energy concentration'' and ``singularity'' synonymously} if  $ E^O_{\text{conc}} \neq 0$ and doesn't concentrate if
 $E^O_{\text{conc}} = 0.$
\begin{cor} \label{flux_px1_zero}
\begin{align*}
  \text{\textnormal{Flux}}(\mathbf{P}_\mathbf{X_1})(\tau) \to 0 \,\,\text{as}\,\, \tau \to 0.
\end{align*}
\end{cor}
\begin{proof}
Consider the equation \eqref{stokes_mono} for $s \to 0$, we have
\begin{align} \label{flux_to_zero}
 0  = -E^O_{\text{conc}}
+ \int_{\Sigma^O_{\tau}} \mathbf{e} \,\,\bar{\mu}_q
+ \text{Flux} (\mathbf{P}_\mathbf{X_1})(\tau)
\end{align}
where
\begin{align*}
 \text{Flux}(\mathbf{P_X})(\tau) \fdg = \lim_{s\to 0} \text{Flux}(\mathbf{P_X})(\tau,s)
\end{align*}
Now by the definition \eqref{E_conc}, as $\tau \to 0$ we get
\begin{align}
\lim_{\tau \to 0} \int_{\Sigma^O_{\tau}} \mathbf{e} \,\,\bar{\mu}_q \to E^O_{\text{conc}} \label{}
\end{align}
Therefore, it follows from \eqref{flux_to_zero} that $\text{\textnormal{Flux}}(\mathbf{P}_\mathbf{X_1})(\tau)\to 0$ as $\tau \to 0.$
\end{proof}
\section{Coordinate Null Basis Vectors}
Previously, we defined the 1-forms $\tild{\ell}$ and $\tild{n}$. Their corresponding vectors are null, given by
\begin{align*}
 \tild{\ell}   =&\, e^{-\Omega}\ptl_t + e^{-\gamma}\ptl_r \\
\tild{n} =&\, e^{-\Omega} \ptl_t - e^{\gamma}\ptl_r \,.
\end{align*}
Let us consider their Lie bracket $[ \tild{\ell}, \tild{n}]$,
since 
\begin{align*}
[ \tild{\ell}, \tild{n}] &\equiv 2 \left( e^{-\gamma} \Omega_r\, \mathbf{X_1} - e^{-\Omega}\gamma_t \,\mathbf{X_2} \right) \\
&\equiv 2 e^{- (\gamma + \Omega)} (\Omega_r \ptl_t -\gamma_t \ptl_r) \\
& \not\equiv 0 
\end{align*}
 $\tild{\ell}$ and $\tild{n}$ cannot necessarily be part of a coordinate basis. We shall try to normalize  $\tild{\ell}$ and $\tild{n}$ 
to get coordinate basis vectors. So let us introduce a coordinate null triad $\ell$, $n$ and $m$
\begin{align}\label{coonulltriad}
\ell \fdg = e^\mathcal{F} \tild{\ell},\,n \fdg = e^\mathcal{G} \tild{n}\,\,\, \text{and}\,\,\, m \fdg= \frac{1}{r}\ptl_\theta, 
\end{align}
where the scalar functions $\mathcal{F}$ and $\mathcal{G}$ such that 
$[\ell, n ] \equiv 0$. Furthermore, $\mathcal{F}$ and $\mathcal{G}$ can be set to $0$ on the axis.
Now consider $[ \ell , n ],$
\begin{align*}
 [ \ell , n ] & = e^{(\cal{F} + \cal {G})} \left( [\tild{\ell},\tild{n}] + \tild{\ell}(\cal{G}) \tild{n} - \tild{n}(\cal{F})\tild{\ell} \right) \\
& = e^{(\cal{F} + \cal {G})} \left \{ \left(2 e^{- (\gamma + \Omega)} \Omega_r + e^{-\Omega} \left(\ell(\cal{G})-n (\cal{F}) \right)\right) \ptl_t 
- \left( 2 e^{- (\gamma + \Omega)}\gamma_t + e^{-\gamma} \left( \ell(\cal{G}) +n (\cal{F}) \right) \right) \ptl_r  \right \} .
\end{align*}
So  $ [ \ell , n ] \equiv 0 \iff$ $\cal{F}$ and $\cal{G}$ are such that 
\begin{align*}
 \tild{\ell}(\cal{G}) - \tild{n} (\cal{F}) &= -2 r \mathbf{\alpha} e^{\gamma} (\mathbf{e} -\mathbf{f}) \\
\tild{\ell}(\cal{G}) + \tild{n} (\cal{F}) &= 2 r \mathbf{\alpha} e^{\gamma} \mathbf{m}
\end{align*}
 $\iff$
\begin{subequations} \label{nullode}
\begin{align}
 \tild{\ell}(\cal{G}) &= -r \mathbf{\alpha} e^{\gamma} (\mathbf{e} + \mathbf{m} -\mathbf{f}) \\
\tild{n}(\cal{F}) &= r \mathbf{\alpha} e^{\gamma} (\mathbf{e} - \mathbf{m} -\mathbf{f}). 
\end{align}
\end{subequations}

\begin{lem} \label{null_uniform}
 There exist constants $c^{\,-}_{\,\cal{G}},\,c^{\,+}_{\,\cal{G}},\,c^{\,-}_{\,\cal{F}}\,\text{and}\,\,c^{\,+}_{\,\cal{F}} $ such that
 the following uniform bounds hold
\begin{align*}
 c^{\,-}_{\,\cal{G}} \,\leq&\, \cal{G}\, \leq\, c^{\,+}_{\,\cal{G}} \\
 c^{\,-}_{\,\cal{F}}\, \leq&\, \cal{F}\, \leq \,c ^{\,+}_{\,\cal{F}}. \\
 \end{align*}
\end{lem}

\begin {proof}
 From \eqref{nullode} we have the following equations for $\tild{\ell}(\cal{G})$ and $\tild{n} (\cal{F}) $
\begin{align*}
 \tild{\ell}(\cal{G}) &= -r \mathbf{\alpha} e^{\gamma} (\mathbf{e} + \mathbf{m} -\mathbf{f}) \\
\tild{n}(\cal{F}) &= r \mathbf{\alpha} e^{\gamma} (\mathbf{e} - \mathbf{m} -\mathbf{f}) 
\end{align*}
 $\tild{\xi}$ and $\tild{\eta}$ are the parameters along the integral curves of $\tild{\ell}$ and $\tild{n}$ 
respectively.
\begin{figure}[!hbt]
\psfrag{teqtau}{$t=\tau$}
\psfrag{teq-1}{$t=-1$}
\psfrag{O}{$O$}
\psfrag{xitild}{$\tild{\xi}$}
\psfrag{etatild}{$\tild{\eta}$}
\centerline{\includegraphics[height=2.5in]{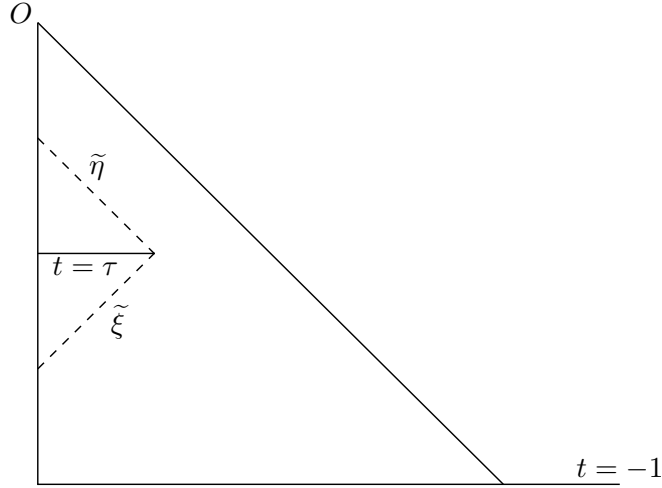}}

\caption{Application of the fundamental theorem of calculus along the integral curves of $\tild{\ell}$\,\,(parameterized by 
$\tild{\xi}$\,) and $\tild{n}$\,\,(parameterized by $\tild{\eta}$\,), for the estimates on $\cal{G}$ and $\cal{F}$ respectively.} 
\label{fig:null_coordinates_estimates}
\end{figure} 
Integrating each of \eqref{nullode} using the fundamental theorem of calculus in the region $J^-(O)$, we have
\begin{align*}
 \cal{G}(\tild{\xi}) &= \cal{G}(0) - \int^{\tild{\xi}}_0 \tild{\ell}(\cal{G}) d\,\tild{\xi}' \\
\cal{F}(\tild{\eta}) &= \cal{F}(0) - \int^{\tild{\eta}}_0 \tild{n}(\cal{F}) d\,\tild{\eta}'.
\end{align*}
Therefore, we have
\begin{align*}
  \cal{G}(\tild{\xi})  &= -\int^0_{\tild{\xi}} \ptl_{\tild{\xi}'} \cal{G}  \, d \, \tild{\xi}' \\
 &= \mbo{\alpha} \int^0_{\tild{\xi}} re^{\gamma}(\mathbf{e} + \mathbf{m}-\mathbf{f}) \, d\, \tild{\xi}'\\
 & \leq \mbo{\alpha} \int^0_{\tild{\xi}} re^{\gamma}(\mathbf{e} + \mathbf{m}) \, d\, \tild{\xi}' \\
 & \leq c(E_0)
\end{align*}
and 
\begin{align*}
 \cal{G}(\tild{\xi}) &= \mbo{\alpha} \int^0_{\tild{\xi}} re^{\gamma} \left(\mathbf{e_0} + \mathbf{m}-\halb \mathbf{f}\right) \, d\, \tild{\xi}' \\
 & \geq -\halb \mbo{\alpha} \int^0_{\tild{\xi}} re^{\gamma}\mathbf{f} \, d\, \tild{\xi}' \\
 & \geq -\mbo{\alpha} \int^0_{\tild{\xi}} re^{\gamma}(\mathbf{e} + \mathbf{m}) \, d\, \tild{\xi}' \\
 & \geq -c(E_0).
\end{align*}
Similarly,
\begin{align*}
 \cal{F} (\tild{\eta}) &= -\int^{\tild{\eta}}_0 \tild{n} (\cal{F}) \, d\, \tild{\eta}' \\
 & = \mbo{\alpha} \int^0_{\tild{\eta}} re^{\gamma}(-(\mathbf{e}-\mathbf{m}) + \mathbf{f}) \, d\, \tild{\eta}' \\
 &= \mbo{\alpha} \int^0_{\tild{\eta}} re^{\gamma}\left(-(\mathbf{e_0}-\mathbf{m}) + \halb\mathbf{f} \right) \, d\, \tild{\eta}' \\
 & \leq \mbo{\alpha} \int^0_{\tild{\eta}} re^{\gamma}\left(\halb \mathbf{f}\right) \, d\, \tild{\eta}' \\
 & \leq \mbo{\alpha} \int^0_{\tild{\eta}} re^{\gamma}((\mathbf{e}-\mathbf{m})) \, d\, \tild{\eta}' \\
 & \leq c(E_0)
\end{align*}
and
\begin{align*}
\cal{F} (\tild{\eta}) &= \mbo{\alpha} \int^0_{\tild{\eta}} re^{\gamma}(-(\mathbf{e}-\mathbf{m}) + \mathbf{f}) \, d\, \tild{\eta}' \\
& \geq -\mbo{\alpha} \int^0_{\tild{\eta}} re^{\gamma}((\mathbf{e}-\mathbf{m})) \, d\, \tild{\eta}' \\
& \geq - c(E_0).
\end{align*}

\end {proof}
It may be noted that, the existence of solutions of the system \eqref{nullode} is equivalent to the null $1$-forms $\ell$ and $n$ being closed.
Let us then introduce the null coordinates $(\xi,\eta,\theta)$ such that the
line element in $(M,g)$ can be represented as
\begin{align*}
d\,s_g^2 = -e^{2Z(\xi,\eta)}\,  d\,\xi d\,\eta + r^2\,(\xi,\eta)\, d\, \theta^2
\end{align*}
Let us also introduce the scalar functions $T$ and $R$ such that 
\begin{align*}
 T + R &= \xi \\
T - R & = \eta .
\end{align*}
So the line element in $(T,R, \theta)$ coordinates is
\begin{align*}
 d\,s_g^2 = e^{2Z(T,R)}\, (- d\,T^2 + d\,R^2) + r^2\,(T,R)\, d\, \theta^2
\end{align*}
In null coordinates $(\xi,\eta,\theta)$, the Ricci tensor is
\begin{align*}
 \mathbf{R}_{\xi\xi} =& r^{-1} (2Z_\xi r_\xi - r_{\xi\xi}),\\
\mathbf{R}_{\xi\eta} =& -(2 Z_{\xi\eta} + r^{-1} r_{\xi \eta}),\\
\mathbf{R}_{\eta \eta} =& r^{-1} (2Z_\eta r_\eta - r_{\eta\eta}), \\
\mathbf{R}_{\theta\theta} =& 4 r\,e^{-2 Z}  r_{\xi\eta}, \\
\mathbf{R}_{\xi\theta} =& 0\,\,\text{and} \\
\mathbf{R}_{\eta\theta} =& 0.
\end{align*}
The scalar curvature is
\begin{align*}
 R_g={}&8 e^{-2 Z} \bigl( Z_{\xi\eta} + r^{-1}r_{\xi\eta} \bigr)
\end{align*}
and the Einstein tensor is given by
\begin{align*}
\mathbf{E}_{\xi\xi} =& r^{-1} (2 Z_\xi r_\xi - r_{\xi\xi}),\\
\mathbf{E}_{\xi\eta} =& r^{-1} r_{\xi \eta},\\
\mathbf{E}_{\eta \eta} =& r^{-1} (2 Z_\eta r_\eta - r_{\eta\eta}),\\
\mathbf{E}_{\theta \theta} =& -4 r^2 \,e^{-2 Z} Z_{\xi\eta},\\
\mathbf{E}_{\xi \theta} =& 0\,\,\text{and} \\
\mathbf{E}_{\eta\theta} =& 0.
\end{align*}
Let us consider the Jacobian $\mathbf{J}$ of the transition functions between $(t,r,\theta)$ and $(\xi,\eta,\theta)$
\begin{align*}
 \mathbf{J} \fdg= &  \left( \begin{array}{ccc}
\ptl_\xi t & \ptl_\eta t & \ptl_\theta t \\
\ptl_\xi r& \ptl_\eta r & \ptl_\theta r\\
\ptl_\xi \theta & \ptl_\eta \theta & \ptl_\theta \theta \end{array} \right) \\ \\
=& \left( \begin{array}{ccc}
e^{\mathcal{F}-\Omega} & e^{\mathcal{G}-\Omega} &0  \\
e^{\mathcal{F}-\gamma} & -e^{\mathcal{G}-\gamma} &0 \\
 0 & 0 & 1 \end{array} \right)
\end{align*}
then the determinant $|\mathbf{J}|$ and inverse Jacobian $\mathbf{J}^{-1}$ are given by
\begin{align*}
|\mathbf{J}| =& -2 e^{(\mathcal{F} + \mathcal{G}) - (\gamma + \Omega)} \\ \\
 \mathbf{J}^{-1} =& \halb \left( \begin{array}{ccc}
e^{-\mathcal{F}+\Omega} & e^{-\mathcal{F}+ \gamma} &0  \\
e^{-\mathcal{G}+\Omega} & -e^{\mathcal{G}+\gamma} &0 \\
 0 & 0 & 2 \end{array} \right).
\end{align*}
Therefore,
\begin{align}
 \ptl_t \xi =& \halb e^{-\cal{F} + \Omega},\,\, \ptl_r \xi = \halb e^{-\cal{F} + \gamma} \notag \\
\ptl_t \eta = & \halb e^{-\cal{G} + \Omega},\,\, \ptl_r \eta = -\halb e^{-\cal{G} + \gamma}
\end{align}
so that 
\begin{align*}
 d\, \xi &= \halb \left(  e^{(\mathcal{-F} + \Omega)} d\,t + e^ {(\mathcal{-F} + \gamma)} d\,r\right) \\
d\, \eta &= \halb \left(  e^{(\mathcal{-G} + \Omega)} d\,t - e^ {(\mathcal{-G} + \gamma)}d \, r \right).
\end{align*}

\begin{cor}
 There exist constants $c^{\,-}_{\,\mu\nu},c^{\,+}_{\,\mu \nu}$ and $C^{\,-}_{\,\mu \nu}, C^{\,+}_{\,\mu \nu}$
such that all the entries of the Jacobian $\mathbf{J}$ and its inverse 
$\mathbf{J}^{-1}$ are uniformly bounded 
\begin{align*}
 c^{\,-}_{\,\mu \nu} \,\leq&\, \mathbf{J}_{\,\mu \nu}\, \leq\, c^{\,+}_{\,\mu \nu} \\
C^{\,-}_{\,\mu \nu} \,\leq&\, \mathbf{J}^{-1}_{\,\mu \nu}\, \leq\, C^{\,+}_{\,\mu \nu}
\end{align*}
for $\mu,\nu = 0,1,2.$
\end{cor}
\begin{proof}
The proof follows from Lemmas \ref{metric_uniform} and \ref{null_uniform}.
\end{proof}
\begin{cor}
 There exist constants $c^-_Z$ and $c^+_Z$ such that the following uniform bounds hold
 on the metric function $Z$ in null coordinates.

\begin{align}
 c^-_Z \leq Z \leq c^+_Z.
\end{align}
\end{cor}
\begin{proof}
 We have
\begin{align*}
 -e^{2Z} \,d\,\xi\, d\,\eta = -e^{2\Omega}d\,t^2 + e^{2\gamma}d\,r^2,
\end{align*}
therefore,
\begin{align*}
 e^Z = \frac{1}{4} \, e^{\cal{F}+ \cal{G}}.
\end{align*}
The result now follows from the Lemma \ref{null_uniform}.
\end{proof}

\begin{cor}
 There exist constants $c_1,c_2,c_3,c_4$ such that the
pointwise bounds 
\begin{align*}
 r \geq & \, c_1\, R,\,\,\,\,\,\,\,\,\,\,\,\,\,\,\,\,\,  t \geq \,c_3\, T, \\
r \leq &\,c_2\, R \,\,\,\,\,\text{and}\,\,\,\,\,  t \leq  \, c_4\, T 
\end{align*}
hold for the scalar functions $r,t,R$ and $T$.
\end{cor}
\begin{proof}
 We have 
\begin{subequations} \label{pointwise_scalar}
\begin{align}
\ptl_{\,R}\, r \leq& \,\, |\ptl_R r| = |\ptl_\xi r -\ptl_\eta r | = | e^{\cal{F}-\gamma} + e^{\cal{G}-\gamma}| \leq c_1 (E_0) \\ \notag\\
\ptl_{\,r} R \leq& \,\, |\ptl_r R| = \halb |\ptl_r \xi -\ptl_r \eta| = \frac{1}{4} |e^{-\cal{G}_ \gamma} -e^{-\cal{G} + \gamma}| \leq c_2 (E_0)   \\ \notag \\
\ptl_{\,T} \,t \leq&\,\, |\ptl_T t| = | \ptl_\eta t + \ptl_\eta t|= |e^{\cal{F}-\Omega}+ e^{\cal{G} -\Omega} | \leq c_3 (E_0) \\ \notag \\
\ptl_{\,t}\, T \leq& \,\,|\ptl_t T| = \halb |\ptl_t \xi + \ptl_t \eta| = \frac{1}{4} |e^{-\cal{F}+ \Omega} + e^{-\cal{G}+ \Omega}| \leq c_4 (E_0).
\end{align}
\end{subequations}
The proof follows by applying the fundamental theorem of calculus to each of \eqref{pointwise_scalar} in the region $J^-(O)$ and noting
that at $O$, $t=T=0$ and $r=R=0$ on the axis. 
\end{proof}
Let us now revisit the Stokes' theorem for $\bar{\mu}_g$-divergence of $\mathbf{P}_{\mathbf{X}}$ in $K(\tau,s)$.
The 1-forms $\ell$ and $n$ are
\begin{align*}
\ell =& -e^{\cal{F}}\left(e^{\Omega} d\,t - e^{\gamma}d\,r \right) \\
n=&-e^{\cal{G}}\left( e^{\Omega} d\,t + e^{\gamma}d\,r \right)
\end{align*}
and we have
\begin{align*}
 d\,\xi =& -\halb e^{-(\cal{F}+ \cal{G})} n = -\halb e^{-\cal{F}} \tild{n}  \\
 d\,\eta =&-\halb e^{-(\cal{F}+\cal{G})}\ell = -\halb e^{-\cal{G}} \tild{\ell}
\end{align*}
The volume 3-form of $(M,g)$ is
\begin{align*}
 \bar{\mu}_g =&\, re^{\gamma + \Omega} d\,t\wedge d\,r \wedge d\,\theta \\
 =&\, \halb \, re^{2Z}d\,\eta \wedge d\,\xi \wedge d\,\theta
\end{align*}
Let us introduce the 2-forms $\bar{\mu}_\xi$ and $\bar{\mu}_\eta$ as follows
\begin{align*}
 \bar{\mu}_g= d\,\xi \wedge \bar{\mu}_\xi \\
 \bar{\mu}_g= d\,\eta \wedge \bar{\mu}_\eta
\end{align*}
so that
\begin{align*}
 \bar{\mu}_\xi  =& -\halb r e^{2Z} \left( d\,\eta \wedge d\,\theta  \right) \\
 \bar{\mu}_\eta =& \,\halb r e^{2Z} \left( d\,\xi \wedge d\,\theta  \right).
\end{align*}
Now,
\begin{align*}
 \text{Flux}(\mathbf{P}_{\mathbf{X}}) (\tau,s) =& \int_{C(\tau,s)} d\,\xi(\mathbf{P_X})\, \bar{\mu}_\xi, \\
 \intertext{for instance,}
 \text{Flux}(\mathbf{P}_{\mathbf{X_1}}) (\tau,s) =& \int_{C(\tau,s)} d\,\xi(\mathbf{P}_\mathbf{X_1}) \,\bar{\mu}_\xi,\\
 =& -\halb \int_{C(\tau,s)}e^{-\cal{F}}(\mathbf{e-m})  \bar{\mu}_\xi.
\end{align*}
Note that $d\,\xi (n)= d\,\eta (\ell)=0.$

\chapter{Non-Concentration of Energy}
In this Chapter we shall use the vector fields method introduced in Section \ref{vfm} to prove that
the energy of the system \eqref{ewmcauchy_equi_new} does not concentrate. We start with proving that
the energy does not concentrate away from the axis using the divergence free vector $\mathbf{P}_{\mathbf{X_1}}$.
\section{Non-Concentration of Energy Away from the Axis}


\begin{lem} \label{ann}
\begin{align*}
 E^{\,O}_{\text{ext}} (\tau) \fdg = \int_{B_{r_2(\tau)}\setminus B_{r_1(\tau)}} \mathbf{e} \,\,\bar{\mu}_q \to 0 \,\, \text{as} \,\, \tau \to 0,
\end{align*} 
where $r = r_2(\tau)$ is the radius where the $t = \tau$ slice intersects the $R = |T|$ curve i.e the mantel of the null cone $ J^-(O)$
 and  $r = r_1 (\tau)$ is the radius where the $ R = \lambda |T|$ curve intersects the $t = \tau$ slice, for 
$\lambda \in (0,1)$.
 Observe that both $r_1(\tau)$
and $r_2(\tau) \to 0$ as $\tau \to 0$.
\end{lem}
\begin{proof}
Consider a tubular region $\cal{S}$ with triangular cross section (as shown in the figure \ref{fig:annular_disc_1} ) in  $R > \lambda T, \, \lambda \in (0,1)$ 
of the spacetime i.e., the ``exterior'' part of the interior of the past null cone of $O$.
\begin{figure}[!hbt]
\psfrag{O}{$O$}
\psfrag{O}{$O$}
\psfrag{teqtau}{$t=\tau$}
\psfrag{teq-1}{$t=-1$}
\psfrag{S}{$\cal{S}$}
\psfrag{1}{$\ptl \cal{S}_1$}
\psfrag{2}{$\ptl \cal{S}_2$}
\psfrag{3}{$\ptl \cal{S}_3$}
\psfrag{Req0}{$R=0$}
\psfrag{ReqlT}{$R=\lambda T$}
\psfrag{ReqT}{$R=|T|$}

\centerline{\includegraphics[height=2.5in]{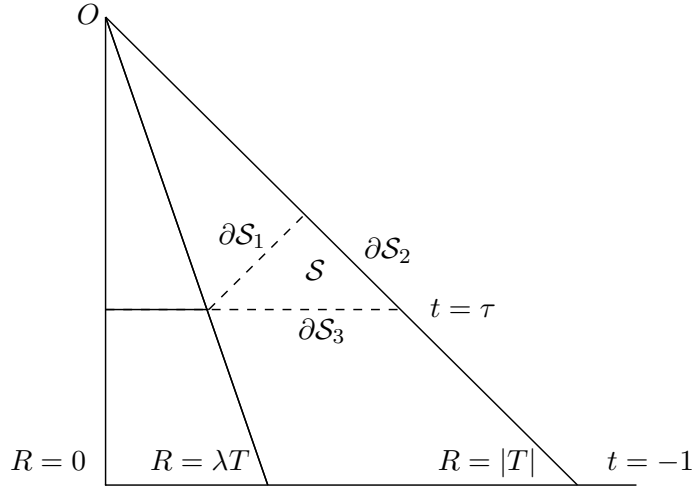}}

\caption{ Application of Stokes' theorem on the $\bar{\mu}_g$-divergence free $\mathbf{P}_{\mathbf{X_1}}$ to relate the fluxes through 
surfaces $\ptl \cal{S}_1$,\,$\ptl \cal{S}_2$\, and $\ptl \cal{S}_3$} 
\label{fig:annular_disc_1}
\end{figure} 

 As shown in the figure \ref{fig:annular_disc_1}, let us use the divergence-free vector field 
$\mathbf{P}_{\mathbf{X_1}}$  and the Stokes' theorem in the region $\cal{S}$ to relate the fluxes through the three boundary segments
$\ptl \cal{S}_1$,\,$\ptl \cal{S}_2$ and $\ptl \cal{S}_3$.
We have,
\begin{align}\label{div_tube}
 \int_\cal{S} \nabla_\nu  \mathbf{P}^\nu_\mathbf{X_1} = 0 =&  \int_{\ptl \cal{S}_1} d\,\eta(\mathbf{P}_\mathbf{X_1}) \bar{\mu}_\eta 
 + \int_{\ptl \cal{S}_2} d\,\xi(\mathbf{P}_\mathbf{X_1})\bar{\mu}_\xi
 - \int_{\ptl \cal{S}_3} e^{\Omega} \mathbf{P}^t_\mathbf{X_1} \, \bar{\mu}_q \notag \\
 =& - \halb\int_{\ptl \cal{S}_1} e^{-\mathcal{G}} (\mathbf{e+m})\,\bar{\mu}_\eta  
-\halb \int_{\ptl \cal{S}_2} e^{-\mathcal{F}} (\mathbf{e-m})\,\bar{\mu}_\xi + \int_{\ptl \cal{S}_3} \mathbf{e} \, \bar{\mu}_q .
\end{align}
To analyze the behaviour of the flux terms $\int_{\ptl \cal{S}_1}$ and $\int_{\ptl \cal{S}_2}$ in \eqref{div_tube} close to $O$,
let us define
\begin{align*}
 \hatt{l} &\fdg= e^{\gamma + \Omega}\tild{\ell} = e^{\gamma}\ptl_t + e^{\Omega}\ptl_r \\
\hatt{n} &\fdg= e^{\gamma + \Omega}\tild{n} = e^{\gamma}\ptl_t - e^{\Omega}\ptl_r \\
\mathcal{A}^2 &\fdg= r(\mathbf{e}-\mathbf{m}) \\
\mathcal{B}^2 &\fdg = r(\mathbf{e}+ \mathbf{m}). \\
\end{align*}
From \eqref{equiv_div_px1}, we have
\begin{align}\label{equiv_div_px1_new}
0= \grad_\nu \mathbf{P}_\mathbf{X_1}^\nu  = \frac{1}{re^{\gamma + \Omega}} \left(-\ptl_t (re^{\gamma}\mathbf{e}) 
+ \ptl_r (re^{\Omega}\mathbf{m}) \right).
\end{align}
Let us try to get another useful identity with the $\mathbf{X_2} = e^{-\gamma}\ptl_r$ multiplier. 
Recall the vector $\mathbf{P}_{\mathbf{X_2}} -e^{-\Omega} \mathbf{m} \, \ptl_t + e^{-\gamma} (\mathbf{e}-\mathbf{f}) \,\ptl_r  $
and the two equivalent expressions for its divergence in \eqref{div_px2} and \eqref{equiv_div_px2}

\begin{align}\label{equal_div_px2}
 \grad_\nu \mathbf{P}_\mathbf{X_2}^\nu &= \frac{1}{\sqrt {-g}} \ptl_\nu (\sqrt{-g}\, \mathbf{P}_\mathbf{X_2}^\nu) \notag\\
 & = \frac{1}{re^{\gamma + \Omega}} \left(-\ptl_t (re^{\gamma}\mathbf{m}) 
+ \ptl_r ((\mathbf{e}-\mathbf{f}) re^{\Omega}) \right) \notag\\
 & =  \halb \, \leftexp{(\mathbf{X_2})} {\mbo{\pi}}_{\alpha \beta} \mathbf{T}^{\alpha \beta} \notag \\
& = - e^{-\gamma} \Omega_r \,\mathbf{e} + \frac{1}{2r} e^{-\gamma} (e^{-2\Omega} u_t^2 - e^{-2\gamma}u_r^2 + \mathbf{f}) + e^{-\Omega} \gamma_t \mathbf{m}.
\end{align}
Therefore, we have the following identities from \eqref{equiv_div_px1_new} and \eqref{equal_div_px2}
\begin{subequations} \label{px1px2_identities}
\begin{align}
 \ptl_t(re^{\gamma}\mathbf{e})- \ptl_r(re^\Omega \mathbf{m}) & = 0 \\
\ptl_t(r e^{\gamma}\mathbf{m}) - \ptl_r({re^{\Omega} \mathbf{e}}) &= L 
\end{align}
\end{subequations}
where 
\[L \fdg = \frac{re^{\Omega}\Omega_r}{2} \left((\mathbf{X_1}u)^2 + (\mathbf{X_2}u)^2 -\mathbf{f} \right) + e^{\Omega} L_0 - r\gamma_t e^{\gamma} \mathbf{m}  \]
for
\[ L_0 \fdg = \halb \left(-(\mathbf{X_1}u)^2 + (\mathbf{X_2}u)^2 + \mathbf{f} \right) -\frac{2 f(u)f_u(u) u_r}{r} .\]
Furthermore, we can construct the following using the identities in \eqref{px1px2_identities}
\begin{subequations}\label{px1px2_newidentities}
\begin{align}
 \ptl_{\alpha}\left(re^{\gamma + \Omega}(\mathbf{e}-\mathbf{m})\tild{\ell} ^\alpha \right) &= \ptl_\alpha(\mathcal{A}^2 \hatt{\ell}^{\alpha}) = -L \\
\ptl_{\alpha}\left(re^{\gamma + \Omega}(\mathbf{e}+\mathbf{m})\tild{n}^\alpha \right) &= \ptl_\alpha(\mathcal{B}^2 \hatt{n}^{\alpha}) = L. 
\end{align}
\end{subequations}
Let us try to express $L$ in terms of $\cal{A}^2$ $\cal{B}^2$ after using the Einstein's equations
\begin{align}\label{simplifyL}
 L &= e^{\Omega}L_0 + \mbo{\alpha} r^2 e^{2\gamma + \Omega} \left( \mathbf{e}-\mathbf{f} \right)^2 - \alpha r^2 e^{2\gamma + \Omega} \mathbf{m}^2 \notag\\
&= e^{\Omega}L_0 + \mbo{\alpha} r^2 e^{2\gamma + \Omega} \left( \mathbf{e}^2 -2 \,\mathbf{e} \,\mathbf{f} + \mathbf{f}^2 -\mathbf{m}^2 \right) \notag\\
&= e^{\Omega}L_0 + \mbo{\alpha} e^{2\gamma + \Omega} \left( \mathcal{A}^2 \, \mathcal{B}^2 -2\, r^2 \,\mathbf{e}\,\mathbf{f} + r^2 \mathbf{f}^2 \right). 
\end{align}
We would like to set up a Gr\"onwall estimate for $\cal{B}$ using the identities in \eqref{px1px2_newidentities}. However, the quantity $L$ as shown
in \eqref{simplifyL} has nonlinear terms involving $\mathbf{e}$ and $\mathbf{f}$. Therefore, in what follows we introduce the parameters $k_\ell$ and $k_n$,
and use Einstein's equations to estimate these terms. \\

Firstly note that
\begin{align*}
  \hatt{\ell}^\mu \ptl_\mu e^{k_{\ell} \gamma} =& k_{\ell} e^{k_{\ell} \gamma} (e^{\gamma}\gamma_t + e^{\Omega}\gamma_r)  &
\hatt{n}^\mu \ptl_\mu e^{k_n \gamma} =& k_n e^{k_n \gamma} (e^{\gamma}\gamma_t - e^{\Omega}\gamma_r) \\ 
=& k_{\ell} e^{k_{\ell} \gamma} \alpha r e^{2\gamma + \Omega}(\mathbf{m} + \mathbf{e})  &
 =& k_n e^{k_n \gamma} \alpha r e^{2\gamma + \Omega}(\mathbf{m} - \mathbf{e})  \\
=& k_{\ell} \alpha e^{k_{\ell} \gamma}e^{2\gamma + \Omega} \mathcal{B}^2 &
=& -k_n \alpha e^{k_n \gamma}e^{2\gamma + \Omega} \mathcal{A}^2 \\
\end{align*}
and
\begin{align*}
\ptl_\mu \hatt{\ell}^\mu =& e^{\gamma} \gamma_t + e^{\Omega} \Omega_r & 
\ptl_\mu \hatt{n}^\mu =& e^{\gamma} \gamma_t - e^{\Omega} \Omega_r \\
=& r \mbo{\alpha} e^{2\gamma + \Omega} \left(\mathbf{e} + \mathbf{m} - \mathbf{f} \right) &
=& r \mbo{\alpha} e^{2\gamma + \Omega} \left(-\mathbf{e} + \mathbf{m} + \mathbf{f} \right)\\
=& \mbo{\alpha} e^{2\gamma + \Omega} \left(\mathcal{B}^2 -r \,\mathbf{f} \right) &
=&  \mbo{\alpha} e^{2\gamma + \Omega} \left(-\mathcal{A}^2 + r \, \mathbf{f} \right) . \\
\end{align*}
Now consider the quantities $ \ptl_{\mu} (e^{k_{\ell}\gamma} \mathcal{A}^2 \hatt{\ell}^\mu)$ and $ \ptl_{\mu} (e^{k_n\gamma} \mathcal{B}^2 \hatt{n}^\mu)$,
\begin{align*}
\hatt{\ell}^\mu \ptl_\mu (e^{k_{\ell}\gamma}\mathcal{A}^2) &= \ptl_{\mu} (e^{k_{\ell}\gamma} \mathcal{A}^2 \hatt{\ell}^\mu) - e^{k_{\ell}\gamma}\mathcal{A}^2 \ptl_\mu \hatt{\ell}^\mu \\
& = e^{k_{\ell} \gamma} \ptl_\mu (\mathcal{A}^2 \hatt{\ell}^\mu) + \mathcal{A}^2 \hatt{\ell}^\mu \ptl_\mu e^{k_{\ell}\gamma} -\mbo{\alpha} e^{k_{\ell}\gamma}e^{2\gamma + \Omega}\mathcal{A}^2 \mathcal{B}^2  
+ r \mbo{\alpha} e^{k_{\ell}\gamma}e^{2\gamma + \Omega}\mathcal{A}^2  \, \mathbf{f} \\
& = - e^{k_{\ell} \gamma} L + (k_{\ell} -1) \alpha e^{k_{\ell}\gamma}e^{2\gamma + \Omega}\mathcal{A}^2 \mathcal{B}^2 
+ r \mbo{\alpha} e^{k_l\gamma}e^{2\gamma + \Omega}\mathcal{A}^2 \, \mathbf{f} \\
&= e^{k_{\ell} \gamma}e^{\Omega}\left( -L_0 + \mbo{\alpha} e^{2\gamma}(k_{\ell}-2) \mathcal{A}^2 \mathcal{B}^2 + 2\,r^2\, \mathbf{e}\,\mathbf{f}- r^2\mathbf{f}^2 
+ r \mathcal{A}^2 \mathbf{f}  \right) \\
& = e^{k_{\ell} \gamma}e^{\Omega}\left( -L_0 + \mbo{\alpha} r^2 e^{2\gamma}\left( (k_{\ell}-2)(\mathbf{e}^2 - \mathbf{m}^2) + 3 \mathbf{e}\,\mathbf{f}  - \mathbf{f}^2
 - \mathbf{m}\, \mathbf{f} \right) \right) \\
\end{align*}
\begin{align*}
\hatt{n}^\mu \ptl_\mu (e^{k_n\gamma}\mathcal{B}^2) &= \ptl_{\mu} (e^{k_n\gamma} \mathcal{B}^2 \hatt{n}^\mu) - e^{k_n\gamma}\mathcal{B}^2 \ptl_\mu \hatt{n}^\mu \\
& = e^{k_n \gamma} \ptl_\mu (\mathcal{B}^2 \hat{n}^\mu) + \mathcal{B}^2 \hat{n}^\mu \ptl_\mu e^{k_n\gamma} + \alpha e^{k_n\gamma}e^{2\gamma + \Omega}\mathcal{A}^2 \mathcal{B}^2  
- r \alpha e^{k_n\gamma}e^{2\gamma + \Omega}\mathcal{B}^2 \frac{f^2(u)}{r^2} \\
& = e^{k_n \gamma} L +  (-k_n + 1) \alpha e^{k_n\gamma}e^{2\gamma + \Omega}\mathcal{A}^2 \mathcal{B}^2 
- r \alpha e^{k_n\gamma}e^{2\gamma + \Omega}\mathcal{B}^2 \, \mathbf{f} \\
&= e^{k_n \gamma}e^{\Omega}\left( L_0 + \alpha e^{2\gamma}\left((-k_n+2) \mathcal{A}^2 \mathcal{B}^2 - 2r^2 \mathbf{f}\mathbf{e} + r^2\mathbf{f}^2 
- r \mathcal{B}^2 \mathbf{f} \right) \right) \\ 
& = e^{k_n \gamma} e^{\Omega} \left(L_0 + \alpha r^2 e^{2\gamma}\left( (-k_n+2)(\mathbf{e}^2 - \mathbf{m}^2) - 3 \mathbf{e}\,\mathbf{f} + \mathbf{f}^2 
+ \mathbf{m}\,\mathbf{f}  \right) \right). \\
\end{align*}
Let us define
\begin{align*}
 S_{k_\ell} \fdg &= (k_{\ell}-2 )(\mathbf{e}^2 - \mathbf{m}^2) + 3 \mathbf{e} \, \mathbf{f} - \mathbf{f}^2 - \mathbf{m}\, \mathbf{f}  \\
& = (k_{\ell}-2 )(\mathbf{e}^2 - \mathbf{m}^2)  + (\mathbf{e} - \mathbf{m}) \, \mathbf{f} + \mathbf{e_0} \, \mathbf{f}   \\
& \geq 0
\end{align*}
for $ k_{\ell} \geq 2$. Note that we have $\mathbf{e} \geq |\mathbf{m}|.$
Similarly define
\begin{align*}
 S_{k_n} \fdg &= (-k_n + 2)(\mathbf{e}^2 - \mathbf{m}^2) - 3 \mathbf{e} \, \mathbf{f} + \mathbf{f}^2 + \mathbf{m}\, \mathbf{f}  \\
& = (-k_n + 2) (\mathbf{e}^2 - \mathbf{m}^2) - (\mathbf{e} - \mathbf{m}) \, \mathbf{f} - \mathbf{e_0} \, \mathbf{f} \\
& \leq 0
\end{align*}
for $k_n \geq 2$.
Hence, for the choice of $ k_{\ell} =2= k_n = 2$, let us now introduce the quantities $\widehat{\cal{A}}$ and $\widehat{\cal{B}}$
such that 
\[  \widehat{\cal{A}} \fdg = e^{\gamma} \cal{A} \]
and 
\[ \widehat{\cal{B}} \fdg = e^{\gamma} \cal{B}. \]
In the following we will try to estimate $L_0^2$ by $\mathbf{e}^2 - \mathbf{m}^2$.
We will use the following identities which are valid for all real $a,b,c$
\begin{align*}
  (a + b + c) ^2 & = 3(a^2 + b^2 + c^2)- \left( (a-b)^2 + (b-c)^2 + (c-a)^2 \right) \\
& \leq 3(a^2 + b^2 + c^2). \\
\frac{1}{4}(-a^2 + b^2)^2 &= \frac{1}{4}(a^2 + b^2)^2 -a^2 b^2 .\\
\end{align*}
So consider,
\begin{align*}
 L_0^2 & \leq 3\left( \frac{1}{4}\left(-(\mathbf{X_1}u)^2 + (\mathbf{X_2}u)^2\right)^2 +  4  f_u^2(u)u^2_r\, \mathbf{f} + \frac{1}{4}\,\mathbf{f}^2 \right) \\
&=  3 \left( \frac{1}{4} \mathbf{e_0}^2 + 4  f_u^2(u)u^2_r\, \mathbf{f} + \frac{1}{4}\,\mathbf{f}^2 -\mathbf{m}^2 \right) \\
& \leq 3\left(\frac{1}{4} \mathbf{e_0}^2 + \frac{c}{2} (\mathbf{X_2}u)^2 \, \mathbf{f} + \frac{1}{4}\,\mathbf{f}^2 -\mathbf{m}^2  \right) \\
& \leq 3 \left(\frac{1}{4} \mathbf{e_0}^2 + \frac{c}{2} (\mathbf{X_2}u)^2 \, \mathbf{f} + \frac{c}{2} (\mathbf{X_1}u)^2 \, \mathbf{f} + \frac{1}{4}\,\mathbf{f}^2 -\mathbf{m}^2  \right)  \\
& \leq c\left(\frac{1}{4} \mathbf{e_0}^2 + \frac{1}{2} \, \mathbf{e_0}\, \mathbf{f} + \frac{1}{4}\,\mathbf{f}^2 -\mathbf{m}^2  \right)  \\
& = c(\mathbf{e}^2 -\mathbf{m}^2 ) \\
\end{align*}
where we have used the fact that both $||u||_{L^\infty}$ and $||\gamma||_{L^\infty} \leq c$. 
Furthermore we have,
\begin{align*}
L_0^2 \leq c \frac{\widehat{\cal{A}}^2\,\widehat{\cal{B}}^2} {r^2}
\end{align*}
consequently, 
\begin{align*}
 \ptl_{\xi} \widehat{\cal{A}}^2 =& e^{\gamma + \cal{F}} \left(-L_0 + \mbo{\alpha}\, r^2\, e^{2\gamma}S_2\right)  \\
\leq  & (-L_0).
\end{align*}
So,
\begin{align*}
\widehat{\cal{A}}\, \ptl_\xi \widehat{\cal{A}} \leq -c L_0 \leq c|L_0| \leq c \frac{\widehat{\cal{A}}\,\widehat{\cal{B}}}{r}
\end{align*}
that gives us
\[  \ptl_\xi \widehat{\cal{A}} \leq c \frac{\widehat{\cal{B}}}{ r}  \]
and similarly,
\[  \ptl_\eta \widehat{\cal{B}} \leq c \frac{\widehat{\cal{A}}} { r}. \]
The rest of the proof is comparable to the case of wave maps on the Minkowski background as in \cite{jal_tah1} and \cite{chris_tah1}.
Consider the region of spacetime $[\xi , 0] \times [\eta_0, \eta]$ where $\xi,  \eta \leq 0$. The integral curve of the vector field 
$\mathbf{X_1}$ passing through $O$ is the axis $r=0$ of $M$.
\begin{figure}[!hbt]
\psfrag{O}{$O$}
\psfrag{etaeqeta0}{$\eta_0$}
\psfrag{etaeqeta}{$\eta$}
\psfrag{xieqxi}{$\xi= \xi$}
\psfrag{Req0}{$R=0$}
\psfrag{ReqlT}{$R=\lambda T$}
\psfrag{ReqT}{$R=|T|$}
\psfrag{reqReq0}{$r=R=0$}
\psfrag{teq-1}{$t=-1$}
\centerline{\includegraphics[height=2.5in]{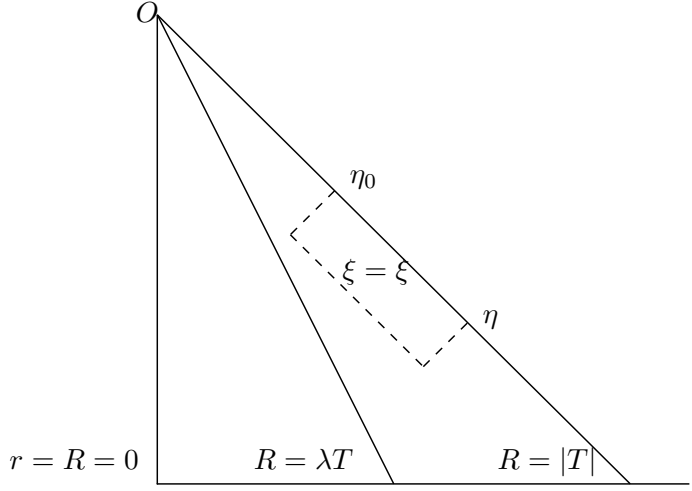}}

\caption{Application of the fundamental theorem of calculus for $\widehat{A}$ and $\widehat{B}$ in the region
$[\xi , 0] \times [\eta_0, \eta]$ } 
\label{fig:annular_disc_2}
\end{figure}

Using the fundamental theorem of calculus,
\[ \widehat{\cal{A}}(0,\eta) - \widehat{\cal{A}}(\xi,\eta) = \int^0_\xi \ptl_{\xi'} \widehat{\cal{A}}(\xi',\eta)\,\, d\, \xi' \]
\[ \widehat{\cal{B}}(\xi,\eta) - \widehat{\cal{B}}(\xi,\eta_0) = \int^\eta_{\eta_0} \ptl_{\eta'} \widehat{\cal{B}}(\xi,\eta')\,\, d\, \eta'. \]
So,
\begin{align}
 \widehat{\cal{B}}(\xi,\eta) &= \widehat{\cal{B}}(\xi,\eta_0) + \int^\eta_{\eta_0} \ptl_{\eta'} \widehat{\cal{B}}(\xi,\eta')\,\, d\, \eta' \notag \\
&\leq \widehat{\cal{B}}(\xi,\eta_0) + c\int^\eta_{\eta_0} \frac{\widehat{\cal{A}}(\xi,\eta')}{r(\xi,\eta')}\,\, d\, \eta' \label{ftc_B} \\
\widehat{\cal{A}}(\xi,\eta') &= \widehat{\cal{B}}( 0,\eta') - \int^0 _{\xi} \ptl_{\xi'} \widehat{\cal{A}}(\xi',\eta')\,\, d\, \xi' \notag  \\
&\leq \widehat{\cal{A}}(0,\eta') + c\int^0_{\xi} \frac{\widehat{\cal{B}}(\xi',\eta')}{r(\xi',\eta')}\,\, d\, \xi' \label{ftc_A}.
\end{align}
After plugging in $ \widehat{\cal{A}}(\xi, \eta')$ in \eqref{ftc_B} we get,
\begin{align} \label{original_B}
  \widehat{\cal{B}}(\xi,\eta) & \leq \widehat{\cal{B}}(\xi,\eta_0) + c \left(\int^\eta_{\eta_0} \frac{\widehat{\cal{A}}(0,\eta')}{r(\xi, \eta')}\,\, d\, \eta'
 + c \int^\eta_{\eta_0}\frac{1}{r(\xi,\eta')} \left(\int^0_\xi \frac{\widehat{\cal{B}}(\xi',\eta')}{r(\xi',\eta')}\,\, d\, \xi'\right) \,\,d\,\eta' \right) \notag\\
&= \widehat{\cal{B}}(\xi,\eta_0) + c \left(\int^\eta_{\eta_0} \frac{\widehat{\cal{A}}(0,\eta')}{r(\xi, \eta')}\,\, d\, \eta' \right) 
 + c \left( \int^\eta_{\eta_0}\int^0_\xi \frac{\widehat{\cal{B}}(\xi',\eta')}{ r(\xi, \eta') r(\xi',\eta')}\,\, d\, \xi' \,\,d\,\eta' \right).
\end{align}
Now consider the second term in the right hand side of \eqref{original_B}, firstly recall
\[ r(\xi,\eta')  \geq c R(\xi,\eta') = c \, \halb (\xi-\eta'), \]
\begin{align} \label{eq_before_pxi}
\int^\eta_{\eta_0} \frac{\widehat{\cal{A}}(0,\eta')}{r(\xi, \eta')}\,\, d\, \eta' & \leq  \left(\int^\eta_{\eta_0} \widehat{\cal{A}}^2(0,\eta')\,\,d\,\eta' \right)^{\halb} 
\left(\int^\eta_{\eta_0} \frac{1}{(\xi-\eta')^2}\,\,d\,\eta' \right)^{\halb} \notag \\
& \leq c \, \text{Flux}^{\halb} (\mathbf{P}_{\mathbf{X_1}}) (\eta_0,\eta) \left(\frac{1}{\xi-\eta} -  \frac{1}{\xi-\eta_0}\right)^{\halb} \notag \\
&\leq c\, \text{Flux}^{\halb} (\mathbf{P}_{\mathbf{X_1}}) (\eta_0)\left(\frac{1}{\xi-\eta}\right)^{\halb}.
\end{align}
Let us define the function $\widehat{\cal{H}}(\xi,\eta) \fdg= \sup_{\xi\leq \xi' \leq 0}\sqrt{\xi'-\eta}\,\, \widehat{\cal{B}}(\xi',\eta).$ 
Note that we are working in the region where $(\xi',\eta') \in [\xi,0] \times [\eta_0, \eta]$ such that $\xi' \neq \eta'$

\begin{align}
  \widehat{\cal{B}}(\xi,\eta) &\leq \widehat{\cal{B}}(\xi,\eta_0) + c\,\frac{\text{Flux}^\halb(\eta_0)}{(\xi-\eta)^\halb} 
+ c  \left( \int^\eta_{\eta_0}\int^0_\xi \frac{\widehat{\cal{B}}(\xi',\eta')}{ r(\xi, \eta') r(\xi',\eta')}\,\, d\, \xi' \,\,d\,\eta' \right)  .
\end{align}
We have,
\[ \sqrt {\xi' -\eta'}\,\widehat{\cal{B}}(\xi',\eta') \leq \sup_{\xi \leq \xi' \leq \eta} \sqrt{\xi'-\eta'}\,\widehat{\cal{B}}(\xi',\eta')  = \widehat{\cal{H}}(\xi,\eta'). \]
So,
\begin{align} \label{Hestimate}
 (\xi-\eta)^\halb \widehat{\cal{B}}(\xi,\eta) &\leq \left(\frac{\xi-\eta}{\xi-\eta_0}\right) ^\halb (\xi-\eta_0)^\halb \widehat{\cal{B}}(\xi,\eta_0) + c\text{Flux}^\halb(\eta_0) \notag \\
&\quad + c \left( \int^\eta_{\eta_0}\int^0_\xi  \widehat{\cal{H}}(\xi,\eta') \, \frac{(\xi-\eta)^\halb}{ (\xi- \eta') (\xi' -\eta')^{3/2}}\,\, d\, \xi' \,\,d\,\eta' \right). \notag \\
\end{align}
Now consider the function $p(\xi)$ defined as follows
\[ p \fdg = \frac{\xi-\eta}{\xi-\eta_0},\] 
we have $\xi-\eta \leq \xi-\eta_0$ so $ p \leq 1 .$
Differentiating $p(\xi)$ with respect to $\xi$, we get
\begin{align*}
p_\xi (\xi) =& \frac{(\xi-\eta_0) (-\eta)-(\xi-\eta)(-\eta_0)}{(\xi-\eta_0)^2} \\
=& \frac{\xi(\eta-\eta_0)}{(\xi-\eta_0)^2} \\
\leq & \, 0.
\end{align*}
Therefore we have,
\begin{align} \label{pxi}
 p(0) \leq p(\xi).
\end{align}
Let us go back to the inequality \eqref{Hestimate} and use \eqref{pxi}, we have
\begin{align}
(\xi-\eta)^\halb \widehat{\cal{B}}(\xi,\eta) & \leq \left(\frac{-\eta}{-\eta_0}\right) ^\halb (\xi-\eta_0)^\halb \widehat{\cal{B}}(\xi,\eta_0)+ c\text{Flux}^\halb(\eta_0)  \notag \\
& \quad + c \left( \int^\eta_{\eta_0}\widehat{\cal{H}}(\xi,\eta') \, \frac{(\xi-\eta)^\halb}{ (\xi- \eta')}\,\left( \frac{1}{\sqrt{\xi-\eta'}}-\frac{1}{\sqrt{-\eta'}} \right)\,d\,\eta' \right) .
\end{align}
Consequently,
\begin{align}
\widehat{\cal{H}}(\xi,\eta) &\leq \left(\frac{-\eta}{-\eta_0}\right)^\halb \widehat{\cal{H}}(\xi,\eta_0) + c \text{Flux}^\halb(\eta_0) 
+ c \int^{\eta}_{\eta_0} \widehat{\cal{H}}(\xi,\eta') \frac{\xi}{\eta' (\xi-\eta')} \, d\, \eta'
\end{align}
\begin{align}
 \widehat{\cal{H}}(\xi,\eta_0) & = \sup_{\xi\leq \xi' \leq 0}\sqrt{\xi'-\eta_0}\,\, \widehat{\cal{B}}(\xi',\eta_0) \notag \\
& \leq \sup_{\xi\leq \xi' \leq 0}\sqrt{\xi'-\eta_0} \,\, \sup_{\xi\leq \xi' \leq 0} \widehat{\cal{B}}(\xi',\eta_0) \notag \\
& \leq c(\eta_0) \sqrt{-\eta_0} 
\end{align}
where we have used the fact that $u$ is regular away from the axis so that $\widehat{\cal{B}}(\xi,\eta_0)$ is finite. So,
\begin{align} \label{before_gronwall}
\widehat{\cal{H}}(\xi,\eta) &\leq c(\eta_0) \sqrt{-\eta} + c \text{Flux}^\halb(\eta_0) 
+ c \int^{\eta}_{\eta_0} \widehat{\cal{H}}(\xi,\eta') \frac{\xi}{\eta' (\xi-\eta')} \, d\, \eta'.
\end{align}
Let us now use the Gronwall's lemma to convert the implicit estimate in \eqref{before_gronwall} to an explicit one, 
for $\eta \in (\eta_0, \frac{\xi}{\lambda'})$
where $ \lambda' \fdg = \frac{1-\lambda}{1+ \lambda} < 1 $
\begin{align}
 \widehat{\cal{H}}(\xi,\eta) & \leq  \sqrt{-\eta} c(\eta_0) + c \, \text{Flux}^\halb(\eta_0) \notag \\
& \quad + c \int^{\eta}_{\eta_0} \left(  \sqrt{-\eta} c(\eta_0) + c \, 
\text{Flux}^\halb(\eta_0)\right) \left( \frac{\xi}{\eta'(\xi-\eta')} \right) e^{\int^\eta_{\eta'}\frac{\xi}{\eta'' (\xi -\eta'')} \, d\, \eta''} \, d\,\eta'.
\end{align}
We have for $\eta_0 \leq \eta' \leq \eta $ and setting $\xi = \lambda' \eta$,
\begin{align*}
 \int^\eta_{\eta'} \frac{\xi}{\eta'' (\xi -\eta'')} d\, \eta'' = & \log \frac{\eta(\lambda' \eta -\eta')}{\eta' (\lambda' \eta -\eta)} \\
\leq & \log \frac{1}{1-\lambda'}.
\end{align*}
For any $\epsilon > 0$ we can choose an $\eta_0$ small enough such that $c\text{Flux}^\halb(\eta_0)< \frac{\epsilon}{2}$. Furthermore
one can choose $\eta \in (\eta_0, 0)$ small enough such that   $c(\eta_0)\sqrt{-\eta} < \frac{\epsilon}{2}$. \\
So we have $\widehat{\cal{H}}(\xi,\eta) < \epsilon$ for $\eta_0 <\eta<0$ small enough.
Then, $\widehat{\cal{B}}(\xi,\eta) \leq \frac{\widehat{\cal{H}}(\xi,\eta)}{\sqrt{\xi-\eta}} \leq \frac{\epsilon}{\sqrt{\xi-\eta}}.$
Now going back to the flux integrals $\int_{\ptl \cal{S}_1} e^{-\cal{G}}(\mathbf{e+m}) \bar{\mu}_\eta $ and $\int_{\ptl \cal{S}_2} e^{-\cal{F}}(\mathbf{e+m}) \bar{\mu}_\xi $
in \eqref{div_tube}, we have
\begin{align}
 \int_{\ptl \cal{S}_1} e^{-\cal{G}}(\mathbf{e+m}) \bar{\mu}_\eta &\leq c \int^0_\xi \widehat{\cal{B}}(\xi',\eta) \,d\,\xi' \notag \\
  &\leq \epsilon \int^0_\xi \frac{1}{(\xi'-\eta)}\,\, d\, \xi' \notag \\
 & =  \eps  \int^0_{\lambda' \eta} \frac{1}{(\xi'-\eta)}\,\, d\, \xi' \notag s\\
& = \eps \log\left( \frac{-\eta}{(\lambda'  -1)\eta} \right) \notag \\
& = \eps \log \frac{1}{\lambda' -1} \notag \\
& < c \eps
\end{align}
and 
\begin{align}
 \halb \int_{\ptl \cal{S}_2} r e^{2z-\mathcal{F}} (\mathbf{e-m})\,\,  d\, \eta \wedge d\, \theta &= \text{Flux} (\mathbf{P_X})(\eta_0,\eta) \notag \\
& < \eps
\end{align}
for $\eta_0, \eta$ small enough.
Finally, since $\int_{\ptl \cal{S}_1}$ and $\int_{{\ptl \cal{S}_1}} \to 0$ in \eqref{div_tube}
we conclude that $ E^O_{\text{ext}} (\tau) \to 0$ as $\tau \to 0.$
\end{proof}
\begin{lem}[Non-concentration of integrated kinetic energy]\label{kin}
Let the kinetic energy be defined as 
\[\mathbf{e}_{\text{kin}} \fdg = \halb e^{-2\Omega} u_t^2 \]
then the spacetime integral of $\mathbf{e}_{\text{kin}}$ does not concentrate in the past null cone
of $O$, i.e.,
\begin{align*}
\frac{1}{r_2(\tau)}\int_{K_\tau}  \mathbf{e}_{\text{kin}} \,\, \bar{\mu}_g  \to 0 \,\, \text{as} \,\, \tau \to 0
\end{align*}
where $r_2(\tau)$ is the radial function defined as in Lemma \ref{ann}.
\end{lem}
\begin{proof}
Consider the vector field $\mathbf{P_{X_3}}$ and its divergence, 
\begin{align*}
\mathbf{P_{X_3}} =& e^{(1-k)\gamma} \left( -r\,\mathbf{m}\,\mathbf{X_1} + r (\mathbf{e}-\mathbf{f}) \mathbf{X_2}\right) \\
=& -re^{(1-k)\gamma - \Omega} \,\mathbf{m}\, \ptl_t + re^{-k\gamma}(\mathbf{e} -\mathbf{f})\ptl_r \\
\nabla_\nu  \mathbf{P}^\nu_{\mathbf{X_3}} =&\, \halb \leftexp{(\mathbf{X_3})} {\mbo{\pi}}_{\alpha \beta} \mathbf{T}^{\alpha \beta} \\
= & \,e^{- 2 \Omega} u_t^2
\end{align*}
Using the Stokes theorem as in \eqref{stokes_gen} for $\mathbf{P}_{\mathbf{X_3}}$

\begin{align}\label{stokes_px3}
 \int_{K(\tau,s)} \nabla_\nu  \mathbf{P}^\nu_\mathbf{X_3}\, \bar{\mu}_g  & = \int_{\Sigma^O_{s}} e^{\Omega}  \mathbf{P}^t_\mathbf{X_3} 
\, \bar{\mu}_q
- \int_{\Sigma^O_{\tau}}e^{\Omega} \mathbf{P}^t_\mathbf{X_3}  \, \bar{\mu}_q
+ \text{Flux}(\mathbf{P_{X_3}}) (\tau,s) \notag \\
\intertext{that is}
 \int_{K(\tau,s)} e^{- 2 \Omega} u_t^2 \,\bar{\mu}_g  & = - \int_{\Sigma^O_{s}} r  e^{\gamma} \mathbf{m} \, \bar{\mu}_q
+ \int_{\Sigma^O_{\tau}} r  e^{\gamma} \mathbf{m} \, \bar{\mu}_q
+ \text{Flux}(\mathbf{P_{X_3}}) (\tau,s) 
\end{align}
where,
\begin{align*}
 \text{Flux}(\mathbf{P_{X_3}}) (\tau,s) = & \int_{C(\tau,s)} d\,\xi(\mathbf{P}_{\mathbf{X_3}})\, \bar{\mu}_\xi \\
 =& \,\halb \int_{C(\tau,s)}r e^{\gamma-\cal{F}} \mathbf{(e-m-f)}  \, \bar{\mu}_\xi \\
 \leq &\, c \, r_2(\tau) \int_{C(\tau,s)} (\mathbf{e}-\mathbf{m}) \, \bar{\mu}_\xi \\
= &\, -c \, r_2(\tau) \text{Flux}(\mathbf{P_{X_1}}) (\tau,s). \\
\end{align*}
We have,
\begin{align*}
 \int_{K(\tau,s)} e^{- 2 \Omega} u_t^2 \,  \bar{\mu}_g  & \leq  \int_{\Sigma^O_{s}} r  e^{\gamma} \mathbf{e} \,   \bar{\mu}_q
+ \int_{\Sigma^O_{\tau}} r  e^{\gamma} \mathbf{e} \,\, \bar{\mu}_q
- c \, r_2(\tau) \text{Flux}(\mathbf{P_{X_1}}) (\tau,s) \\
& \leq c r_2(s)\int_{\Sigma^O_s} \mathbf{e} \, \bar{\mu}_q
+ \int_{\Sigma^O_{\tau}} r  e^{\gamma} \mathbf{e} \, \bar{\mu}_q
- c \, r_2(\tau) \text{Flux}(\mathbf{P_{X_1}}) (\tau,s)  \\
\end{align*}
Now let $s \to 0$ in \eqref{stokes_px3}, we get
\begin{align*}
\frac{1}{ r_2(\tau) } \,\int_{K(\tau)} e^{-2 \Omega} u_t^2 \,  \bar{\mu}_g  &\leq \, \frac{1}{ r_2(\tau) }\, \int_{\Sigma^O_{\tau}} r  e^{ \gamma} \mathbf{e} \, \bar{\mu}_q
- c \,\text{Flux}(\mathbf{P_{X_1}}) (\tau) \\
\intertext{therefore,}
\frac{1}{ r_2(\tau) } \int_{K(\tau)} e^{- 2 \Omega} u_t^2 \, \bar{\mu}_g & \leq c \frac{1}{ r_2(\tau) } \int_{B_{r_2}(\tau)} r  e^{ \gamma} \mathbf{e} \,\bar{\mu}_q
- c \,\,  \text{Flux}(\mathbf{P_{X_1}}) (\tau) \\
& = c\, \frac{1}{ r_2(\tau) } \left( \int_{B_{r_2(\tau)}} r  e^{ \gamma} \mathbf{e} \, \bar{\mu}_q + 
\int_{B_{r_2(\tau)}\setminus B_{r_1(\tau)}} r  e^{ \gamma} \mathbf{e} \, \bar{\mu}_{q} \right) \\
& \quad - c\,  \text{Flux}(\mathbf{P_{X_1}}) (\tau) \\
& \leq c \,\lambda\, E_0 +  c\,  E^O_{\text{ext}} (\tau) - c\,  \text{Flux}(\mathbf{P_{X_1}}) (\tau).
\end{align*}
For any $\eps > 0$ we can choose $\lambda$ small enough so that the first term $ < \frac{\eps}{3}$, then we can make $\tau$ small enough
so that $E^{O}_{\text{ext}} (\tau) < \frac{\eps}{3} $ and $ |\text{Flux}(\mathbf{P_{X_1}}) (\tau)| < \frac{\eps}{3} $
as discussed previously.
\end{proof}
\section{Non-Concentration of Energy with Grillakis Condition}
Recall the expression for energy 
\begin{align}\label{newenergy}
 \mathbf{e} =& \mathbf{T}(\mathbf{X_1},\mathbf{X_1}) \notag\\
 =& \halb \left(\Vert \mathbf{X_1} (U)\Vert^2_h + \Vert \mathbf{X_2} (U)\Vert^2_h +  \Vert m(U)\Vert^2_h  \right) \notag\\
 =& \halb \left(e^{-2\Omega}u_t^2 + e^{-2\gamma}u_r^2+ \frac{f^2(u)}{r^2} \right)
\end{align}
where $m = \frac{1}{r}\ptl_\theta$ as defined in \eqref{coonulltriad}.
In Lemma \ref{kin} we proved that the spacetime integral of $e^{-2\Omega}u_t^2$ does not concentrate in the past null cone of $O$.
In the following lemma we shall prove that the spacetime integral of rotational potential energy i.e.,
\[ \int_{K_\tau} \Vert m(U) \Vert^2_h \bar{\mu}_g = \int_{K_\tau} \frac{f^2(u)}{r^2} \bar{\mu}_g = \int_{K_\tau} \mathbf{f} \bar{\mu}_g \]
does not concentrate.
The proof is based on the condition that the target manifold $(N,h)$ satisfies the Grillakis
condition
\begin{align}
f_s(s)f(s)+f^2(s) >0  \,\,\,\text{for}\,\,\, s>0.
\end{align}
This condition is weaker than the condition that $(N,h)$ is geodesically convex \eqref{geoconvex}.
\begin{lem}[Non-concentration of integrated rotational potential energy]\label{fterm}
Let $(N,h)$ be the target manifold satisfying 
 \begin{align}
 f(u)f_u(u) u + f^2(u) > 0 \,\,\,\text{for}\,\,\,u>0
 \end{align}
 then the spacetime integral of rotational potential energy does not concentrate i.e.,
 \begin{align}
 \int_{K_\tau} \mathbf{f} \,\bar{\mu}_g \to 0 \,\,\,\text{as}\,\,\, \tau \to 0.
 \end{align}
\end{lem}

\begin{proof}
 Recall the momentum vector field $\mathbf{P}_\mathbf{X_4}$
\begin{align*}
 \mathbf{P}_\mathbf{X_4} & = -e^{\gamma -\Omega} r^a \, \mathbf{m} \ptl_t + r^a \,(\mathbf{e}-\mathbf{f}) \ptl_r  \\
& = e^{\gamma} r^a (-\mathbf{m} \mathbf{X_1} + (\mathbf{e}-\mathbf{f}) \mathbf{X_2})
\end{align*}
and the divergence from \eqref{intro_px4}
\begin{align*}
\nabla_\nu  \mathbf{P}^\nu_\mathbf{X_4} &= \halb \left (  (1+a) r^{a-1}  \right) e^{-2\Omega} u^2_t
+ \halb \left ( (a-1) r^{a-1}  \right) e^{-2\gamma} u^2_r \\
& \quad + \halb \left ( (1-a) r^{a-1}  \right) \frac{f^2(u)}{r^2}.
 \end{align*}
Let now us construct the vector $\mathbf{P}_\kappa^\nu$ such that
\begin{align*}
 \mathbf{P}_\kappa^\nu \fdg = \kappa u^\nu u - \kappa^\nu \frac{u^2}{2},
\end{align*}
where $\kappa \fdg = \frac{1-a}{2}r^{a-1}$ for $a \in (\halb, 1)$
then the divergence,
\begin{align*}
 \nabla_\nu  \mathbf{P}^\nu_\kappa &= \nabla_\nu (\kappa u^\nu u) - \nabla_\nu (\kappa^\nu \frac{u^2}{2}) \\
&= \kappa (\square u)u + \kappa u^\nu u_\nu + u^\nu \kappa_\nu u - ( \square \kappa) \frac{u^2}{2} - \kappa^\nu u u_\nu \\
&= \kappa \frac{f(u)f_u(u)u}{r^2} + \kappa u^\nu u_\nu - ( \square \kappa) \frac{u^2}{2}
\end{align*}
and
\begin{align*}
 \square \kappa &= e^{-2\gamma} \left(\kappa_{rr} + \frac{\kappa_r}{r} + (\Omega_r -\gamma_r)\kappa_r \right) \\
&= e^{-2\gamma}r^{a-3}\frac{(1-a)^2}{2} \left(1-a + r^2 \mbo{\alpha}e^{2\gamma}\, \mathbf{f} \right). \\
\end{align*}
Let us define a vector $\mathbf{P}^\nu_{\text{tot}}$ such that
\begin{align*}
 \mathbf{P}^\nu_{\text{tot}} \fdg =\mathbf{P}^\nu_{\mathbf{X}_4} + \mathbf{P}^\nu_\kappa. \\
\end{align*}
Therefore,
\begin{align*}
 \nabla_\nu  \mathbf{P}^\nu_{\text{tot}} &= \nabla_\nu  \mathbf{P}^\nu_{\mathbf{X_4}}+ \nabla_\nu  \mathbf{P}^\nu_\kappa \\
&= \kappa \frac{f(u)f_u(u)u}{r^2} + a r^{a-1} e^{-2\Omega} u_t^2 + \kappa \,\mathbf{f} - e^{-2\gamma}\frac{(1-a)^2}{2} r^{a-1} \left(1-a 
+ r^2 \mbo{\alpha}e^{2\gamma} \mathbf{f} \right) \frac{u^2}{r^2}.
\end{align*}
Applying the Stokes' theorem on $K(\tau,s),$
\begin{align}\label{stokes_px4}
 \int_{K(\tau,s)} \nabla_\nu  \mathbf{P}_{\text{tot}} ^\nu\,  \bar{\mu}_g  = \int_{\Sigma^O_{s}} e^{\Omega} \, \mathbf{P}_{\text{tot}}^{\,t} \,\bar{\mu}_q
- \int_{\Sigma^O_{\tau}}e^{\Omega} \,\mathbf{P}_{\text{tot}} ^{\,t}  \,\bar{\mu}_q
+ \text{Flux}(\mathbf{P}_{\text{tot}}) (\tau,s). 
\end{align}

\begin{align} \label{px4_s}
 \int_{\Sigma^O_{s}} e^{\Omega} \, \mathbf{P}_{\text{tot}}^{\,t} \, \bar{\mu}_q &= 
 -\int_{\Sigma^O_{s}} \mathbf{m}\, r^a e^{\gamma} +e^{\Omega} \kappa\, u^t \, u \,\bar{\mu}_q \notag\\
&\leq \int_{\Sigma^O_{s}} \mathbf{e}\,r^a e^{\gamma} +|  \, e^{-\Omega} \, u_t|\, |\kappa\,u| \,\bar{\mu}_q, \notag \\
\intertext{applying the Cauchy-Schwarz inequality, we get}
\int_{\Sigma^O_{s}} e^{\Omega} \, \mathbf{P}_{\text{tot}}^{\,t} \, \bar{\mu}_q
&\leq c\,r^a_2(s)\int_{\Sigma^O_{s}} \mathbf{e} \,\bar{\mu}_q   
+ \frac{1-a}{2}\left( \int_{\Sigma^O_{s}} e^{-2\Omega} \, u^2_t r^{2a}\, 
\bar{\mu}_q \right)^\halb \left(\int_{\Sigma^O_{s}} \frac{u^2}{r^2}\,   \bar{\mu}_q \right)^\halb \notag \\
& \leq  c\,r^a_2(s) \notag \\
 & \to 0
\end{align}
as $s \to 0$.
Similarly, the second term in \eqref{stokes_px4} can be estimated as 
\begin{align}\label{px4_tau}
-\int_{\Sigma^O_{\tau}} e^{\Omega} \, \mathbf{P}_{\text{tot}}^{\,t} \,  \bar{\mu}_q \leq c\, \int_{\Sigma^O_{\tau}} \mathbf{e}\,r^a\,  \bar{\mu}_q 
+ c \left( \int_{\Sigma^O_\tau} \mathbf{e}\,r^{2a}\,\bar{\mu}_q \right)^{\halb}
\end{align}
The flux of $\mathbf{P}_{\text{tot}}$ though the null surface $C(\tau,s)$ can be written as
\begin{align}\label{ptot_genflux}
 \text{Flux}(\mathbf{P}_{\text{tot}}) (t,s) & = \int_{C(\tau,s)} d\, \xi (\mathbf{P}_{\text{tot}}) \bar{\mu}_\xi \notag\\
&= \int_{C(\tau,s)} d\, \xi (\mathbf{P}_{\mathbf{X_4}}) \bar{\mu}_\xi + 
\int_{C(\tau,s)} d\, \xi (\mathbf{P}_{\kappa}) \bar{\mu}_\xi.
\end{align}
Let us consider the terms in the right side of \eqref{ptot_genflux} individually.
We have
\begin{align} \label{ptot_flux}
 \text{Flux}(\mathbf{P}_{\mathbf{X_4}})(\tau,s) &= \int_{C(\tau,s)} d\, \xi (\mathbf{P}_{\mathbf{X_4}}) 
 \bar{\mu}_\xi \notag\\
&=\halb \int_{C(\tau,s)} e^{\gamma-\cal{F}} r^a (\mathbf{e}-\mathbf{m}-\mathbf{f}) \bar{\mu}_\xi \notag \\
& \leq -c r^a_2(\tau)\,\text{Flux}(\mathbf{P}_{\mathbf{X_1}})(\tau,s) \notag\\ \notag\\
\intertext{and}
\text{Flux}(\mathbf{P}_{\kappa})(\tau,s) &= \int_{C(\tau,s)} d\, \xi (\mathbf{P}_{\kappa}) \bar{\mu}_\xi \notag \\
& = \halb \int_{C(\tau,s)}  \left( u \left(-\mathbf{X_1}(u) + \mathbf{X_2}(u)\right) 
+  \halb \kappa e^{-(\gamma +\cal{F})} (1-a)r^{-1} u^2  ) \right)\,\bar{\mu}_\xi \notag\\
&= \halb \int_{C(\tau,s)}  \left( u \left(-\mathbf{X_1}(u) + \mathbf{X_2}(u)\right) 
+  e^{-(\gamma +\cal{F})} \frac{(1-a)^2}{4} \,\frac{ u^2}{r^2}r^a   \right)\,\bar{\mu}_\xi \notag\\
&\leq \halb \int_{C(\tau,s)}  \left( u \left(-\mathbf{X_1}(u) + \mathbf{X_2}(u)\right) 
+  c\, \frac{(1-a)^2}{4} \mathbf{f}\, r^a \,e^{-\cal{F}} \right)\,\bar{\mu}_\xi \notag\\
&\leq\halb \int_{C(\tau,s)}  \left( u \left(-\mathbf{X_1}(u) + \mathbf{X_2}(u)\right) 
+  c\, \frac{(1-a)^2}{2} (\mathbf{e-m})\, r^a \,e^{-\cal{F}} \right)\,\bar{\mu}_\xi. \\
\intertext{Using the Cauchy-Schwarz inequality, \eqref{ptot_flux} can be estimated as}
\text{Flux}(\mathbf{P}_{\kappa})(\tau,s) & \leq c r^a_2(\tau) \left(\int_{C\tau,s)} (\mathbf{e} - \mathbf{m} ) \bar{\mu}_\xi \right)^\halb +
c r^a_2(\tau) \left(\int_{C\tau,s)} (\mathbf{e} - \mathbf{m} ) \bar{\mu}_\xi \right)\notag\\
& \leq - c r_2^a(\tau) \text{Flux}^\halb(\mathbf{P}_\mathbf{X_1})(\tau,s) - c r_2^a(\tau) \text{Flux}(\mathbf{P}_\mathbf{X_1})(\tau,s).
\end{align}
Therefore,
\begin{align}
\text{Flux}(\mathbf{P}_{\text{tot}}) (t,s) 
&\leq - c r_2^a(\tau) \text{Flux}^\halb (\mathbf{P}_\mathbf{X_1})(\tau,s)- c r_2^a(\tau) \text{Flux}(\mathbf{P}_\mathbf{X_1})(\tau,s).
\end{align}
If $f(u)f_u(u)u + f^2(u) >0$ for $u>0$, we can choose `$a$' close enough to $1$ such that 
\[ f(u)f_u(u)u + f^2(u) \geq e^{-2\gamma}(1-a)^2 u^2 \]
so that 
\[f(u)f_u(u)u + f^2(u) - \frac{e^{-2\gamma}(1-a)^2}{2} u^2 \geq \frac{e^{-2\gamma}(1-a)^2}{2} u^2. \]
Now, if we go back to the Stokes' theorem \eqref{stokes_px4} and use the estimates \eqref{px4_s}, \eqref{px4_tau} and 
\eqref{ptot_flux}, we get
\begin{align*}
 &a \int_{K(\tau,s)} e^{-2\Omega} u^2_t r^{a-1} d\bar{\mu}_g  + \frac{(1-a)^2}{2} \int_{K(\tau,s)} e^{-2\gamma} \frac{u^2}{r^2} r^{a-1} \,\bar{\mu}_g \\
&\leq c \, r^a_2(s) + c \, \int_{\Sigma^O_{\tau}} \mathbf{e}\,r^a\,  \bar{\mu}_q 
+ c\left(\int_{\Sigma^O_{\tau}} \mathbf{e}\,r^{2a}\,  \bar{\mu}_q\right)^{\halb} \\
&\quad -c \, r^a_2(\tau) \text{Flux}^\halb(\mathbf{P}_\mathbf{X_1})(\tau,s)  -c \, r^a_2(\tau) \text{Flux}(\mathbf{P}_\mathbf{X_1})(\tau,s)
\end{align*}
as $s \to 0$ we get,
\begin{align}\label{simpler_stokes_ptot}
  &a \int_{K(\tau)} e^{-2\Omega} u^2_t r^{a-1} \,\bar{\mu}_g   + \frac{(1-a)^2}{2} \int_{K(\tau)} e^{-2\gamma} \frac{u^2}{r^2} r^{a-1} \,\bar{\mu}_g \notag \\
&\leq c \, \int_{\Sigma^O_{\tau}} \mathbf{e}\,r^a \, \bar{\mu}_q + c \left(\int_{\Sigma^O_{\tau}} \mathbf{e}\,r^{2a} \, \bar{\mu}_q\right)^{\halb} \notag \\
&\quad -c\,r^a_2(\tau)\text{Flux}(\mathbf{P}_\mathbf{X_1})(\tau) -c \, r^a_2(\tau) \text{Flux}^\halb(\mathbf{P}_\mathbf{X_1})(\tau,s).
\end{align}
In \eqref{simpler_stokes_ptot}, we can estimate
\begin{align} \label{energy_ra}
r^{-a}_2(\tau)\int_{\Sigma^O_{\tau}} \mathbf{e}\,r^a\,  \bar{\mu}_q 
&= r^{-a}_2(\tau)\left( \int_{B_{r_1(\tau)}} \mathbf{e}\,r^a\,\bar{\mu}_q + \int_{B_{r_2(\tau)} \setminus B_{r_1(\tau)}} 
\mathbf{e}\,r^a\,\bar{\mu}_q  \right) \notag\\
 &\leq r^{-a}_2(\tau)\left(r^a_1(\tau) \int_{B_{r_1(\tau)}} \mathbf{e}\,\bar{\mu}_q + r^a_2(\tau) \int_{B_{r_2(\tau)} \setminus B_{r_1(\tau)}} 
\mathbf{e}\,\bar{\mu}_q  \right) \notag\\
 & \leq c \lambda^a \, E_0 +  c\,E^O_{\text{ext}} (\tau)
\end{align}
and
\begin{align}\label{energy_ra_sqrt}
 r^{-a}_2(\tau)\left(\int_{\Sigma^O_{\tau}} \mathbf{e}\,r^{2a} \, \bar{\mu}_q\right)^{\halb} 
 &= r_2^{-a}(\tau) \left(\int_{B_{r_1(\tau)}} \mathbf{e}\,r^{2a}\,\bar{\mu}_q + \int_{B_{r_2(\tau)} \setminus B_{r_1(\tau)}} \mathbf{e}\, r^{2a}\,\bar{\mu}_q \right)^{\halb} \notag\\
 &\leq \left(\left(\frac{r_1(\tau)}{r_2(\tau)}\right)^{2a}\int_{B_{r_1(\tau)}} \mathbf{e}\,\bar{\mu}_q 
 + \int_{B_{r_2(\tau)}} \mathbf{e}\,\bar{\mu}_q \right)^{\halb} \notag\\
 &\leq \left( \lambda^{2a} E_0+ E^O_{\text{ext}}(\tau) \right)^{\halb}.
\end{align}
Hence, in view of \eqref{energy_ra}, \eqref{energy_ra_sqrt}, Corollary \ref{flux_px1_zero} and Lemma \ref{kin},
we can choose $\lambda$ and $\tau$ in \eqref{simpler_stokes_ptot} small enough so that
\[ \frac{1}{r^a_2(\tau)}\int_{K(\tau)}\frac{u^2}{r^2} r^{a-1} \,\bar{\mu}_g < \eps \]
for any $\eps >0$.
Furthermore, from equation 2.11 in \cite{jal_tah1} there exists a real constant $c$ dependent only
on the initial energy $E_0$ such that
\begin{align}
 \frac{1}{c} u^2 \leq f^2(u) \leq c u^2.
\end{align}
Consequently,
\[ \Vert m(U) \Vert_h^2 \equiv \mathbf{f} \leq \frac{u^2}{r^2}, \]
where $m = \frac{1}{r}\ptl_\theta$ as defined in \eqref{coonulltriad}.
Therefore it follows that
 \begin{align*}
 \frac{1}{r^a_2(\tau)}\int_{K_\tau} \mathbf{f}\,\,r^{a-1}\, \,\bar{\mu}_g \to 0 \,\,\,\text{as}\,\,\, \tau \to 0.
 \end{align*}

\end{proof}
The remaining term in \eqref{newenergy} is $\Vert \mathbf{X_2}(U) \Vert^2_h = e^{-2\gamma}u_r^2$. We prove the
non-concentration of this term by using the Stokes' theorem on the divergence of $\mathbf{P}_{\mathbf{X_4}}$. 

\begin{cor}\label{pot}
Under the usual notation, the spacetime integral of radial potential energy in the past null cone of $O$ does not concentrate
\begin{align}
 \frac{1}{r^a_2(\tau)}\int_{K_{\tau}} e^{-2\gamma}u^2_r r^{a-1} \,\bar{\mu}_g \to 0\,\,\, \text{as}\,\,\, \tau \to 0
 \end{align}
\end{cor}
\begin{proof}
Let us again apply the Stokes' theorem for the $\bar{\mu}_g$-divergence of $\mathbf{P}_{\mathbf{X_4}}$
\begin{align*}
 \int_{K(\tau,s)} \nabla_\mu  \mathbf{P}_{\mathbf{X_4}} ^\mu\, \bar{\mu}_g  = \int_{\Sigma^O_{s}} e^{\Omega} \, \mathbf{P}^t_{\mathbf{X_4}} \, \bar{\mu}_q
- \int_{\Sigma^O_{\tau}}e^{\Omega} \,\mathbf{P}^t_{\mathbf{X_4}} \, \bar{\mu}_q
+ \text{Flux}(\mathbf{P}_{\mathbf{X_4}}) (\tau,s) 
\end{align*}
therefore, as $s \to 0$
\begin{align*}
 \int_{K(\tau)} e^{-2\gamma} u^2_r \,r^{a-1}\, \bar{\mu}_g &\leq c \int_{K(\tau)} \left(e^{-2\Omega} u^2_t + \frac{f^2(u)}{r^2} \right)r^{a-1}\, \bar{\mu}_g 
  + \int_{\Sigma^O_{\tau}}\mathbf{e}\,r^a  \, \bar{\mu}_q + r^a_2(\tau)\text{Flux}(\mathbf{P}_{\mathbf{X_1}}) (\tau). \\
 \end{align*}
Hence,
\begin{align*}
\frac{1}{r^a_2(\tau)}\int_{K(\tau)} e^{-2\gamma} u^2_r \,r^{a-1}\, \bar{\mu}_g < \eps
\end{align*}
for $\tau$ small enough.

\end{proof}

\begin{thm}[Non-concentration of energy] \label{theorem_ener_nonconc}
Let $(M,g,U)$ be a smooth, globally hyperbolic, equivariant maximal development of smooth, compactly supported equivariant 
initial data set $(\Sigma,q,\mathbf{K}, U_0, U_1)$ with finite initial energy and satisfying the constraint equations,
and let $(N,h)$ be a rotationally symmetric, complete, connected Riemannian manifold satisfying
\[ f_s(s)f(s)+f^2(s) >0  \,\,\,\text{for}\,\,\, s>0 \] 
and 
\[ \int_0^u f(s) \,d\,s \to \infty \,\,\,\text{as}\,\,\, u \to \infty, \] 
then the energy of the Einstein-wave map system \eqref{ewmcauchy_equi_new} cannot concentrate,
i.e., $E^O (t) \to 0$, where $O$ is the first (hypothetical) singularity of $M$.
\end{thm}

\begin{proof}
If we collect the terms from Lemmas \ref{kin}, \ref{fterm} and Corollary \ref{pot}, we get
\[  \frac{1}{r^a_2(\tau)} \int_{K(\tau)} \mathbf{e} \,r^{a-1} \, \bar{\mu}_g \to 0 \]
as $\tau \to 0.$
But then,
\begin{align}\label{integrated_energy}
 \frac{1}{r^a_2(\tau)} \int_{K(\tau)} \mathbf{e}\, r^{a-1} \, \bar{\mu}_g
& \geq c\,\frac{1}{r^a_2(\tau)} \int_{K(\tau)} \mathbf{e} \, r^{a-1} \, \bar{\mu}_q \,d\,t \notag \\
&\geq c \frac{1}{r_2(\tau)} \int_{K(\tau)} \mathbf{e} \, \bar{\mu}_q \,d\,t \notag \\
&\to 0
\end{align}
as $\tau \to 0$ from the Sandwich theorem.
We claim that there exists a sequence $ \{\tau_i\} _i$ such that 
\begin{align}\label{energy_noncon_claim}
 \int_{\Sigma^O_{\tau_i}} \mathbf{e}\, \bar{\mu}_q \to 0
\end{align}
as $ \{ \tau_i\} _i \to 0 $.
Let us prove the claim by contradiction. Suppose there exists no sequence
such that \eqref{energy_noncon_claim} holds true. Then there exists an $\eps > 0$ 
such that 
\[ \int_{\Sigma^O_{\tau}} \mathbf{e} \,\bar{\mu}_q > \eps \]
for all $\tau \in (-1,0).$
Consequently,
\[ \frac{1}{|\tau|}\int_{\Sigma^O_{\tau}} \mathbf{e} \,\bar{\mu}_q \, d\,t > \eps  \]
This implies,
\begin{align}\label{normalized_spatial}
 \frac{1}{r_2(\tau)} \int_{K(\tau)} \mathbf{e} \,  \bar{\mu}_q \,d\,t > \eps
\end{align}
for all $\tau \in [-1,0)$.
This contradicts \eqref{integrated_energy}.
Hence, there exists a $ \{\tau_i\} _i$ such that 
\begin{align}
E^O (\tau_i) =  \int_{\Sigma^O_{\tau_i}} \mathbf{e}\,\bar{\mu}_q \to 0.
\end{align}
But $ E^O (\tau)$ is monotonic with respect to $\tau$, therefore
\[ E^O (\tau) \to 0\]
for all $\tau \to 0$ i.e., $E^O_{\text{conc}} =0$. This concludes the proof.
\end{proof}
\chapter{Outlook}

As shown in Chapter 2, vacuum Einstein's equations in $3+1$ dimensions for spacetimes with 
1-parameter isometry group can be interpreted as
self-gravitating wave maps in $2+1$ dimensions with the hyperbolic 2-plane as the target manifold. 
Therefore, any progress in understanding large energy global existence of critical self-gravitating wave maps is valuable 
in understanding the global behavior of Einstein's equations. This makes the critical self-gravitating wave maps problem
a fundamental problem in general relativity. Here, the rich variety of techniques developed for critical wave maps on
Minkowski background should be used to full advantage.

Earlier we spoke about the non-concentration of energy for self-gravitating wave maps under some
conditions on the target manifold (Grillakis condition). One may hope to extend this result further by weakening the conditions on
the target manifold and simultaneously establish a blow up criterion. In this context one may formulate the following 
conjecture
\begin{description}
\item[(C3) Rescaled convergence to a nontrivial harmonic map] Let us assume that an energy critical self-gravitating wave map blows up, 
then there exists a blow-up sequence of rescaled energy critical self-gravitating wave maps that converges strongly to a 
nontrivial harmonic map\footnote{a static solution of the Einstein wave map system} in $H^1_{\text{loc}}.$
\end{description}
Consequently, one can view the existence of a nontrivial static solution as a blow-up criterion for wave maps. If the energy of the 
wave map or the geometry of the base and the target manifolds does not allow the existence of a static solution then, by contradiction,
one can rule out the formation of blow up. This has been resolved on the flat background by Struwe \cite{struwe_equi}
for the case of equivariant wave maps and later followed by Sterbenz and Tataru \cite{sterb_tata_main,sterb_tata_long} for general 
wave maps with compact target. In the above context, the result which says that the kinetic energy density integrated over the backward 
null cone of a point does not concentrate, plays a vital role. This strategy seems to be the most promising in addressing the conjecture
\textbf{(C1)} for critical self-gravitating wave maps.

\subsection*{Small Energy Geodesic Completeness}
The resolution of \textbf{(C'2)} can be based on a Strichartz estimate. However, the fact that one is dealing with a dynamical
background causes additional obstacles which need to be overcomed. In the case of equivariant symmetry one hopes that the 
conservation law comes to the rescue. In the case of Minkowski background, \textbf{(C2)} has been proved using a version 
of the Strichartz estimate after reducing the wave maps equation to 4+1 critical wave equation with power nonlinearity. A similar transformation
can be thought of for the case of self-gravitating equivariant wave maps.
Define $v(t,r)$ such that $u = rv$, so we have

\begin{align*}
 u_t = r v_t,\,\,& u_{tt} = rv_{tt}, \\
 u_r= rv_r+ v,\,\,& u_{rr} = rv_{rr} + 2v_r. 
\end{align*}
then $\leftexp{3}{\square}_g u$ can be rewritten as
\begin{align*}
 \leftexp{3}{\square}_g u =& -e^{-2\Omega}(u_{tt} + (\gamma_t-\Omega_t)u_t) + e^{-2 \gamma}(u_{rr} + \frac{u_r}{r} + (\Omega_r - \gamma_r)u_r) \\
 =& r\bigg(-\,e^{-2\Omega}(v_{tt} + (\gamma_t-\Omega_t)v_t) + e^{-2 \gamma}(v_{rr}+(3r^{-1}+ \Omega_r - \gamma_r)v_r+ 
 (r^{-1}+\Omega_r - \gamma_r)vr^{-1}) \bigg) \\
 =& r \bigg( \leftexp{3}{\square}_g v + e^{-2\gamma} (2 v_r r^{-1} + (r^{-1}+ \Omega_r -\gamma_r)vr^{-1}) \bigg).
\end{align*}
Therefore \eqref{wmequi} translates to
\begin{align*}
\leftexp{3}{\square}_g v = & \frac{f_u(u) f(u)}{r^3} - e^{-2\gamma} \left( 2v_r r^{-1} + (r^{-1}+ \Omega_r -\gamma_r)vr^{-1}\right) \\
\end{align*}
Consider a manifold with the following metric
\begin{align*}
 d\,s^2_{\mbo{g}} = -e^{2\Omega} d\,t^2 + e^{2\gamma} d\,r^2 + r^2\, d\,\omega^2_{\mathbb{S}^3}
\end{align*}
where 
\[ d\,\omega^2_{\mathbb{S}^3} = d\,\theta_1 ^2 + \sin^2 \theta_1 ( d\,\theta_2 + \sin^2 \theta_2 \, d\,\theta^2_3) \]
then
\begin{align*}
 \leftexp{5}{\square}_{\mbo{g}} v = & \frac{1}{\sqrt{|\mbo{g}|}} \big( \ptl_t(\sqrt{|\mbo{g}|}) \mbo{g}^{tt}v_t + \ptl_r(\sqrt{|\mbo{g}|}) \mbo{g}^{rr}v_r \big) \\
=& -\frac{1}{e^{\gamma + \Omega}} \big(-\ptl_t(e^{\gamma -\Omega}v_t) + \frac{1}{r^3}\ptl_r(r^3 e^{\Omega -\gamma} v_r)  \big) \\
=& -e^{-2\Omega} (v_{tt} + (\gamma_t-\Omega_t)v_t)+ e^{-2\gamma}(v_rr + (\Omega_r -\gamma_r)v_r + 3v_r r^{-1}) \\
=& \leftexp{3}\square v + 2e^{2\gamma} v_r r^{-1}.
\end{align*}

Alternatively, one can use the wave kernel representation formula to prove that the solution can be globally and smoothly 
extended for small energy. This method has been used by Christodoulou, Tahvildar-Zadeh and Shatah 
for the cases of spherical and equivariant symmetry\cite{chris_tah1,jal_tah}. 
This opens the door for a variety of techniques to be tested in the resolution of \textbf{(C2)}.

\subsection*{Branches of problems}
 The main research program explained above gives rise to many interesting branches of problems that are significant
 in their own right. Here is a selection of a few.
 \begin{description}
 \item[Open Problem 1]
 We spoke of equivariant wave maps $ U = (u(t,r), k\theta) $ for $k=1$. The results in this work can be extended for a general $k$. 
 In the situation where there is blow up, the concentration profile of the wave map inside the backward null cone of blow up
 point depends on $k.$ On flat background this dependence is quantified by Raphael and Rodnianski\cite{raph_rod}. It is an interesting
 problem to study the equivalent situation in the self-gravitating case. For the self-gravitating case it is expected that the 
 gravitational coupling constant $\mbo{\alpha}$ also plays a role.
 \item [Open Problem 2]
 When global existence holds, a natural question to ask is the asymptotic behavior of the wave map field.
 In \cite{chris_tah2}\cite{jal_tah} Christodoulou, Tahvildar-Zadeh and Shatah have established quantitative behavior of spherically symmetric
 and equivariant wave maps on flat background. The proofs are based on estimates on the wave kernel of the representation
 formula of solutions the wave maps equation. It is an interesting problem to study the asymptotic behavior of 
 critical self-gravitating wave maps. The wave kernel representation formula of Vincent Moncrief for wave equations 
 on curved background could be a fruitful starting point in the resolution of this question.
 \item [Open Problem 3]
 So far we focused on the critical case of $2+1$ dimensions for wave maps. Christodoulou, through a series of beautiful
 papers, mathematically studied the gravitational collapse of Einstein- free wave equation system with spherical symmetry
 \cite{chris_selfgrav},\cite{chris_gravcoll}.The work provided many new insights on the evolution of Einstein's equations
 and eventually supported the cosmic censorship conjectures of Roger Penrose. It is a worthwhile problem to study the 
 dynamics of 3+1 Einstein wave map system with spherical symmetry based on the techniques of Christodoulou. The effect of the additional nonlinearity of the
 wave maps equations in the system is to be understood.
 \end{description}
 
\bibliographystyle{plain}
\cleardoublepage
\phantomsection
\addcontentsline{toc}{chapter}{Bibliography}
\bibliography{Phd_main.bib}
\end{document}